\documentclass[reqno]{amsart}

\usepackage[utf8]{inputenc}
\usepackage{amsmath}
\numberwithin{equation}{section}
\numberwithin{figure}{section}
\numberwithin{table}{section}
\usepackage{graphicx}
\usepackage{amssymb}
\usepackage{mathrsfs}
\usepackage{amsfonts}
\usepackage[colorlinks=true, allcolors=blue]{hyperref}
\usepackage{subfigure} 
\usepackage{color}
\usepackage{extarrows}
\usepackage{booktabs}
\usepackage{hyperref}
\usepackage{tikz}
\usepackage{extarrows}
\usepackage{booktabs}
\usepackage{amsthm}
\usepackage{enumerate}
\usepackage{enumitem} 
\usepackage{indentfirst}
\usepackage{embrac}
\usepackage{cite} 
\usetikzlibrary{positioning,shapes}

\def\bZ{{\mathbb Z}}

\def\bR{{\mathbb R}}

\def\sE{{\mathscr E}}
\def\sF{{\mathscr F}}
\def\sA{{\mathscr A}}
\def\sG{{\mathscr G}}

\def\bs{\mathbf{s}}
\def\br{\mathbf{r}}
\def\bN{\mathbb{N}}
\def\sS{\mathscr{S}}

\def\fm{\mathfrak{m}}

\def\${|\!|\!|}
\def\l|{\left|\!\left|\!\left|}
\def\r|{\right|\!\right|\!\right|}

\newtheorem{theorem}{Theorem}[section]
\newtheorem{lemma}[theorem]{Lemma}
\newtheorem{proposition}[theorem]{Proposition}
\newtheorem{corollary}[theorem]{Corollary}
\theoremstyle{definition}
\newtheorem{definition}[theorem]{Definition}
\newtheorem{example}[theorem]{Example}

\theoremstyle{remark}
\newtheorem{remark}[theorem]{Remark}
\numberwithin{equation}{section}

\begin{document}

\title[On diffusions with discontinuous scales]{On diffusions with discontinuous scales}

\author{Liping Li}
\address{Fudan University, Shanghai, China.  }
\address{Bielefeld University,  Bielefeld, Germany.}
\email{liliping@fudan.edu.cn}
\thanks{The author is partially supported by NSFC (No.  11931004) and Alexander von Humboldt Foundation in Germany.  }


\subjclass[2010]{Primary 31C25, 60J35,  60J45.}



\keywords{}

\begin{abstract}
It is well known that a regular diffusion on an interval $I$ without killing inside is uniquely determined by a canonical scale function $\bs$ and a canonical speed measure $\fm$.  Note that $\bs$ is a strictly increasing and continuous function and $\fm$ is a fully supported Radon measure on $I$.  In this paper we will associate a general triple $(I,\bs,\fm)$,  where $\bs$ is only assumed to be increasing and $\fm$ is not necessarily fully supported,  to certain Markov processes by way of Dirichlet forms.  Using two transformations,  called scale completion and darning respectively,  to rebuild the topology of $I$,  we will successfully regularize the triple $(I,\bs,\fm)$ and obtain a regular Dirichlet form associated with it.  The corresponding Markov process is called the regularized Markov process associated with $(I,\bs,\fm)$.  In fact,  it is the unique Markov process up to homeomorphism that can be associated with $(I,\bs,\fm)$ in the context of regular representations of Dirichlet forms.  As a byproduct of regularized Markov process,  a continuous simple Markov process,  which does not satisfy the strong Markov property,  will be also raised to be associated to $(I,\bs,\fm)$ without operating regularizing program.  Furthermore,  we will show that the regularized Markov process is  identified with a skip-free Hunt process in one dimension as well as a quasidiffusion without killing inside.  Note that the skip-free Hunt process generalizes the concept of regular diffusion and admits a scale function and a speed measure in an analogous manner. 

\end{abstract}

\maketitle
\tableofcontents

\section{Introduction}

The title is a little misleading because the scale function of a nice diffusion in one dimension is always continuous.  Actually it focuses on the problem that how to associate an increasing but not necessarily continuous function on an interval to a certain Markov process.  But it is roughly correct because the final result shows that the desirable Markov process is very similar to a diffusion,  besides that the continuity of sample paths is replaced by so-called skip-free property.  This problem has been considered by Sch\"utze \cite{S79} by way of generalizing the second order differential operator raised by Feller,  and it turns out that the resulting process thereof is closely related to a widely studied Markov process,  called \emph{quasidiffusion} in,  e.g.,  \cite{BK87,  K86},  \emph{generalized diffusion} in,  e.g.,  \cite{W74,  KW82, LM20} and \emph{gap diffusion} in,  e.g.,  \cite{K81}.  In this paper we will adopt another treatment by virtue of the theory of Dirichlet forms.  A Dirichlet form is a closed symmetric form with Markovian property  on an $L^2$-space.  Due to a series of important works by Fukushima and Silverstein in the 1970s,  the ``\emph{regularity}" of a Dirichlet form assures that it is associated with  a symmetric Markov process.  We refer readers to \cite{FOT11, CF12} for notations and terminologies in the theory of Dirichlet forms. 

What is a diffusion in one dimension? As one of the most important stochastic models,  it means a continuous strong Markov process $X=(X_t)_{t\geq 0}$ on an interval $I=\langle l, r\rangle$ where $l$ or $r$ may or may not be contained in $I$; see,  e.g., \cite{I06,  IM74,  M68}.  Due to the study in the celebrated monograph  \cite{IM74} by It\^o and McKean,  every diffusion can be decomposed into ``regular" pieces: Like the concept of irreducibility for Markov chains,  this regularity means that every point in a piece can be visited by the diffusion starting from any other point in the same piece in finite time.  So one loses little generality and gains much simplification by considering a regular diffusion.  More precisely,  $X$ is called \emph{regular} if $\mathbf{P}_x(T_y<\infty)>0$ for any $x\in \mathring{I}:=(l,r)$ and $y\in I$,  where $T_y:=\inf\{t>0: X_t=y\}$.  For simplification we further assume that $X$ has no killing inside in the sense that $X_{\zeta-}\notin I$ if $\zeta<\infty$ where $\zeta$ is the lifetime of $X$.  Then a significant characterization tells us that $X$ is uniquely determined by a \emph{canonical scale function} $\bs$ and a \emph{canonical speed measure} $\fm$; see,  e.g.,  \cite[V\S7]{RW87} and \cite[VII\S3]{RY99}.  Note that $\bs$ is a continuous and strictly increasing function on $I$ giving the hitting distributions of $X$ in the sense that for any $a,x,b\in I$ with $a<x<b$,
\begin{equation}\label{eq:02}
	\mathbf{P}_x(T_b<T_a)=\frac{\bs(x)-\bs(a)}{\bs(b)-\bs(a)}. 
\end{equation}
When $\bs(x)=x$,  $X$ is called \emph{on its natural scale}. 
The canonical speed measure $\fm$ is a fully supported Radon measure on $I$,  which is roughly defined as $-\frac{1}{2}h''_{a,b}$ on every open interval $(a,b)\subset I$,  where $h_{a,b}(x):=\mathbf{E}_xT_a\wedge T_b$ is concave and $-h''_{a,b}$ is the Radon measure induced by the second derivative of $-h_{a,b}$ in the sense of distribution.  A more comprehensible explanation for canonical speed measure is as follows. When on its natural scale,  $X$ can be expressed as a time change of Brownian motion and $\fm$ measures the speed of its movements: In regions where $\fm$ is large,  $X$ moves slowly.  Particularly,  the Brownian motion is on its natural scale and its canonical speed measure is the Lebesgue measure.  

A famous analytic treatment to reach a regular diffusion brings into play the generalized second order differential operator 
\begin{equation}\label{eq:01}
\mathscr L:=\frac{1}{2}\frac{d^2}{d\fm d\bs},
\end{equation}
and a systematic introduction is referred to in \cite{M68}.  Roughly speaking,  $\mathscr L$  generates a Feller semigroup that leads to a regular diffusion with canonical scale function $\bs$ and canonical speed measure $\fm$. 
Sch\"utze \cite{S79} as well as contributions for quasidiffusions dealt with \eqref{eq:01} for certain general pair $(\bs,\fm)$.   The treatment of quasidiffusions considers the natural scale $\bs(x)=x$ while $\fm$ is induced by a (not strictly) increasing right continuous function $m$ on $\bR$.  It turns out that \eqref{eq:01} generates a standard process,  named \emph{quasidiffusion},  and $m$ is also called the \emph{speed measure} of quasidiffusion.  Unsurprisingly a quasidiffusion moves only on the support $E_m$ of $m$,  a closed subset of $\bR$.  The continuity of sample paths is not satisfied,  and instead the so-called \emph{skip-free} property holds: A quasidiffusion jumps across only the gaps of $E_m$.  Sch\"utze \cite{S79} studied the operator \eqref{eq:01} for another interesting case that $\bs$ is only strictly increasing and $\fm$ is a fully supported Radon measure such that $\fm$ has an isolated mass at points where $\bs$ is neither right nor left continuous.  It is insightful to point out in \cite{S79} that \eqref{eq:01} still corresponds to a ``simple" Markov process on $I$,  whereas the strong Markov property fails;  see \cite[Corollary~4.7]{S79}.  To associate $(\bs, \fm)$ to a strong Markov process,  the completion of state space $I$ with respect to $\bs$,  similar to the Ray-Knight compactification,  must be applied so that a desirable Feller semigroup would be obtained on the completed space. 


Theory of Dirichlet forms provides another technique to explore a regular diffusion $X$.  The crucial fact is that $X$ is symmetric with respect to its canonical speed measure $\fm$,  and thus it is possible to associate $X$ to a certain Dirichlet form on $L^2(I,\fm)$.  More precisely,  the Dirichlet form of $X$ on $L^2(I,\fm)$ is expressed as
\begin{equation}\label{eq:13-2}
\begin{aligned}
	&\sF^{(\bs,\fm)}=\{f\in L^2(I,\fm):f\ll \bs,  df/d\bs\in L^2(I,d\bs),  \\
	&\qquad\qquad \qquad f(j)=0\text{ if } j\notin I\text{ and }|\bs(j)|<\infty  \text{ for }j=l\text{ or }r\}, \\
	&\sE^{(\bs,\fm)}(f,g)=\frac{1}{2}\int_{I}\frac{df}{d\bs}\frac{dg}{d\bs},\quad f,g\in \sF^{(\bs,\fm)},
\end{aligned}
\end{equation}
where $f\ll \bs$ stands for that $f$ is absolutely continuous with respect to $\bs$.  
This was first established in \cite{F10,  FHY10},  and a systematic summarization is also referred to in \cite[\S2.2.3]{CF12}.  With the help of the theory of Dirichlet forms,  one can go farther in related studies.  For example,  the connection between Dirichlet form and the operator \eqref{eq:02} was studied in \cite{F14}.  The Dirichlet form characterization for diffusions without regular property was accomplished in \cite{LY19,  L21}.  Rich ``singular" diffusions were found out in a series of articles concerning so-called regular Dirichlet subspaces; see,  e.g,  \cite{LY17,  LY19-2,  LSY20}.

The object of the current paper is to associate a triple $(I,\bs,\fm)$ to a certain Markov process by means of Dirichlet forms,  where $I$ is an interval,  $\bs$ is an increasing function on $I$ and $\fm$ is a positive measure on $I$ which is Radon on its interior.  Note that $\bs$ is not assumed to be strictly increasing nor left/right continuous,  and $\fm$ is not necessarily fully supported.  This triple is much general than that considered in \cite{S79}. To extend \eqref{eq:13-2} for this general triple,  the main thrust is to explain the absolute continuity with respect to a discontinuous scale function.  This step will be completed in \S\ref{SEC2}.  Roughly speaking,  by eliminating left or right discontinuities  from $\bs$,  one can obtain a continuous and increasing function.  This continuous part of $\bs$ and the gaps of $\bs$ at left/right discontinuous points induce three positive Radon measures,  each of which corresponds to a concept of absolute continuity.  Eventually the absolute continuity with respect to $\bs$ is by definition the sum of these three ones.  Given this setting,  we define a quadratic form $(\sE,\sF)$ on $L^2(I,\fm)$ as \eqref{eq:25},  called the \emph{Dirichlet form associated with} $(I,\bs,\fm)$,  which is certainly a generalization of that reaches a regular diffusion: If $\bs$ is continuous and strictly increasing and $\fm$ is fully supported,  then \eqref{eq:25} is identified with \eqref{eq:13-2}.  However,  
 $(\sE,\sF)$ is only a \emph{Dirichlet form in the wide sense},  because $\fm$ is not fully supported and every function $f\in \sF$ is constant on intervals where $\bs$ is constant; see Theorem~\ref{LM12}.  As a result we cannot associate $(\sE,\sF)$ to a certain Markov process straightforwardly.  
 
Instead our goal will be achieved in two different manners in the sequel.  The first one is in the context of \emph{regular representations} of $(\sE,\sF)$.  This concept was raised in a famous paper \cite{F71} of Fukushima.  Roughly speaking,  a regular representation of $(\sE,\sF)$ is a regular Dirichlet form $(\sE',\sF')$ on another $L^2$-space $L^2(I',\fm')$ such that there is an algebra isomorphism $\Phi$ from $\sF_b$ to $\sF'_b$,  where $\sF_b:=\sF\cap L^\infty(I,\fm)$ and $\sF'_b:=\sF'\cap L^\infty(I',\fm')$,  and $\Phi$ preserves three kinds of metrics as listed in \eqref{eq:31}.  By making use of Gelfand representations of subalgebras of $L^\infty$-space,  Fukushima \cite{F71} proved that a Dirichlet form in the wide sense always has regular representations,  and as is shown in \cite[Theorem~A.4.2]{FOT11},  all of them are quasi-homemorphic in the sense of,  e.g.,  \cite[Definition~1.4.1]{CF12}.  

Concerning the Dirichlet form $(\sE,\sF)$ associated with $(I,\bs,\fm)$,  we will single out two special but significant regular representations in Theorem~\ref{THM6} by raising a program to regularize $(I,\bs,\fm)$.  
More precisely,  two transformations are used to rebuild the topology of $I$ so that we can transform $(I,\bs,\fm)$ accordingly:  One transformation,  called \emph{scale completion} due to Sch\"utze \cite{S79},  makes the completion of $I$ with respect to $\bs$ and the other transformation,  called \emph{darning},  collapses each interval where $\bs$ is constant into an abstract ``point".  
The space obtained by ``regularizing" $I$ is denoted by $I^*$,  and $\bs, \fm$ are also transformed into new scale function $\bs^*$ and new speed measure $\fm^*$ on $I^*$ by a natural way.  For convenience we call $(I^*,\bs^*,\fm^*)$ the \emph{regularization} of $(I,\bs,\fm)$.  Finally,   transforming $(\sE,\sF)$ into another quadratic form $(\sE^*,\sF^*)$ on $L^2(I^*,\fm^*)$ in accord with the regularizing of $(I,\bs,\fm)$,  we will conclude in Theorem~\ref{THM6} that $(\sE^*,\sF^*)$ is actually a regular representation of $(\sE,\sF)$.  On the other hand,  as will be explained in Example~\ref{EXA31},  the regularized space $I^*$ might be very abstract and hence hard to handle.  Fortunately we find that $\bs^*$ is a homeomorphism mapping $I^*$ onto a \emph{nearly closed subset} $\widehat{I}$ of $\bR$,  i.e.  $\widehat{I}\cup \{\widehat{l},\widehat{r}\}$ is closed in the extended real number system $\overline{\bR}$ and $\widehat{I}\subset \bR$,  where $\widehat{l}=\inf\{\widehat{x}: \widehat{x}\in \widehat{I}\}$ and $\widehat{r}=\sup\{\widehat{x}: \widehat{x}\in \widehat{I}\}$.  
 Applying this homeomorphism,  one gets the image Dirichlet form $(\widehat{\sE},\widehat{\sF})$ of $(\sE^*,\sF^*)$ on $L^2(\widehat{I},\widehat{\fm})$, where $\widehat{\fm}$ is the image measure of $\fm^*$ under $\bs^*$.  This is the second special regular representation of $(\sE,\sF)$ that is of great interest to us. 

Due to the celebrated construction theorem by Fukushima in \cite{F71-2},  there is an $\fm^*$-symmetric (resp.  $\widehat{\fm}$-symmetric) Markov process $X^*$ (resp.  $\widehat{X}$) on $I^*$ (resp.  $\widehat{I}$) associated with $(\sE^*,\sF^*)$ (resp.  $(\widehat{\sE},\widehat{\sF})$).  
They are definitely significant,  because $X^*$ plays an analogous role to regular diffusion in this Dirichlet form characterization,  and $\widehat{X}$ is the image process of $X^*$ on its natural scale as will be shown in \S\ref{SEC33}.  In fact,  $(\widehat{\sE},\widehat{\sF})$ is a so-called \emph{trace Dirichlet form},  as leads to identifying $\widehat{X}$ with a \emph{time-changed Brownian motion} with speed measure $\widehat{\fm}$.  This fact also sheds light on the significance of $\widehat{\fm}$: Like a regular diffusion,  $\widehat{X}$ moves slowly in regions where $\widehat{\fm}$ is large.  Such processes deserve names and we call $X^*$ (resp.  $\widehat{X}$) the \emph{regularized Markov process} (resp.  \emph{image regularized Markov process}) associated with $(I,\bs,\fm)$.  


It is worth emphasising that $X^*$ and $\widehat{X}$ are linked via the homeomorphism $\bs^*$,  i.e.  $(\widehat{X})_{t\geq 0}$ is equivalent to $\bs^*(X^*_t)_{t\geq 0}$.  This connection seems to hold true only for special regular representations,  because there usually exists an essential difference between quasi-homeomorphism and homeomorphism.  However we will obtain a remarkable result in Theorem~\ref{THM311} showing that the Markov process associated with an arbitrary regular representation of $(\sE,\sF)$ must be a homeomorphic image of $\widehat{X}$.  Hence  by way of regular representations,  we have associated $(I,\bs,\fm)$ to a unique Markov process up to homeomorphisms.  

In a second manner we will give an $\fm$-symmetric continuous simple Markov process on $I$ but not satisfying the strong Markov property,  whose Dirichlet form on $L^2(I,\fm)$ is $(\sE,\sF)$; see Theorem~\ref{THM82}.  This process,  called the \emph{unregularized Markov process} associated with $(I,\bs,\fm)$,  is a byproduct of the regularized one $X^*$.  Heuristically speaking,  the transformation of scale completion in the regularizing program divides each discontinuous point of $\bs$ into distinct ones so that $\bs$ is continuous on the completed space.  The unregularized Markov process is obtained by applying the inverse transformation to $X^*$,  that is,  merging divided distinct points and hence eliminating jumps of $X^*$.  As a result,  we get a symmetric simple Markov process with continuous paths,  while the strong Markov property fails for certain stopping times like the first hitting times for sets containing discontinuous points of $\bs$.  


The rest of this paper is devoted to a probabilistic description of the Markov processes associated with $(I,\bs,\fm)$ by way of regular representations.  We begin with introducing a so-called \emph{skip-free Hunt process in one dimension}.  
It is,  by definition,  a Hunt process on a nearly closed subset $E$ of $\overline{\bR}$ satisfying the skip-free property (SF),  the regular property (SR) (like the regular property for a diffusion) and having no killing inside (SK); see Definition~\ref{DEF51}.  This concept obviously generalizes regular diffusion: When $E$ is an interval,  (SF) is identified with the continuity of sample paths and a skip-free Hunt process is actually a regular diffusion.  In \S\ref{SEC5} we will show that a skip-free Hunt process $X$ also gives a \emph{scale function} in the same sense as \eqref{eq:02},  and when on its natural scale,  $X$ admits a \emph{speed measure} roughly defined as $-\frac{1}{2}h''_{a,b}$ on $E\cap (a,b)$ for any $a,b\in E$ with $a<b$,  where $E$ must be a subset of $\bR$ and $h_{a,b}(x):=\mathbf{E}_xT_a\wedge T_b, x\in E\cap (a,b)$ can be extended to a concave function on $(a,b)$.  
In addition,  one can find in Corollary~\ref{COR64} that,  like a regular diffusion,  a skip-free Hunt process is uniquely determined by its scale function and speed measure.
After these ingredients are put forward,  the main result, Theorem~\ref{THM71}, states that $\widehat{X}$ is equivalent to either of the following with identifying speed measures:
\begin{itemize}
\item[(1)] A skip-free Hunt process on its natural scale;  
\item[(2)] A quasidiffusion with no killing inside.  
\end{itemize}
As said above,  the Markov processes associated with regular representations of $(\sE,\sF)$ are homeomorphic images of $\widehat{X}$.  Hence all of them are actually Hunt processes with the properties (SF),  (SR) and (SK) on a space homeomorphic to a nearly closed subset of $\bR$.

Throughout this paper the condition (DK) for $(I,\bs,\fm)$,  the condition (SK) for a skip-free Hunt process and the condition (QK) for a quasidiffusion are assumed to bar the possibility of killing inside.  We wish to state emphatically that without assuming (DK) or (QK),  the formulation of (image) regularized Markov process or quasidiffusion still holds true,  while there probably appears killing at endpoints contained in the state space; see Remarks~\ref{RM38} and \ref{RM73}.  The case of skip-free Hunt process is more complicated.  Without (SK),  a skip-free Hunt process may admit killing everywhere,  so a \emph{killing measure} should be introduced to characterize the killing part of the process.  We hope to make it in a further contribution. 

The paper is organized as follows.  In \S\ref{SEC2}, we will introduce the Dirichlet form $(\sE,\sF)$ associated with $(I,\bs,\fm)$, and the main result, Theorem~\ref{LM12}, shows that it is a Dirichlet form in the wide sense.  The section \S\ref{SEC3} is devoted to studying the regular representations of $(\sE,\sF)$ and their associated Markov processes.  In \S\ref{SEC5} we turn to investigate the skip-free Hunt process in one dimension.  Its scale function will be obtained in Theorem~\ref{THM58} and its speed measure will be explicitly defined in \S\ref{SEC53}.  A review of quasidiffusion will be outlined in \S\ref{SEC6}.  Note that a quasidiffusion is a semimartingale.  In particular we will distinguish Markov local times and  semimartingale local times for a quasidiffusion in Lemma~\ref{LM65},  so that the It\^o-Tanaka-Meyer formula is set up.  The correspondence between image regularized Markov process,  skip-free Hunt process and quasidiffusion will be established in \S\ref{SEC7}.  
In \S\ref{SEC9} the unregularized Markov process associated with $(I,\bs,\fm)$ will be given.  The last section \S\ref{SEC8} gives several examples of Markov processes that can be obtained by regularizing a certain triple $(I,\bs,\fm)$.  


\subsection*{Notations}
Let $\overline{\mathbb{R}}=[-\infty, \infty]$ be the extended real number system.  A set $E\subset \overline{\bR}$ is called a \emph{nearly closed subset} of $\overline{\bR}$ if $\overline E:= E\cup \{l,r\}$ is a closed subset of $\overline{\bR}$ where $l=\inf\{x: x\in E\}$ and $r=\sup\{x: x\in E\}$.  The point $l$ or $r$ is called the left or right endpoint of $E$.  Denote by $\overline{\mathscr K}$ the family of all nearly closed subsets of $\overline{\bR}$.  Set
\[
	\mathscr K:=\{E\in \overline{\mathscr K}: E\subset \bR\},
\]
and every $E\in \mathscr K$ is called a \emph{nearly closed subset} of $\bR$.  

Let $E$ be a locally compact separable metric space.  We denote by $C(E)$ the space of all real continuous functions on $E$.  In addition,  $C_c(E)$ is the subspace of $C(E)$ consisting of all continuous functions on $E$ with compact support and
\[
	C_\infty(E):=\{f\in C(E): \forall \varepsilon>0, \exists K\text{ compact},  |f(x)|<\varepsilon, \forall x\in E\setminus K\}.  
\]
The functions in $C_\infty(E)$ are said to be vanishing at infinity.  
Given an interval $J$,  $C_c^\infty(J)$ is the family of all smooth functions with compact support on $J$.  

The abbreviations CAF and PCAF stand for \emph{continuous additive functional} and \emph{positive continuous additive functional} respectively.  

\section{Dirichlet forms related to discontinuous scales}


\subsection{Discontinuous scale on an interval}\label{SEC2}

Let $\overline{\mathbb{R}}=[-\infty, \infty]$ be the extended real number system and the following triple $(I,\bs,\fm)$ is what we are concerned with:
\begin{itemize}
\item[(a)] A closed interval $I=[l,r]\subset \overline{\bR}$, where $-\infty \leq l<0<r\leq \infty$.  
\item[(b)] A non-constant increasing real valued function $\bs$ on $(l,r)$ such that $\bs(0-)=\bs(0)=\bs(0+)=0$.  Set $\bs(l):=\lim_{x\downarrow l}\bs(x)$ and $\bs(r):=\lim_{x\uparrow r}\bs(x)$.  
\item[(c)]  A positive measure $\fm$ on $I$, which is Radon on $(l,r)$,  i.e.  $\fm([a,b])<\infty$ for any $[a,b]\subset (l,r)$,  and such that $ \fm(\{0\})=0$.  
\end{itemize}
The conditions  $\bs(0-)=\bs(0)=\bs(0+)=0$ and $ \fm(\{0\})=0$ lose no generality,  and for convenience we further make the convention $\bs(l-):=\bs(l)$ and $\bs(r+):=\bs(r)$.  We emphasis that $\bs(l),\bs(r), \fm(\{l\})$ or $\fm(\{r\})$  may be infinite. 

The function $\bs$ is called a \emph{scale function} (\emph{scale} for short) on $I$. 
It is not assumed to be continuous nor strictly increasing. Set
\begin{equation}\label{eq:13}
	U:=\{x\in I: \bs\text{ is constant on }(x-\varepsilon,x+\varepsilon)\cap I\text{ for some }\varepsilon>0\}.  
\end{equation}
Clearly $U \setminus \{l,r\}$ is open and can be written as a disjoint union of open intervals 
\begin{equation}\label{eq:14}
U\setminus \{l,r\}=\bigcup_{n\geq 1}(c_n,d_n). 
\end{equation}
Let $D^\pm:=\{x\in I: \bs(x)\neq \bs(x\pm)\}$, $D^0:=D^+\cap D^-$ and $D:=D^+\cup D^-$.  Every point $x\in D^0$ is called \emph{isolated} (with respect to $\bs$),  and the interval $(c_n,d_n)$ is called \emph{isolated} (with respect to $\bs$) if $c_n,d_n\in D\cup \{l,r\}$. 

Let us give a definition classifying $l$ and $r$ in terms of the pair $(\bs,\fm)$.  
For convenience we use $\fm(r-)<\infty$ (resp.  $\fm(l+)<\infty$) to stand for that $\fm((r-\varepsilon,r))<\infty$ (resp.  $\fm((l,l+\varepsilon))<\infty$) for some $\varepsilon>0$.  Otherwise we make the convention $\fm(r-)=\infty$ (resp. $\fm(l+)=\infty$). 

\begin{definition}\label{DEF21}
\begin{itemize}
\item[(1)] The endpoint $r$ (resp.  $l$) is called \emph{approachable} if $\bs(r)<\infty$ (resp.  $\bs(l)>-\infty$).  
\item[(2)] An approachable endpoint $r$ (resp.  $l$) is called \emph{regular} if $\fm(r-)<\infty$ (resp.  $\fm(l+)<\infty$).  
\item[(3)] A regular endpoint $r$ (resp.  $l$) is called \emph{reflecting},  if $\fm(\{r\})<\infty$ (resp.  $\fm(\{l\})<\infty$).  Otherwise it is called \emph{absorbing}. 
\end{itemize}
\end{definition}

Now we have a position to present two additional conditions for the triple $(I,\bs,\fm)$ that will be assumed throughout this paper.  The first one will bar the possibility of killing inside for desirable Markov processes (see Remark~\ref{RM38}): 
\begin{itemize}
\item[(DK)] $l$ or $r$ is reflecting if it is the endpoint of an isolated interval in \eqref{eq:14}.  
\end{itemize}
Another one is related to $\fm$.
Let $E_\fm:=\text{supp}[\fm]$ be the support of $\fm$,  i.e.  the smallest closed subset of $I$ outside which $\fm$ vanishes.  Set $E_\bs:=I\setminus U$,  called the \emph{support} of $\bs$.  Actually if $\bs$ is right continuous,  then $E_\bs$ coincides with the support of the Lebesgue-Stieltjes measure induced by $\bs$.  
In place of $E_\fm=I$,  we assume
\begin{itemize}
\item[(DM)]  $E_\bs\subset E_\fm$,  $\fm(\{x\})>0$ for $x\in D^0$ and for isolated interval $(c_n,d_n)$ in \eqref{eq:14},  it holds that 
\[
	\fm\left(I_{c_n}^{d_n} \right)>0,
\]
where $I_{c_n}^{d_n}$ is an interval ended by $c_n$ and $d_n$ and $c_n\in I_{c_n}^{d_n}$ (resp.  $d_n\in I_{c_n}^{d_n}$) whenever $c_n\notin D^+$ (resp.  $d_n\notin D^-$).  
\end{itemize}
\begin{remark}\label{RM22}
This condition is only used to obtain a fully supported ``regularized" speed measure in Theorem~\ref{THM6}.  
Note that $E_\bs\subset E_\fm$ implies that $(\alpha, \beta)\subset U$ for any $\fm$-negligible open interval $(\alpha,\beta)\subset I$.  When $\bs$ is strictly increasing,  (DM) reads as $E_\fm=I$ and $\fm(\{x\})>0$ for $x\in D^0$,  as is identified with the assumption in \cite{S79}. 
\end{remark}

\subsection{Absolute continuity with respect to $\bs$}

Before moving on we prepare some elements concerning the absolute continuity with respect to $\bs$.  
 Since neither left nor right continuous,   $\bs$ cannot induce a Lebesgue-Stieltjes measure.  Instead we set 
\[
	\mu^+_d:=\sum_{x\in I}\left(\bs(x+)-\bs(x)\right)\cdot \delta_x,\quad \mu^-_d:=\sum_{x\in I}\left(\bs(x)-\bs(x-)\right)\cdot \delta_x
\]
and
\[
	\bs^+_d(x):=\left\lbrace
	\begin{aligned}
	&\mu^+_d((0,x)),\;\;\;\;\; x\geq 0,\\
	&-\mu^+_d([x,0)),\; x<0,
	\end{aligned}
	\right.  \quad
	\bs^-_d(x):=\left\lbrace
	\begin{aligned}
	&\mu^-_d((0,x]),\;\;\;\;\;\, x\geq 0,\\
	&-\mu^-_d((x,0)),\; x<0,
	\end{aligned}
	\right.  
\]
where $\delta_x$ is the Dirac measure at $x$.  
Note that $\mu^\pm_d$ are positive Radon measures on $(l,r)$ with $\mu^\pm_d(\{l,r\})=0$, and $\bs^+_d$ (resp.  $\bs^-_d$) is left (resp.  right) continuous.  Let 
\[
\bs_c:=\bs-\bs^+_d-\bs^-_d.  
\]
Clearly $\bs_c$ is increasing and continuous on $I$ and denote  its Lebesgue-Stieltjes measure by $\mu_c$.  
Hereafter we make the conventions $\int_{(0,x)}:=-\int_{[x,0)}$ and $\int_{(0,x]}:=-\int_{(x,0)}$ for $x<0$.
Define
\[
	\mathscr S^+_d:=\left\{f(\cdot)=\int_{(0,\cdot)}g(y)\mu^+_d(dy):g\in L^2(I,\mu^+_d) \right\},
\]
\[
	\mathscr S^-_d:=\left\{f(\cdot)=\int_{(0,\cdot]}g(y)\mu^-_d(dy): g\in L^2(I,\mu^-_d)\right\},
\]
and 
\[
	\mathscr S_c:=\left\{f(\cdot)=c_0+\int_{(0,\cdot)}g(y)\mu_c(dy): g\in L^2(I,\mu_c),  c_0\in \bR\right\}.  
\]
In either case write $g:=df/d\mu$ for $\mu=\mu^\pm_d$ or $\mu_c$.  We emphasis that the functions in these families are defined pointwisely on $(l,r)$.   
Let
\begin{equation}\label{eq:11-2}
\sS:=\sS_c+\sS^+_d+\sS^-_d=\{f=f^c+f^++f^-: f^c\in \sS_c,  f^\pm\in \sS^\pm_d\}
\end{equation}
and clearly,  the decomposition of $f\in \sS$ in $\sS_c+\sS^+_d+\sS^-_d$ is unique.  
For such $f\in \sS$,  define
\begin{equation}\label{eq:11}
	\int_I \left(\frac{df}{d\bs}\right)^2d\bs:=\int_I \left(\frac{df^c}{d\mu_c}\right)^2d\mu_c+\int_I \left(\frac{df^+}{d\mu^+_d}\right)^2d\mu^+_d+\int_I \left(\frac{df^-}{d\mu^-_d}\right)^2d\mu^-_d.
\end{equation}
The lemma below (cf.  \cite[Proposition~2.2]{S79}) is useful and the proof is straightforward,  so we omit it. 

\begin{lemma}\label{LM31}
$f\in \sS$ is $\bs$-continuous in the sense that $f$ has finite limits from the left and the right on $(l,r)$ and that the right or left continuity of $\bs$ at a point implies the same property of $f$.  Particularly,  when $r$ (resp. $l$) is approachable,  $f\in \sS$ admits a finite limit $f(r):=\lim_{x\uparrow r}f(x)$ (resp.  $f(l):=\lim_{x\downarrow l}f(x)$).
\end{lemma}



\subsection{Dirichlet form associated with $(I,\bs, \fm)$}\label{SEC14}

Let $I_0:=\langle l, r\rangle$ be the interval ended by $l$ and $r$ obtained by taking out non-reflecting endpoints from $I$.  Clearly $\fm$ is Radon on $I_0$ and we denote the restriction of $\fm$ to $I_0$ still by $\fm$.  
 Define a quadratic form on $L^2(I_0,\fm)$:
\begin{equation}\label{eq:25}
\begin{aligned}
	&\sF:= \{f\in \sS\cap L^2(I_0,\fm): f(j)=0\text{ if }j\notin I_0\text{ is approachable for }j=l\text{ or }r\},  \\
	&\sE(f,g):=\frac{1}{2}\int_I \frac{df}{d\bs}\frac{dg}{d\bs}d\bs,\quad f,g\in \sF,
\end{aligned}
\end{equation}
where $\sS$ is defined as \eqref{eq:11-2} and $\sE(f,g)$ is defined by \eqref{eq:11}.   Note incidentally that \eqref{eq:25} coincides with the Dirichlet form associated with a regular diffusion if $\bs$ is strictly increasing and  continuous; see \S\ref{SEC71}. 
However for general scale function, this quadratic form is not even a Dirichlet form because $f\in \sF$ is constant on $(c_n,d_n)$ for any $n\geq 1$,  where $(c_n,d_n)$ appears in \eqref{eq:14},  and hence $\sF$ is not necessarily dense in $L^2(I_0,\fm)$.
Instead we say $(\sE,\sF)$ is a Dirichlet form on $L^2(I_0,\fm)$ \emph{in the wide sense} provided that $(\sE,\sF)$ is a non-negative symmetric closed form satisfying the Markovian property.  For simplicity we call \eqref{eq:25} the \emph{Dirichlet form associated with} $(I,\bs,\fm)$.  

\begin{theorem}\label{LM12}
$(\sE,\sF)$ is a Dirichlet form on $L^2(I_0,\fm)$ in the wide sense.  
\end{theorem}
\begin{proof}
Firstly we note that $(\sE,\sF)$ is clearly a non-negative symmetric quadratic form on $L^2(I_0,\fm)$.  To prove its closeness,  take an $\sE_1$-Cauchy sequence $$f_n=f^c_n+f^+_n+f^-_n\in \sF,\quad n\geq 1.$$
We do not lose a great deal by assuming that $f_n$ converges to $f$ both $\fm$-a.e.  and in $L^2(I_0,\fm)$,
\begin{equation}\label{eq:12}
	g^c_n:=\frac{df^c_n}{d\mu_c}\rightarrow g^c\quad \text{ in }L^2(\mu_c)\text{ and }\mu_c\text{-a.e.}, 
\end{equation}
and
\begin{equation}\label{eq:27}
	g^\pm_n:=\frac{df^\pm_n}{d\mu^\pm_d}\rightarrow g^\pm\quad \text{ in }L^2(\mu^\pm_d)\text{ and }\mu^\pm_d\text{-a.e.}
\end{equation}
Assume further that $f(0)=\lim_{n\rightarrow\infty} f_n(0)=\lim_{n\rightarrow \infty} f^c_n(0)$. 
Set for $x\in I_0$,  
\[
\begin{aligned}
	&f^c(x):=f(0)+\int_{(0,x)}g^c(y)\mu_c(dy)\in \sS_c, \\
	&f^+(x):=\int_{(0,x)}g^+(y)\mu^+_d(dy)\in \sS^+_d, \quad f^-(x):=\int_{(0,x]}g^-(y)\mu^-_d(dy)\in \sS^-_d.
\end{aligned}\] 
Since $\lim_{n\rightarrow \infty} f^c_n(0)=f(0)$,  it follows from \eqref{eq:12} and \eqref{eq:27} that $\lim_{n\rightarrow \infty} f^{c,\pm}_n(x)=f^{c,\pm}(x)$ for $x\in I_0$.  
Define for $x\in I_0$, 
\[
	\hat{f}(x):=f^c(x)+f^+(x)+f^-(x).
\]
Clearly $\hat{f}\in \sS$,  $\sE(f_n-\hat{f},f_n-\hat{f})\rightarrow 0$  and   
\begin{equation}\label{eq:28}
	f_n(x)\rightarrow \hat{f}(x),\quad x\in I_0.
\end{equation}
When $j\notin I_0$ is approachable for $j=l$ or $r$,  \eqref{eq:28} still holds for $x=j$.  Meanwhile $\hat{f}(j)=\lim_{n\rightarrow \infty}f_n(j)=0$.  
To conclude the closeness,  we only need to note that $f=\hat{f}$,  $\fm$-a.e. on $I_0$,  by means of $f_n\rightarrow f$,  $\fm$-a.e.,  and \eqref{eq:28}. 

Finally it suffices to show the Markovian property of $(\sE,\sF)$.  We adopt the notations in \S\ref{SEC32} and Lemma~\ref{LM39} will be employed in this step.  Take $f\in \sF$ and let $g$ be a normal contraction of $f$,  i.e.  
\begin{equation}\label{eq:29}
	|g(x)-g(y)|\leq |f(x)-f(y)|,\; x,y\in I_0,\quad |g(x)|\leq |f(x)|,\;x\in I_0.  
\end{equation}
It is easy to verify that $g$ is $\bs$-continuous,  so that we can define a function $\widehat{g}$ on $\widehat{I}$ as \eqref{eq:35} with $g$ in place of $f$.   Further let $\widehat{f}$ be defined as \eqref{eq:35} for this $f$.  It follows from \eqref{eq:29} that 
\begin{equation}\label{eq:210}
|\widehat g(\widehat x)-\widehat g(\widehat y)|\leq |\widehat f(\widehat x)-\widehat f(\widehat y)|,\; \widehat x,\widehat y\in\widehat  I,\quad |\widehat g(\widehat x)|\leq |\widehat f(\widehat x)|,\;\widehat x\in \widehat I.  
\end{equation}
Using \eqref{eq:210},   Lemma~\ref{LM39},  \eqref{eq:320-2} and applying \cite[Theorem~6.2.1~(2)]{FOT11},  we can obtain that $\widehat{g}\in \widehat{\sS}$.  Note that $\widehat\iota$ is a bijective between $\sS$ and $\widehat{\sS}$.  Thus $g=\widehat{\iota}^{-1}\widehat{g}\in \sS$.  By means of \eqref{eq:29},  one can also get $g\in L^2(I_0,\fm)$ and if $j\notin I_0$ is approachable for $j=l$ or $r$,  then
\[
	|g(j)|=\lim_{x\rightarrow j}|g(x)|\leq \lim_{x\rightarrow j}|f(x)|=|f(j)|=0.  
\]
Consequently $g\in \sF$ is concluded.  In addition,  \eqref{eq:210} also yields that $|\widehat{g}'|\leq |\widehat{f}'|$,  a.e.  on $\widehat{I}$.  Hence $\sE(g,g)\leq \sE(f,f)$ follows from \eqref{eq:210},  \eqref{eq:38} and \eqref{eq:39-2}.  That completes the proof. 
\end{proof}

In the next section we will regularize $(I,\bs,\fm)$ and $(\sE,\sF)$ by a certain way,  so that a Hunt process is obtained.  This process,  called (image) regularized Markov process associated with $(I,\bs,\fm)$, will be further explored in the subsequent sections.  
Then in \S\ref{SEC9} we will come back to $(\sE,\sF)$ for the special case that $\bs$ is strictly increasing.  In this case $(\sE,\sF)$ is a (not regular) Dirichlet form on $L^2(I_0,\fm)$ and there is an $\fm$-symmetric continuous simple (not strong) Markov process on $I_0$ whose Dirichlet form is $(\sE,\sF)$.

\section{Regular representations and (image) regularized Markov processes}\label{SEC3}

Following \cite{F71},  we call $(E_1,\fm_1,\sE^1,\sF^1)$ a \emph{D-space} provided that $(\sE^1,\sF^1)$ is a Dirichlet form on $L^2(E_1,\fm_1)$ in the wide sense.  The space $\sF^1_b:=\sF^1\cap L^\infty(E_1,\fm_1)$ is an algebra over $\bR$.  Let $(E_2,\fm_2,\sE^2,\sF^2)$ be another D-space.  Then $(E_1,\fm_1,\sE^1,\sF^1)$ and $(E_2,\fm_2,\sE^2,\sF^2)$ are called \emph{equivalent} if there is an algebra isomorphism $\Phi$ from $\sF^1_b$ to $\sF^2_b$ and $\Phi$ preserves three kinds of metrics: For $f\in \sF^1_b$, 
\begin{equation}\label{eq:31}
	\|f\|_\infty=\|\Phi(f)\|_\infty, \quad (f,f)_{\fm_1}=(\Phi(f),\Phi(f))_{\fm_2},\quad \sE^1(f,f)=\sE^2(\Phi(f),\Phi(f)),
\end{equation}
where $\|\cdot\|_\infty:=\|\cdot\|_{L^\infty(E_i,\fm_i)}$ and $(\cdot,\cdot)_{\fm_i}=(\cdot,\cdot)_{L^2(E_i,\fm_i)}$ for $i=1,2$.  In addition,  $(E_2,\fm_2,\sE^2,\sF^2)$ is called a \emph{regular representation} of $(E_1,\fm_1,\sE^1,\sF^1)$ if they are equivalent and $(\sE^2,\sF^2)$ is regular on $L^2(E_2,\fm_2)$.  

As obtained in Theorem~\ref{LM12},  $(I_0,\fm,\sE,\sF)$ is a D-space.  This section is 
 devoted to finding two special regular representations,  denoted by $(I^*,\fm^*,\sE^*,\sF^*)$ and $(\widehat{I},\widehat{\fm},\widehat{\sE},\widehat{\sF})$ respectively,  of it.  
Furthermore,  we will show that the quasi-homeomorphism between two regular representations of $(I_0,\fm,\sE,\sF)$ is essentially a strict homoemorphism. 

\subsection{Regularizing program}\label{SEC31}

The crucial step to obtain a regular representation of $(I_0,\fm,\sE,\sF)$ is to utilize two transformations on the state space: 
The first one,  called \emph{scale completion},  makes the completion $\bar{I}^d$ of $I$ with respect to the metric
\[
	d(x,y):=|\tan^{-1} x-\tan^{-1}y|+|\tan^{-1}\bs(x)-\tan^{-1}\bs(y)|.  
\] 
This transformation divides $x\in D^0$ into $\{x-,x,x+\}$ and $x\in D \setminus D^0$ into $\{x-,x+\}$.  Note that for a sequence $I\setminus D\ni  x_n\uparrow x\in D$ (resp.  $I\setminus D\ni x_n\downarrow x\in D$) in $I$,  we have $d(x_n,  x-)\rightarrow 0$ (resp.  $d(x_n, x+)\rightarrow 0$).  
The second is \emph{darning}.  Let $\overline{(c_n,d_n)}^d$ be the closure of $(c_n,d_n)$ in $\bar{I}^d$,  where $(c_n,d_n)$ appears in \eqref{eq:13}.  This transformation collapses each $\overline{(c_n,d_n)}^d$ into an abstract point $p^*_n$ and the neighbourhoods of $p^*_n$ in the new topological structure are determined by those of $\overline{(c_n,d_n)}^d$ in $\bar{I}^d$.  We refer readers to page 347 of \cite{CF12} for more details about this operation.  Denote by $\bar{I}^{d,*}$ the space obtained by darning $\bar{I}^d$.  


\begin{example}\label{EXA31}
This example is to explain the above two transformations.  Consider
\[
	I=[0,3], \quad \bs(x)=\left\lbrace\begin{aligned}
	&x,\quad 0\leq x< 1,\\
	&1,\quad 1\leq x< 2,\\
	&x,\quad 2\leq x\leq 3.  
	\end{aligned}  \right.  
\]
Then $U=(a_1,b_1)=(1,2)$ and $\bs$ has only one discontinuous point $2$.  The transformation of scale completion divides $2$ into $\{2-, 2+\}$ and $$\bar{I}^d=[0,2-]\cup [2+, 3],$$  where $[0,2-]$ (resp.  $[2+,3]$) is homoemorphic to the usual interval $[0,2]$ (resp.  $[2,3])$ but $2-$ and $2+$ are distinct points in $\bar{I}^d$.  The closure $\overline{(a_1,b_1)}^d$ of $(a_1,b_1)$ is $[1,2-]$ and the darning transformation collapses it into the abstract point $p^*_1$,  which can be viewed as the point $1$.  In other words,  $\bar{I}^{d,*}$ may be treated as $[0,1]\cup [2+,3]$.  

We should point out that the abstract point $p^*_n$ might not always be understood as a usual point.  For example,  consider $I=[0,1]$ and that $\bs$ is the standard Cantor function on $[0,1]$,  i.e.  $\bs(x)=\int_0^x 1_{K^c}(y)dy$ where $K\subset [0,1]$ is the standard Cantor set.  Then $\bar{I}^d=[0,1]$,  and darning transformation collapses each open interval in the decomposition of $[0,1]\setminus K$  into an abstract point $p^*_n$.  If all $p^*_n$ are viewed as usual points of zero Lebesgue measure,  $\bar{I}^{d,*}$,  not a singleton,  must be identified with a negligible interval.  This is incomprehensible. 
\end{example}

As $\bar{I}^{d,*}$ is probably very abstract and hard to handle,  we map it to an image space
\begin{equation}\label{eq:34}
	\widehat{I}_\bs=\bs(I)\cup \{\bs(x-): x\in D^-\}\cup \{\bs(x+): x\in D^+\},
\end{equation}	
 the closure of $\bs(I):=\{\bs(x):x\in I\}$ in $\overline{\bR}$,  by means of 
$\bs^*: \bar{I}^{d,*}\rightarrow [-\infty, \infty]$ defined as
\begin{equation}\label{eq:32}
	\bs^*(x^*):=\left\lbrace
	\begin{aligned}
	&\bs(x),\quad \quad\; x^*=x\in \left(I\setminus (D\setminus D^0)\right) \cap \bar{I}^{d,*},  \\
	&\bs(x\pm),\quad\;\; x^*=x\pm\in \bar{I}^{d,*}\text{ with }x\in D,\\
	&\bs(c_n+),\quad\, x^*=p^*_n\text{ for each } n\geq 1.
	\end{aligned}
	\right. 
\end{equation}
It is straightforward to verify the following.

\begin{lemma}\label{LM32}
The map $\bs^*$ is a homeomorphism between $\bar{I}^{d,*}$ and $\widehat{I}_\bs$.  
\end{lemma}

Let us prepare another two triples,  called \emph{regularization} and \emph{image regularization},  of $(I,\bs,\fm)$,  which provide ingredients of the desirable regular representation.  To do this,  set $\widehat{l}:=\bs(l),  \widehat{r}:=\bs(r)$ and $l^*:=\bs^{*-1}(\widehat{l}),  r^*:=\bs^{*-1}(\widehat{r})$,  where $\bs^{*-1}$ is the inverse of $\bs^*$.  We call $l^*$ or $r*$ (resp.  $\widehat{l}$ or $\widehat{r}$)  \emph{reflecting} if $l$ or $r$ is reflecting.  Let 
\begin{equation}\label{eq:33-2}
	\widehat{\fm}:= \fm\circ \bs^{-1}
\end{equation}
be the image measure of $\fm$ under the map $\bs: I\rightarrow \widehat{I}_\bs$.  

\begin{definition}
The \emph{regularization} $(I^*,\bs^*,\fm^*)$ of $(I,\bs,\fm)$ is defined as follows:
\begin{itemize}
\item[(a$^*$)] $I^*$ is the subset of $\bar{I}^{d,*}$ obtained by taking out non-reflecting endpoints from $\bar{I}^{d,*}$,  i.e.  $\bar{I}^{d,*}\setminus \{l^*,r^*\}\subset I^*\subset \bar{I}^{d,*}$ and $l^*\in I^*$ (resp. $r^*\in I^*$) if and only if $l^*$ (resp. $r^*$) is reflecting. 
\item[(b$^*$)] $\bs^*$ is defined as \eqref{eq:32}.  
\item[(c$^*$)]  $\fm^*=\widehat{\fm}\circ \bs^*$,  the image measure of $\widehat{\fm}$ under $\bs^{*-1}$.  
\end{itemize}
Accordingly,  another triple $(\widehat{I},  \widehat{\bs}, \widehat{\fm})$ is called the \emph{image regularization} of $(I,\bs,\fm)$,  where $\widehat{I}=\bs^*(I^*)$,  $\widehat{\bs}(\widehat{x})=\widehat{x}$ for $\widehat{x}\in \widehat{I}$ and $\widehat{\fm}$ is defined as \eqref{eq:33-2},
\end{definition}
	
\begin{remark}\label{RM34}
We point out that if $\widehat{l}\in \widehat{I}$,  then $\widehat{l}\in \bR$.  In addition,  if $\widehat{l}\notin \widehat{I}$,  then $\widehat{l}$ can be approximated by points in $\widehat{I}$.  Analogical facts hold for $\widehat{r}$.  Particularly,   $\widehat{I}$ is a nearly closed subset of $\bR$,  i.e.  $\widehat{I}\in \mathscr K$.  
\end{remark}

\subsection{Two regular representations}\label{SEC32}



Now we have a position to derive regular representations of $(I_0,\fm,  \sE,\sF)$.  Note the the algebra isomorphism $\Phi$ in \eqref{eq:31} plays a central role in this representation.  Our object is to map functions in $\sF$ to certain ones on $I^*$,  so that the metrics in \eqref{eq:31} are preserved.   
To accomplish this,  recall that $\sF\subset \sS$ and we introduce an auxiliary map
\[
	\widehat{\iota}: \sS\rightarrow C_\mathrm{f}(\widehat{I}),\quad  f\mapsto \widehat{f},
\]
where $C_\mathrm{f}(\widehat{I})$ is the family of all continuous functions on $\widehat{I}$ such that $\widehat{f}(\widehat{j}):=\lim_{\widehat{x}\rightarrow \widehat{j}}\widehat{f}(\widehat{x})$ exists in $\bR$ whenever $\widehat{j}\notin \widehat{I}$ is finite for $\widehat{j}=\widehat{l}$ or $\widehat{r}$,  and $\widehat{f}$ is defined as
\begin{equation}\label{eq:35}
	\widehat{f}(\bs(x)):=f(x),\; x\in I_0, \quad
	\widehat{f}(\bs(x\pm)):=f(x\pm),\; x\in D^\pm.  
\end{equation}
The map $\widehat{\iota}$ is well defined due to Lemma~\ref{LM31},  and 
\begin{equation}\label{eq:36-2}
	\widehat{f}(\widehat{l})=f(l),\quad \left(\text{resp. }\widehat{f}(\widehat{r})=f(r) \right)
\end{equation}
if $\widehat{l}$ (resp.  $\widehat{r}$) is finite.  
Clearly $\widehat{\iota}$ is an injective.  Define
\[
	\widehat{\sS}:=\{\widehat{f}=\widehat{\iota}f:  f\in \sS\},\quad \sS^*:=\{f^*=\widehat{f}\circ \bs^*:  \widehat{f}\in \widehat{\sS}\}.  
\]
Set another map
\[
	\iota^*: \sS\rightarrow \sS^*,\quad f\mapsto \iota^* f:=\widehat{\iota}f\circ \bs^*.
\]
It is straightforward to verify that the maps  
\[
	\sS\leftrightarrow \widehat{\sS} \leftrightarrow \sS^*,\quad f\leftrightarrow \widehat{f}=\widehat{\iota}f \leftrightarrow f^*=\iota^* f
\]
are bijective.  With them as algebra isomorphisms,  we will work to obtain regular representations in the sequel.  To make $\iota^*$ preserve the third metric in \eqref{eq:31},  the transformed quadratic form $(\sE^*,\sF^*)$,  called the \emph{regularization} of $(\sE,\sF)$,  is defined as 
\begin{equation}\label{eq:38-2}
\begin{aligned}
	&\sF^*:=\{f^*=\iota^* f: f\in \sF\},\\
	&\sE^*(f^*,g^*):=\sE(\iota^{*-1}f^*,\iota^{*-1}g^*),\quad f^*, g^* \in \sF^*.
\end{aligned}
\end{equation}
Accordingly the quadratic form
\begin{equation}\label{eq:38-3}
\begin{aligned}
	&\widehat\sF:=\{\widehat f=\widehat\iota f: f\in \sF\},\\
	&\widehat\sE(\widehat f,\widehat g):=\sE(\widehat{\iota}^{-1} \widehat f,\widehat{\iota}^{-1} \widehat g),\quad \widehat f, \widehat g \in \widehat \sF
\end{aligned}
\end{equation}
is called the \emph{image regularization} of $(\sE,\sF)$.  

Denote by $\dot H^1_e((\widehat{l},\widehat{r}))$ the family of all absolutely continuous function $\widehat{h}$ on $(\widehat{l},\widehat{r})$ such that $\widehat h'\in L^2((\widehat{l},\widehat{r}))$.  Clearly for $\widehat{h}\in\dot H^1_e((\widehat{l},\widehat{r}))$,   $\widehat{h}(\widehat{j}):=\lim_{\widehat{x}\rightarrow \widehat{j}}\widehat{h}(\widehat{x})$ exists in $\bR$ if $\widehat{j}=\widehat{l}$ or $\widehat{r}$ is finite.  
The lemma below is crucial to proving the main result,  Theorem~\ref{THM6}, of this section.  Its proof is elementary but long,  and we put it in Appendix~\ref{APPA}.   

\begin{lemma}\label{LM39}
It holds that
\begin{equation}\label{eq:33}
\widehat{\sS}=\left\{\widehat{f}=\widehat{h}|_{\widehat{I}}: \widehat{h}\in \dot H^1_e\left((\widehat{l},\widehat{r}) \right)\right\}.
\end{equation}
\end{lemma}



Before stating the main result,  we prepare some notations.  Note that $(\widehat{l},\widehat{r})\setminus \widehat{I}$ is open and thus can be written as a disjoint union of open intervals:
\begin{equation}\label{eq:15}
	(\widehat{l},\widehat{r})\setminus \widehat{I}=\bigcup_{k\geq 1}(\widehat{a}_k,\widehat{b}_k).  
\end{equation}
Both $\widehat{a}_k$ and $\widehat{b}_k$ are finite due to  Remark~\ref{RM34}.  

\begin{theorem}\label{THM6}
Let $(I^*,\bs^*,\fm^*)$ and $(\widehat{I},\widehat{\bs},\widehat{\fm})$ be the regularization and image regularization of $(I,\bs,\fm)$.  Further let $(\sE^*,\sF^*)$ and $(\widehat{\sE},\widehat{\sF})$ be the regularization and image regularization of $(\sE,\sF)$.  
\begin{itemize}
\item[\rm (1)] $\fm^*$ (resp.  $\widehat{\fm}$) is a fully supported Radon measure on $I^*$ (resp.  $\widehat{I}$).  
\item[\rm (2)] $(\sE^*,\sF^*)$ is a regular and irreducible Dirichlet form on $L^2(I^*,\fm^*)$.  Particularly,  $(I^*,\fm^*,\sE^*,\sF^*)$ is a regular representation of $(I_0,\fm,\sE,\sF)$.  
\item[\rm (3)] $(\widehat{\sE},\widehat{\sF})$ is a regular and irreducible Dirichlet form on $L^2(\widehat I,\widehat \fm)$ admitting the representation:
\begin{equation}\label{eq:321}
\begin{aligned}
	&\widehat{\sF}=\left\{\widehat f\in L^2(\widehat I,\widehat\fm)\cap \widehat \sS: \widehat f(\widehat j)=0\text{ if }\widehat j\in \bR\setminus \widehat{I}\text{ for }\widehat j=\widehat l\text{ or }\widehat r\right\},  \\
	&\widehat{\sE}(\widehat{f},\widehat{f})=\frac{1}{2}\int_{\widehat{I}} \widehat f'(\widehat x)^2d\widehat x+\frac{1}{2}\sum_{k\geq 1}\frac{\left(\widehat{f}(\widehat{a}_k)-\widehat{f}(\widehat{b}_k)\right)^2}{|\widehat{b}_k-\widehat{a}_k|},\quad \widehat{f}\in \widehat{\sF},
\end{aligned}
\end{equation}
where $\widehat{a}_k,\widehat{b}_k$ appear in \eqref{eq:15}.  Particularly,  $(\widehat{I},\widehat{\fm}, \widehat{\sE},\widehat{\sF})$ is also a regular representation of $(I_0,\fm,\sE,\sF)$.  
\end{itemize}
\end{theorem}
\begin{proof}
\begin{itemize}
\item[(1)] We only prove the assertion for $\widehat{\fm}$.  To show $\widehat{\fm}$ is fully supported on $\widehat{I}$,  argue by contradiction and suppose that  for some $\widehat a<\widehat b$ with $(\widehat a, \widehat b)\cap \widehat{I}\neq \emptyset$, 
\[
	\widehat{\fm}\left((\widehat a,\widehat b)\cap \widehat{I}\right)=0.  
\]
Note that an isolated point in $\widehat{I}$ must be $\bs(x)$ for some $x\in D^0$ or $\bs((c_n,d_n))$ for some isolated interval $(c_n,d_n)$ in \eqref{eq:14}.  Hence \text{(DM)} yields that  $(\widehat a, \widehat b)\cap \widehat{I}$ contains no isolated points.  Take $\widehat{x}\in (\widehat{a},\widehat{b})\cap \widehat{I}$.  Then there is a sequence $\widehat{x}_p=\bs(x_p)$ with $x_p\in I$ such that $\widehat{x}_p\rightarrow \widehat{x}$.  We may and do assume that $\widehat{x}_p, \widehat{x}_{p+1},  \widehat x_{p+2}, \widehat{x}_{p+3}\in (\widehat{a},\widehat{b})$ and $\widehat{x}_p< \widehat{x}_{p+1}<\widehat{x}_{p+2}<\widehat{x}_{p+3}$ for some $p$.  It follows that $x_p<x_{p+1}<x_{p+2}<x_{p+3}$ and
\[
	\{\bs(y): y\in (x_p,x_{p+3})\}\subset (\widehat{a},\widehat{b})\cap \widehat{I}.  
\]
Hence $\fm((x_p,x_{p+3}))\leq \widehat{\fm}((\widehat{a},\widehat{b})\cap \widehat{I})=0$.  By Remark~\ref{RM22},  $(x_p,x_{p+3})\subset (c_n,d_n)$ for some $(c_n,d_n)$ in \eqref{eq:14}.  Particularly,  $$\widehat{x}_{p+1}=\bs(x_{p+1})=\bs(x_{p+2})=\widehat{x}_{p+2},$$ as leads to the contradiction with $\widehat{x}_{p+1}<\widehat{x}_{p+2}$.  

Next we prove that $\widehat{\fm}$ is Radon on $\widehat{I}$.  Let $\widehat{J}:=\langle \widehat{l},  \widehat{r}\rangle$ be the interval ended by $\widehat{l}$ and $\widehat{r}$ such that $\widehat{l}\in \widehat{J}$ ($\widehat{r}\in \widehat{J}$) if and only if $\widehat{l}\in \widehat{I}$ ($\widehat{r}\in \widehat{I}$).  It suffices to show that $\widehat{\fm}([\widehat a,\widehat b])<\infty$ for any interval $[\widehat a,\widehat b]\subset \widehat{J}$.  To accomplish this,  set 
\[
	a:=\inf\{x\in I: \bs(x)\geq  \widehat{a}\},\quad b:=\sup\{x\in I: \bs(x)\leq \widehat{b}\}.  
\]
Then $\{x\in I: \bs(x)\in [\widehat{a},\widehat{b}]\}\subset [a,b]$.  Note that if $a=l$,  then $\widehat{l}\leq \widehat{a}\leq \bs(l+)=\bs(l)=\widehat{l}$.  Thus $\widehat{l}=\widehat{a}\in \widehat{I}$ must be reflecting endpoint of $\widehat{I}$.  Accordingly $l \in I$ is reflecting.  Analogously if $b=r$,  then $r\in I$ is reflecting.  As a result,  $\widehat{\fm}([\widehat{a},\widehat{b}])=\fm\circ \bs^{-1}([\widehat{a},\widehat{b}])\leq \fm([a,b])<\infty$.  
\item[(2)] This assertion is equivalent to the third one due to Lemma~\ref{LM32}.  We will only give the proof of the third assertion.  
\item[(3)] Firstly we formulate \eqref{eq:321}.  The expression of $\widehat{\sE}(\widehat{f},\widehat{f})$ is a consequence of  \eqref{eq:38} and \eqref{eq:39-2}.  It suffices to prove the first identity in \eqref{eq:321}.  Denote the family on its right hand side by $\widehat\sG$.  

Take $f\in \sF$ and let $\widehat{f}:=\widehat{\iota}f\in \widehat{\sF}$.  Then $\widehat{f}\in \widehat{\sS}\subset C_\mathrm{f}(\widehat{I})$.  Suppose $\widehat{r}\notin \widehat{I}$ and we show that $\widehat{f}|_{[0, \widehat{r})}\in L^2([0, \widehat{r}),\widehat{\fm})$,  so that $\widehat{f}\in L^2(\widehat{I},\widehat{\fm})$ can be obtained by virtue of $\widehat f\in C_\mathrm{f}(\widehat{I})$ and the first assertion.  To do this,  assume without lose of generality that $0\notin \{c_n,d_n:n\geq 1\}$.  Note that
\[
\int_{[0,\widehat{r})}\widehat{f}(\widehat{x})^2\widehat{\fm}(d\widehat{x})=\int_{\{x\in I:  \bs(x)\in [0,\widehat{r})\}} f(x)^2\fm(dx).
\]	
When $r\notin U$,  i.e.  $\bs(x)$ is strictly increasing as $x\uparrow r$,  we have
\begin{equation}\label{eq:39-3}
	\{x\in I:  \bs(x)\in [0,\widehat{r})\}=[0,  r).  
\end{equation}
Otherwise if $r=d_n$ for some $n$,  then $\bs(x)=\widehat{r}<\infty$ for $x\in (c_n,d_n]$ and 
\begin{equation}\label{eq:318}
[0, c_n)\subset 	\{x\in I:  \bs(x)\in [0,\widehat{r})\}\subset [0,c_n].  
\end{equation}
For either case we can obtain $\int_{[0,\widehat{r})}\widehat{f}(\widehat{x})^2\widehat{\fm}(d\widehat{x})<\infty$ by means of $f\in L^2(I_0,\fm)$.  On the other hand,  if $\widehat{r}\in\bR\setminus \widehat{I}$,  then $r$ is approachable but not reflecting, and thus $\fm(r-)=\infty$ or $\fm(\{r\})=\infty$.  We always have $f(r)=0$ by the definition of $\sF$,  and because of \eqref{eq:36-2},  $\widehat{f}(\widehat{r})=f(r)=0$.  Therefore $\widehat{\sF}\subset \widehat{\sG}$ is concluded.  

To the contrary,  let $\widehat{f}\in \widehat{\sG}\subset \widehat{\sS}$.   Then there exists $f\in \sS$ such that $\widehat{f}=\widehat{\iota}f$.  We need to prove that $f\in \sF$.  To show $f\in L^2(I_0,\fm)$,  we assert that $f|_{[0,r]\cap I_0}\in L^2([0,r]\cap I_0,\fm)$. ($f|_{[l,0]\cap I_0}\in L^2([l,0]\cap I_0,\fm)$ can be obtained similarly.)  This trivially holds when $r$ is reflecting.  Suppose $r$ is not reflecting,  equivalently $\widehat{r}\notin \widehat{I}$.  In the case $r\notin U$,  \eqref{eq:39-3} tells us that 
\[
\int_{[0,r)}f(x)^2\fm(dx)=\int_{[0,\widehat{r})}\widehat{f}(\widehat{x})^2\widehat{\fm}(d\widehat{x})<\infty.  
\]
In the case \eqref{eq:318},  $\widehat{r}\in \bR\setminus \widehat{I}$ and hence $f(r)=\widehat{f}(\widehat{r})=0$ due to \eqref{eq:36-2} and $\widehat{f}\in \widehat{\sG}$.  Particularly $f(x)=0$ for any $x\in (c_n,d_n]$.  As a result one can verify that $\int_{[0,r)}f(x)^2\fm(dx)<\infty$.  Finally it suffices to prove $f(r)=0$ if $r\notin I_0$ is approachable.  Since $\widehat{r}$ is finite but not reflecting,  i.e.  $\widehat{r}\in \bR\setminus \widehat{I}$,  it follows from \eqref{eq:36-2} and $\widehat{f}\in \widehat{\sG}$ that $f(r)=\widehat{f}(\widehat{r})=0$.  Eventually we obtain $f\in \sF$, and $\widehat{\sG}\subset \widehat{\sF}$ is concluded. 

Repeating the argument in the proof of Theorem~\ref{LM12},  one can easily verify that $(\widehat{\sE},\widehat{\sF})$ is a Dirichlet form on $L^2(\widehat{I},\widehat{\fm})$ in the wide sense.  We turn to prove its regularity.  This will be completed by treating several cases separately as follows.  (Another proof for regularity involving transformation of time change will be shown in Theorem~\ref{THM19}.)  

\emph{Both $\widehat l$ and $\widehat r$ are finite}.   Clearly $\widehat{\sF}\subset C_\infty(\widehat{I})$,  where $C_\infty(\widehat{I})$ stands for the family of continuous functions vanishing at each endpoint not contained in $\widehat{I}$.  
 To prove the regularity of $(\widehat{\sE},\widehat{\sF})$ on $L^2(\widehat{I},\widehat{\fm})$,  in view of \cite[Lemma~1.3.12]{CF12},  we only need to prove that $\widehat{\sF}$ separates the points in $\widehat{I}$.  Fix $\widehat{x}, \widehat{y}\in \widehat{I}$ with $\widehat{x}\neq \widehat{y}$.  Let $\widehat{J}:=\langle \widehat{l},  \widehat{r}\rangle$ be the interval as in the proof of the first assertion, and take $\widehat{f}\in C_c^\infty(\widehat{J})$ such that $\widehat f(\widehat t)=\widehat t$ for $\widehat t\in [\widehat{x}, \widehat{y}]$.  Then $\widehat{f}|_{\widehat{I}}\in \widehat{\sF}$ separates $\widehat{x}$ and $\widehat{y}$.  

\emph{Neither $\widehat l$ nor $\widehat r$ is finite}.  In this case $\widehat{l}=-\infty$,  $\widehat{r}=\infty$,  and $\widehat{J}=\bR$.  Analogical to the previous case,  we can conclude that $\widehat{\sF}\cap C_c(\widehat{I})$ is dense in $C_c(\widehat{I})$ with respect to the uniform norm.  To show the $\widehat{\sE}_1$-denseness of $\widehat{\sF}\cap C_c(\widehat{I})$ in $\widehat{\sF}$,  fix a bounded function $\widehat{f}\in \widehat{\sF}$ with $M:=\sup_{\widehat x\in \widehat{I}}|\widehat{f}(\widehat{x})|<\infty$.   Take a sequence of functions $\varphi_n\in C_c^\infty(\bR)$ such that
\begin{equation}\label{eq:311}
\begin{aligned}
&\varphi_n(\widehat{x})=1\quad \text{for }|\widehat{x}|< n;  \quad \varphi_n(\widehat{x})=0\quad \text{for } |\widehat{x}|>2n+1;\\
&|\varphi'_n(\widehat{x})|\leq 1/n,\quad n\leq |\widehat{x}|\leq 2n+1;\quad 0\leq \varphi_n(\widehat{x})\leq 1,\quad \widehat{x}\in \bR.  
\end{aligned}\end{equation}
Set $\widehat{f}_n:=\widehat{f}\cdot \varphi_n|_{\widehat{I}}$.  Since $\varphi_n|_{\widehat{I}}\in \widehat{\sF}$,  it follows that $\widehat{f}_n\in \widehat{\sF}\cap C_c(\widehat{I})$.  Clearly $\widehat{f}_n$ converges to $\widehat{f}$ in $L^2(\widehat{I},\widehat{\fm})$.  We prove that $\widehat{\sE}(\widehat{f}_n-\widehat{f},\widehat{f}_n-\widehat{f})\rightarrow 0$.  Denote by $B_R:=\{\widehat{x}: |\widehat{x}|<R\}$ for $R>0$.  In view of \eqref{eq:321},  $2\widehat{\sE}(\widehat{f}_n-\widehat{f},\widehat{f}_n-\widehat{f})$ is not greater than $A^1_n+A^2_n+A^3_n$, where
\[
	A^1_n:=\int_{\widehat{I}\cap B_{2n+1}^c} \widehat{f}'(\widehat{x})^2d\widehat{x}+\sum_{k: \widehat{b}_k>2n+1\text{ or }\widehat{a}_k<-2n-1}\frac{\left(\widehat{f}(\widehat{a}_k)-\widehat{f}(\widehat{b}_k)\right)^2}{|\widehat{b}_k-\widehat{a}_k|},
\]
\[
A^2_n:=\int_{\widehat{I} \cap (B_{2n+1}\setminus B_n)} \left(\widehat{f}'(\widehat{x})(\varphi_n(\widehat{x})-1)+\widehat{f}(\widehat{x})\varphi'_n(\widehat{x})\right)^2d\widehat{x}
\]
and 
\[
	A^3_n:=\sum_{k: n\leq |\widehat{a}_k|\leq 2n+1\text{ or }n\leq |\widehat{b}_k|\leq 2n+1}\frac{\left(\left(\widehat{f}\cdot(\varphi_n-1)\right)(\widehat{a}_k)-\left(\widehat{f}\cdot(\varphi_n-1)\right)(\widehat{b}_k)\right)^2}{|\widehat{b}_k-\widehat{a}_k|}.  
\]
Clearly $A^1_n\rightarrow 0$ as $n\rightarrow \infty$.  Regarding the second term,  we have by means of \eqref{eq:311} that
\[
	A^2_n\leq 2\int_{\widehat{I} \cap (B_{2n+1}\setminus B_n)}\left(\widehat{f}'(\widehat{x})^2+\frac{\widehat{f}(\widehat{x})^2}{n^2} \right)d\widehat{x}\rightarrow 0.  
\]
Note that $A^3_n\leq 2A^{31}_n+2A^{32}_n+2A^{33}_n$,  where
\[
	A^{31}_n:= \sum_{k: n\leq |\widehat{a}_k|\leq 2n+1\text{ or }n\leq |\widehat{b}_k|\leq 2n+1}\frac{\left(\widehat{f}(\widehat{a}_k)-\widehat{f}(\widehat{b}_k)\right)^2}{|\widehat{b}_k-\widehat{a}_k|}\rightarrow 0,
\]
\[
	A^{32}_n:=\sum_{k: n\leq |\widehat{a}_k|\leq 2n+1\text{ or }n\leq |\widehat{b}_k|\leq 2n+1}\frac{\left(\widehat{f}(\widehat{a}_k)\varphi_n(\widehat{a}_k)-\widehat{f}(\widehat{b}_k)\varphi_n(\widehat{a}_k)\right)^2}{|\widehat{b}_k-\widehat{a}_k|}\leq A^{31}_n\rightarrow 0,
\]
and 
\[
\begin{aligned}
A^{33}_n&:=\sum_{k: n\leq |\widehat{a}_k|\leq 2n+1\text{ or }n\leq |\widehat{b}_k|\leq 2n+1}\frac{\left(\widehat{f}(\widehat{b}_k)\varphi_n(\widehat{a}_k)-\widehat{f}(\widehat{b}_k)\varphi_n(\widehat{b}_k)\right)^2}{|\widehat{b}_k-\widehat{a}_k|} \\
&\leq \frac{M^2}{n^2}\sum_{k: n\leq |\widehat{a}_k|\leq 2n+1\text{ or }n\leq |\widehat{b}_k|\leq 2n+1}|\widehat{a}_k-\widehat{b}_k|\rightarrow 0.
\end{aligned}\]
Therefore $\widehat{\sE}(\widehat{f}_n-\widehat{f},\widehat{f}_n-\widehat{f})\rightarrow 0$ is concluded.  

The reminder cases can be treated analogously and we eventually conclude that $(\widehat{\sE},\widehat{\sF})$ is a regular Dirichlet form on $L^2(\widehat{I},\widehat{\fm})$.   

Next we prove the irreducibility of $(\widehat{\sE},\widehat{\sF})$.  Argue by contradiction and suppose that $\widehat A\subset \widehat{I}$ is a non-trivial $\{\widehat T_t\}$-invariant set,  i.e.  $\widehat\fm(\widehat A)\neq 0$ and $\widehat\fm(\widehat{I}\setminus \widehat A)\neq 0$,  where $\{\widehat T_t\}$ is the $L^2$-semigroup of $(\widehat{\sE},\widehat{\sF})$.  Then there exists a closed interval $0\in [\widehat l_0,\widehat r_0]\subset \widehat J$,  $\widehat{l}_0, \widehat{r}_0\in \widehat{I}$,  such that 
\begin{equation}\label{eq:320}
	\widehat\fm( [\widehat l_0,\widehat r_0] \cap \widehat A)>0,\quad \widehat \fm([\widehat l_0,\widehat r_0]\setminus \widehat A)>0.  
\end{equation}
Take another closed interval $[\widehat{l}_1,\widehat{r}_1]\subset \widehat{J}$ such that $[\widehat{l}_0,\widehat{r}_0]\subset [\widehat{l}_1,\widehat{r}_1]$ and $\widehat{l}_1<\widehat{l}_0$ (resp.  $\widehat{r}_1>\widehat{r}_0$) unless $\widehat{l}_0=\widehat{l}\in \widehat{I}$ (resp.  $\widehat{r}_0=\widehat{r}\in \widehat{I}$).  Note that $\{\widehat{h}|_{\widehat{I}}: \widehat{h}\in C_c^\infty(\widehat{J})\}\subset \widehat{\sF}$ due to Lemma~\ref{LM39}. Take a function $\widehat{f}=\widehat h|_{\widehat{I}}\in \widehat{\sF}$ with $\widehat{h}\in C_c^\infty(\widehat{J})$ and $\widehat{h}=1$ on $[\widehat l_1,\widehat r_1]$.  In view of \cite[Proposition~2.1.6]{CF12},  $\widehat{f}\cdot 1_{\widehat{A}}\in \widehat{\sF}$ and on account of Lemma~\ref{LM39},  $\widehat{f}\cdot 1_{\widehat{A}}\in \widehat{\sF}$ admits a continuous $\fm$-a.e. version denoted by $\tilde{f}_1$.  Clearly $\tilde{f}_1=0$ or $1$ pointwisely on $[\widehat l_1,\widehat r_1]\cap \widehat{I}$.  Consider the family of intervals:
\[
	\mathscr{I}:=\{(\widehat{a}_k,\widehat{b}_k)\subset [\widehat{l}_0,\widehat{r}_0]: k\geq 1\}.  
\]
When $\mathscr I$ is empty,  $\tilde{f}_1$ must be constant on $[\widehat{l}_0,\widehat{r}_0]$,  as leads to a contradiction with \eqref{eq:320}.  Now we consider $\mathscr I\neq \emptyset$ and assert that there exists $(\widehat{a}_k,\widehat{b}_k)\in \mathscr I$ such that $\tilde{f}_1(\widehat{a}_k)\neq \tilde{f}_1(\widehat{b}_k)$.  To accomplish this, set
\[
\begin{aligned}
\widehat\alpha&:=\sup\{\widehat{x}\in [\widehat{l}_0,0]\cap \widehat{I}: \tilde{f}_1(\widehat{x})\neq \tilde{f}_1(0) \}, \\
\widehat\beta&:=\inf\{\widehat{x}\in [0,\widehat{r}_0]\cap \widehat{I}: \tilde{f}_1(\widehat{x})\neq \tilde{f}_1(0)\},
\end{aligned}\]
where $\sup \emptyset:=-\infty$ and $\inf \emptyset:=\infty$.  Then \eqref{eq:320} implies that $\widehat\alpha \neq -\infty$ or $\widehat\beta \neq \infty$.  The former case leads to $\widehat\alpha=\widehat{a}_p$ for some $p$ and the latter one leads to $\widehat\beta=\widehat{b}_q$ for some $q$ by virtue of the continuity of $\tilde{f}_1$.  We have either $\tilde{f}_1(\widehat{a}_p)\neq \tilde{f}_1(\widehat{b}_p)$ or $\tilde{f}_1(\widehat{a}_q)\neq \tilde{f}_1(\widehat{b}_q)$.  Therefore the existence of such $(\widehat{a}_k,\widehat{b}_k)$ is obtained.  Without loss of generality assume that
\[
	\widehat{f}_1(\widehat{a}_k)=1,\quad \widehat{f}_1(\widehat{b}_k)=0.  
\]
Note that there exists $\widehat{\varepsilon}>0$ such that $\widehat{a}_k+\widehat{\varepsilon}<\widehat{b}_k-\widehat{\varepsilon}$ and 
\[
\begin{aligned}
	&(\widehat{a}_k-\widehat{\varepsilon},  \widehat{a}_k+\widehat{\varepsilon})\cap \widehat{I}\subset  \widehat{A},\quad \widehat{\fm}\text{-a.e.},  \\
	&(\widehat{b}_k-\widehat{\varepsilon},  \widehat{b}_k+\widehat{\varepsilon})\cap \widehat{I}\subset  \widehat{I}\setminus \widehat{A},\quad \widehat{\fm}\text{-a.e.},
\end{aligned}\]
as can be obtained by means of the continuity of $\tilde{f}_1$ and 
\[
	\{\widehat{x}\in [\widehat{l}_1,\widehat{r}_1]\cap \widehat{I}: \tilde{f}_1(\widehat{x})=1\}=\widehat{A}\cap [\widehat{l}_1,\widehat{r}_1],\quad \widehat{\fm}\text{-a.e. }
\]
When $\widehat{a}_k>\widehat{l}$ (resp.  $\widehat{b}_k<\widehat{r}$),  we may and do assume that $\widehat{a}_k-\widehat{\varepsilon}>\widehat{l}$ (resp.  $\widehat{b}_k+\widehat{\varepsilon}<\widehat{r}$).
Take another function $\widehat{g}\in \widehat{\sF}$ such that 
\[
\begin{aligned}
&\widehat{g}|_{(\widehat{a}_k-\widehat{\varepsilon}/2,  \widehat{a}_k+\widehat{\varepsilon}/2)\cap \widehat{I}}\equiv 1, \quad \widehat{g}|_{(\widehat{b}_k-\widehat{\varepsilon}/2,  \widehat{b}_k+\widehat{\varepsilon}/2)\cap \widehat{I}}\equiv 1, \\
&\widehat{g}|_{\widehat{I}\setminus \left((\widehat{a}_k-\widehat{\varepsilon},  \widehat{a}_k+\widehat{\varepsilon}) \cup (\widehat{b}_k-\widehat{\varepsilon},  \widehat{b}_k+\widehat{\varepsilon}) \right)}\equiv 0.  
\end{aligned}\]
Using \cite[Proposition~2.1.6]{CF12},  we get that $\widehat{g}_1:=\widehat{g}\cdot 1_{\widehat{A}}\in \widehat{\sF},  \widehat{g}_2:=g-\widehat{g}_1\in \widehat{\sF}$ and $\widehat{\sE}(\widehat{g}_1,\widehat{g}_2)=0$.  However,  in view of \eqref{eq:321},  a computation yields that
\[
	\widehat{\sE}(\widehat{g}_1,\widehat{g}_2)=-\frac{1}{2|\widehat{b}_k-\widehat{a}_k|}\neq 0
\]
leading to a contradiction.  

Finally it suffices to show that $(\widehat{I},\widehat{\fm},\widehat{\sE},\widehat{\sF})$ is a regular representation of $(I_0,\fm,\sE,\sF)$.  Denote by $\Phi$ the restriction of $\widehat{\iota}$ to $\sF_b$.  In view of \eqref{eq:33-2}, \eqref{eq:35} and \eqref{eq:38-3},  one may easily get that $\Phi$ is an algebra isomorphism and $\Phi \sF_b\subset \widehat{\sF}_b$.  To verify \eqref{eq:31},  we only consider the case $l\notin I_0, r\in I_0$.  The other cases can be treated analogically.  The last identity in \eqref{eq:31} is the consequence of \eqref{eq:38-3}.  If $l\neq c_n$,  then 
\[
	\widehat{I}=\bs(I_0)\cup \{\bs(x-):x\in D^-\}\cup \{\bs(x+): x\in D^+\}.  
\]
By means of $\widehat{\fm}=\fm\circ \bs^{-1}$ and \eqref{eq:35}, we can obtain the first and second identities in \eqref{eq:31}.  If $l=c_n$ for some $n$,  then (DK) implies that $d_n\notin D$.  Particularly,  
\[
	\widehat{I}= \bs(I\setminus [c_n,d_n])\cup \{\bs(x-):x\in D^-\}\cup \{\bs(x+): x\in D^+\}.  
\]
Note that $f(x)=0$ for $f\in \sF$ and $x\in [c_n,d_n]$ and $\widehat{f}(\widehat{l})=0$ for $\widehat{f}\in \widehat{\sF}$.  These facts,  together with $\widehat{\fm}=\fm\circ \bs^{-1}$ and \eqref{eq:35},  yield the first and second identities in \eqref{eq:31}. 
\end{itemize}
That completes the proof.
\end{proof}


\begin{remark}\label{RM38}
Without assuming (DK),  it might occur that 
\[
	l=c_n<d_n\in D
\]
for some $n$ while $\fm(l+)+\fm(\{l\})=\infty$.  Then $\widehat{l}$ is isolated in $\widehat{I}_\bs$ with $\widehat{\fm}(\{\widehat{l}\})=\infty$ and $\widehat{l}=\widehat{a}_k$ for some $k$.  This endpoint becomes an absorbing point and for $\widehat{f}\in \widehat{\sF}$,  we must have $\widehat{f}(\widehat{l})=0$.  As a consequence,  $\widehat{\sE}(\widehat{f},\widehat{f})$ contains a killing part $\widehat{f}(\widehat{b}_k)^2/(2|\widehat{b}_k-\widehat{a}_k|)$.  In other words,  (DK) excludes killing part from the desirable Dirichlet form. 
\end{remark}

It is well known that every regular Dirichlet form corresponds to a unique symmetric Hunt process; see,  e.g., \cite{FOT11}.  Denote by $X^*=(X^*_t)_{t\geq 0}$ (resp.  $\widehat{X}=(\widehat{X}_t)_{t\geq 0}$) the $\fm^*$-symmetric (resp. $\widehat{\fm}$-symmetric) Hunt process on $I^*$ (resp.  $\widehat{I}$) associated with $(\sE^*,\sF^*)$ (resp.  $(\widehat{\sE}, \widehat{\sF})$).  Note that $\widehat{X}$ is equivalent to $\left(\bs^*(X^*_t)\right)_{t\geq 0}$ in the sense that their transition functions coincide.  We would call $X^*$ and $\widehat{X}$ the \emph{regularized Markov process} and \emph{image regularized Markov process associated with} $(I,\bs,\fm)$ respectively if there is no risk of ambiguity.


\subsection{Image regularized Markov process}\label{SEC33}

In this subsection we will show that the image regularized Markov process $\widehat{X}$ is a time-changed Brownian motion with speed measure $\widehat{\fm}$.  

Let $\widehat{J}:=\langle \widehat{l},\widehat{r}\rangle$ be the interval as in the proof of Theorem~\ref{THM6},  i.e.  $\widehat{l}\in \widehat{J}$ ($\widehat{r}\in \widehat{J}$) if and only if $\widehat{l}\in \widehat{I}$ ($\widehat{r}\in \widehat{I}$).  Denote by $\widehat B=(\widehat B_t)_{t\geq 0}$ the Brownian motion on $\widehat{J}$ which is absorbing at each finite open endpoint and reflecting at each finite closed endpoint.  The life time of $\widehat{B}$ is denoted by $\widehat{\zeta}$.  For example,  when $\widehat{J}=[0,1)$,  $\widehat B$ is a Brownian motion on $[0,1)$, which is reflecting at $0$ while absorbing at $1$,  and its lifetime is identified with the first time when $\widehat{B}$ hits $1$. The associated Dirichlet form of $\widehat B$ on $L^2(\widehat{J})$ is 
\[
\begin{aligned}
	&H^1_0(\widehat{J}):=\bigg\{\widehat f\in L^2(\widehat{J}): \widehat f\text{ is absolutely continuous on } (\widehat{l},\widehat{r}) \text{ and }\\
	&\qquad \qquad\qquad \int_{\widehat{J}}\widehat f'(\widehat x)^2d\widehat x<\infty,  \widehat f(\widehat j)=0\text{ if }\widehat j\notin \widehat{J}\text{ is finite for }\widehat j=\widehat{l} \text{ or }\widehat{r}\},\\
	&\frac{1}{2}\mathbf{D}(\widehat f,\widehat g):=\frac{1}{2}\int_{\widehat{J}}\widehat f'(\widehat x)\widehat g'(\widehat x)d\widehat x,\quad \widehat f,\widehat g\in H^1_0(\widehat{J}).  
\end{aligned}\]
Since viewed as a zero extension to $\widehat{J}$ is Radon on $\widehat{J}$,  $\widehat{\fm}$ is smooth with respect to $(\frac{1}{2}\mathbf{D}, H^1_0(\widehat{J}))$.  Clearly the quasi support of $\widehat{\fm}$ is identified with its topological support $\widehat{I}$.   Denote the PCAF of $\widehat{\fm}$ with respect to $\widehat{B}$ by $\widehat{A}=(\widehat{A}_t)_{t\geq 0}$.  Set 
\[
\widehat \tau_t:=\left\lbrace 
	\begin{aligned}
		&\inf\{s: \widehat{A}_s>t\},\quad t<\widehat{A}_{\widehat{\zeta}-},\\
		&\infty,\qquad\qquad\qquad\;\; t\geq \widehat{A}_{\widehat{\zeta}-}  
	\end{aligned}
\right.
\]
and $\check{X}_t:=\widehat{B}_{\widehat\tau_t}, \check{\zeta}:=\widehat{A}_{\widehat{\zeta}-}$.  
Then $\check{X}$ is a right process on $\widehat{I}$ with lifetime $\check{\zeta}$,  called the \emph{time-changed Brownian motion} with speed measure $\widehat \fm$;  see,  e.g.,  \cite[Theorem~A.3.11]{CF12}.  The following result,  together with \cite[Corollary~5.2.10]{CF12},  shows that the image regularized Markov process $\widehat{X}$ is identified with $\check{X}$.  This argument involving time change also gives an alternative proof for the regularity of $(\widehat{\sE},\widehat{\sF})$.  
The terminologies concerning trace Dirichlet forms are referred to in,  e.g.,  \cite[\S5.2]{CF12}. 


\begin{theorem}\label{THM19}
$(\widehat{\sE},\widehat{\sF})$ is the trace Dirichlet form of $(\frac{1}{2}\mathbf{D}, H^1_0(\widehat{J}))$ on $L^2(\widehat{I}, \widehat{\fm})$.  
\end{theorem}
\begin{proof}
Denote by $(\check{\sE},\check{\sF})$ the trace Dirichlet form of $(\frac{1}{2}\mathbf{D}, H^1_0(\widehat{J}))$ on $L^2(\widehat{I}, \widehat{\fm})$.  Let $H^1_e(\widehat{J})$ be the extended Dirichlet space of $(\frac{1}{2}\mathbf{D}, H^1_0(\widehat{J}))$,  i.e. 
\begin{equation}\label{eq:320-2}
	\begin{aligned}
	&H^1_e(\widehat{J}):=\bigg\{\widehat f\text{ on }\widehat{J}:  \widehat f\text{ is absolutely continuous on }(\widehat{l},\widehat{r}) \text{ and }\\
	&\qquad \qquad\qquad \int_{\widehat{J}}\widehat f'(x)^2dx<\infty,  \widehat f(\widehat j)=0\text{ if }\widehat j\notin \widehat{J}\text{ is finite for }\widehat j=\widehat{l} \text{ or }\widehat{r}\}; 
	\end{aligned}
\end{equation}
see,  e.g.,  \cite[(3.5.10)]{CF12}.
We only need to prove
\begin{equation}\label{eq:17}
	H^1_e(\widehat{J})|_{\widehat{I}}\cap L^2(\widehat{I},\widehat{\fm})=\widehat{\sF}
\end{equation}
and for $\widehat{\varphi}\in \widehat{\sF}$,  $\check{\sE}(\widehat{\varphi},\widehat{\varphi})=\widehat{\sE}(\widehat{\varphi}, \widehat{\varphi})$.  The identity \eqref{eq:17} can be straightforwardly verified by means of Lemma~\ref{LM39} and \eqref{eq:321}.  The expression of $\check{\sE}(\widehat{\varphi},\widehat{\varphi})$ for $\widehat{\varphi}\in H^1_e(\widehat{J})|_{\widehat{I}}\cap L^2(\widehat{I},\widehat{\fm})$ can be obtained by mimicking the proof of \cite[Theorem~2.1]{LY17}.  It is identified with $\widehat{\sE}(\widehat{\varphi},\widehat{\varphi})$ expressed as \eqref{eq:321}.  That eventually completes the proof. 
\end{proof}

As obtained in Theorem~\ref{THM6},  $(\widehat{\sE},\widehat{\sF})$ is irreducible.  By virtue of Theorem~\ref{THM19},  we give a criterion for the characterizations of irreducibility and global property of $(\widehat{\sE},\widehat{\sF})$ or $\widehat{X}$. 

\begin{corollary}\label{COR310}
The following hold:
\begin{itemize}
\item[\rm (1)] $(\widehat{\sE},\widehat{\sF})$ is transient,  if and only if either $\widehat{l}\in \bR\setminus \widehat{I}$ or $\widehat{r}\in \bR\setminus \widehat{I}$.  This is also equivalent to that either $l$ or $r$ is approachable but not reflecting.  Otherwise $(\widehat{\sE},\widehat{\sF})$ is recurrent.
\item[\rm (2)] Every singleton contained in $\widehat{I}$ is of positive capacity with respect to $\widehat{\sE}$.  Particularly,  $\widehat{X}$ is pointwisely irreducible in the sense that 
\[
\widehat{\mathbf{P}}_{\widehat{x}}(\widehat{\sigma}_{\widehat{y}}<\infty)>0
\]
for any $\widehat{x},\widehat{y}\in \widehat{I}$,  where $\widehat{\mathbf{P}}_{\widehat{x}}$ is the probability measure on the sample space of $\widehat{X}$ starting from $\widehat{x}$ and $\widehat{\sigma}_{\widehat{y}}:=\inf\{t>0: \widehat{X}_t=y\}$.  
\end{itemize}
\end{corollary}
\begin{proof}
In view of \cite[Theorem~5.2.5]{CF12},  $(\widehat{\sE},  \widehat{\sF})$ is transient if and only if so is $(\frac{1}{2}\mathbf{D}, H^1_0(\widehat{J}))$.  This,  together with \cite[Theorem~2.2.11]{CF12},  yields the first equivalent condition to the transience of $(\widehat{\sE},  \widehat{\sF})$.  The second equivalent condition is obvious.  Another assertion is the consequence of \cite[Theorems~3.5.6~(1) and 5.2.8~(2)]{CF12}.  That completes the proof. 
\end{proof}

\subsection{Homeomorphisms between regular representations}

In this subsection we will show that all regular representations of $(I_0,\fm,\sE,\sF)$ are essentially homoemorphic.  In other words,  a Markov process corresponding to certain regular representation must be a homeomorphic image of $X^*$ (as well as $\widehat{X}$).  The lemma below is useful for obtaining this result.  

\begin{lemma}\label{LM310}
Let $\{\widehat{F}_n: n\geq 1\}$ be an $\widehat{\sE}$-nest and $\widehat{K}$ be a compact subset of $\widehat{I}$.  Then $\widehat{K}\subset \widehat{F}_n$ for some $n\geq 1$.  
\end{lemma}
\begin{proof}
We first prove that for any $\widehat x\in \widehat K$,  there exists $\varepsilon>0$ such that
\begin{equation}\label{eq:76}
	(\widehat x-\varepsilon, \widehat x+\varepsilon)\cap (\widehat I\setminus \widehat{F}_n)=\emptyset,\quad \text{for some }n\geq 1.  
\end{equation}
Argue by contradiction and take $\widehat x\in \widehat K$ such that $(\widehat x-\varepsilon,\widehat x+\varepsilon)\cap (\widehat I\setminus \widehat{F}_n)\neq \emptyset$ for any $\varepsilon>0$ and $n\geq 1$.  Particularly,  there exists a sequence $\widehat x_n\in \widehat I\setminus \widehat{F}_n$ such that $\widehat x_n\rightarrow \widehat x$.   Take $\widehat{f}\in \widehat{\sF}_{\widehat{F}_k}:=\{f\in \widehat{\sF}: f=0\text{ on }\widehat I\setminus \widehat{F}_k\}$ for some $k$.  Clearly $\widehat{f}|_{\widehat I\setminus \widehat{F}_n}\equiv 0$ for $n\geq k$.  Since every function in $\widehat{\sF}$ is continuous on $\widehat{I}$,  it follows that $\widehat{f}(\widehat x)=\lim_{k<n\rightarrow \infty}\widehat{f}(\widehat x_n)=0$.  Particularly,  
\[
	\cup_{k\geq 1}\widehat{\sF}_{\widehat{F}_k}\subset \{\widehat{f}\in \widehat{\sF}: \widehat{f}(\widehat x)=0\}.  
\]
The family on the left hand side is $\widehat{\sE}_1$-dense in $\widehat{\sF}$ while the right one is not.  This leads to a contradiction.  As a result it follows from \eqref{eq:76} that for any $\widehat x\in \widehat K$,  there exists $\varepsilon>0$ such that $(\widehat x-\varepsilon, \widehat x+\varepsilon)\cap \widehat I\subset \widehat{F}_n$ for some $n$.  Using the compactness of $\widehat K$,  we can obtain that $\widehat K\subset \widehat{F}_n$ for some $n$.  That completes the proof. 
\end{proof}

Before stating the result,  we prepare some notations and terminologies.  
Let $(\sE^1,\sF^1)$ be a Dirichlet form on $L^2(E_1,\fm_1)$.  Take another measurable space $(E_2,\mathcal{B}(E_2))$ and a measurable map $j: (E_1,\mathcal{B}(E_1))\rightarrow (E_2,\mathcal{B}(E_2))$.  Define $\fm_2:=\fm_1\circ j^{-1}$,  the image measure of $\fm_1$ under $j$.  Then 
\[
	j^*: L^2(E_2,\fm_2)\rightarrow L^2(E_1,\fm_1),\quad f\mapsto j^*f:=f\circ j
	\]
 is an isometry.  Define $\sF^2:=\{f\in L^2(E_2,\fm_2): j^*f\in \sF^1\}$ and 
 \[
 	\sE^2(f,g):=\sE^1(j^*f, j^*g),\quad f,g\in \sF^1. 
 \]
 If $j^*$ maps $L^2(E_2,\fm_2)$ onto $L^2(E_1,\fm_1)$,  then $(\sE^2,\sF^2)$ is a Dirichlet form on $L^2(E_2,\fm_2)$,  which is called the \emph{the image Dirichlet form} of $(\sE^1,\sF^1)$ under $j$.  
 Particularly,  if both $E_1$ and $E_2$ are locally compact separable metric spaces and $j$ is an a.e. homeomorphism,  i.e.  there is an $\fm_1$-negligible set $N_1$ and an $\fm_2$-negligible set $N_2$ such that $j: E_1\setminus N_1\rightarrow E_2\setminus N_2$ is a homeomorphism,   then $j^*$ is surjective.  


\begin{theorem}\label{THM311}
Let $(I',\fm',\sE',\sF')$ be a regular representation of $(I_0,\fm, \sE,\sF)$.  Then there exists a unqiue $\sE'$-polar set $N'\subset I'$ and a unique homeomorphism $j': \widehat{I} \rightarrow I'\setminus N'$ such that $(\sE',\sF')$ is the image Dirichlet form of $(\widehat\sE,\widehat \sF)$ under $j'$.  
\end{theorem}
\begin{proof}
In view of \cite[Theorem~A.4.9]{FOT11},  there exists a regular representation $(\tilde{I}, \tilde{\fm}, \tilde{\sE},\tilde{\sF})$ such that $(I',\fm',\sE',\sF')$ and $(\widehat{I},\widehat{\fm},\widehat{\sE},\widehat{\sF})$ are equivalent to it by isomorphisms $\Phi'$ and $\widehat{\Phi}$, and 
\[
	\Phi'\left(\sF'\cap C_\infty(I') \right)\subset \tilde{\sF}\cap C_\infty(\tilde{I}),\quad \widehat\Phi(\widehat\sF\cap C_\infty(\widehat I) )\subset \tilde{\sF}\cap C_\infty(\tilde{I}).
\]
Applying \cite[Lemma~A.4.8]{FOT11} to $\widehat{\Phi}$ and repeating its proof,  we can obtain a continuous map $\widehat{\gamma}: \tilde{I}\rightarrow \widehat{I}$,  an $\tilde{\sE}$-nest $\{\tilde{F}^1_n: n\geq 1\}$ and an $\widehat{\sE}$-nest $\{\widehat{F}_n\}$ such that $f\circ \widehat{\gamma}\in C_\infty(\tilde{I})$ for any $f\in C_\infty(\widehat{I})$ and
\begin{equation}\label{eq:319}
	\widehat{\gamma}_n:=\widehat{\gamma}|_{\tilde{F}^1_n}: \tilde{F}^1_n\rightarrow \widehat{F}_n,\quad n\geq 1
\end{equation}
are homeomorphisms.  In addition,  $(\widehat{\sE},\widehat{\sF})$ is the image Dirichlet form of $(\tilde{\sE},\tilde{\sF})$ under $\widehat{\gamma}$.  Analogously there is a continuous map $\gamma': \tilde{I}\rightarrow I'$,  an $\tilde{\sE}$-nest $\{\tilde{F}^2_n: n\geq 1\}$ and an $\tilde{\sE}'$-nest $\{F'_n: n\geq 1\}$ such that $f\circ \gamma'\in C_\infty(\tilde{I})$ for any $f\in C_\infty(I')$ and
\begin{equation}\label{eq:320-3}
	\gamma'_n:=\gamma'|_{\tilde{F}^2_n}: \tilde{F}^2_n \rightarrow F'_n,\quad n\geq 1
\end{equation}
are homeomorphisms.  In addition,  $(\sE',\sF')$ is the image Dirichlet form of $(\tilde{\sE},\tilde{\sF})$ under $\gamma'$.  Without loss of generality we may and do assume $\tilde{F}^1_n=\tilde{F}^2_n=:\tilde{F}_n$.  Otherwise we can replace $\tilde{F}^1_n$ and $\widehat{F}_n$ by $\tilde{F}^1_n\cap \tilde{F}^2_n$ and  $\widehat{F}_n \cap \widehat{\gamma}(\tilde{F}^2_n)$ in \eqref{eq:319}.  The maps in \eqref{eq:320-3} can be treated similarly.  On account of Corollary~\ref{COR310}~(2),  we have 
\[
	\widehat{I}=\cup_{n\geq 1}\widehat{F}_n.
\]
Set $I'_0:=\cup_{n\geq 1}F'_n$ and $\tilde{I}_0:=\cup_{n\geq 1}\tilde{F}_n$.  

Secondly,  we assert that for any compact subset $K'$ of $I'$,  it holds that $K'\cap I'_0\subset F'_n$ for some $n\geq 1$.  
To do this,  take $f\in C_\infty(I')$ such that $f=1$ on $K'$,  and set $\tilde{K}:=\gamma'^{-1}(K')$.  Since $\gamma'$ is continuous,  $\tilde{K}$ is closed in $\tilde{I}$.  In addition,  
\[
	\tilde{K}\subset \{\tilde{x}\in \tilde{I}: f\circ \gamma'(\tilde{x})=1\}
\]
and $f\circ \gamma'\in C_\infty(\tilde{I})$ yields that the right hand side is a subset of a compact set in $\tilde{I}$.  Particularly $\tilde{K}$ is compact in $\tilde{I}$.  It follows from the continuity of $\widehat{\gamma}$ that $\widehat{K}:=\widehat{\gamma}(\tilde{K})$ is compact in $\widehat{I}$.  Applying Lemma~\ref{LM310} to $\widehat{K}$ and using homeomorphisms \eqref{eq:319} and \eqref{eq:320-3},  we get that $\widehat{K}\subset \widehat{F}_n$ for some $n$ and thus 
\[
	K'\cap I'_0= \gamma'(\tilde{K}\cap \tilde{I}_0)\subset \gamma'(\widehat \gamma^{-1}(\widehat{K})\cap \tilde{I}_0)\subset \gamma'_n(\widehat{\gamma}^{-1}_n(\widehat{F}_n))=F'_n.  
\]  

Denote by $\widehat{q}$ the inverses of $\widehat{\gamma}|_{\tilde{I}_0}$.  We show that $\widehat{q}$ is continuous on $\widehat{I}$.  Particularly,   
\[
	\widehat{q}: \widehat{I}\rightarrow \tilde{I}_0
\] 
is a homeomorphism and thus $(\tilde{\sE},\tilde{\sF})$ is the image Dirichlet form of $(\widehat{\sE},\widehat{\sF})$ under $\widehat{q}$.  To this end,  take an arbitrary precompact open subset  $\widehat U$ of $\widehat{I}$.  Since $\widehat{U}\subset \widehat{F}_n$ for some $n$,  it follows that $\widehat q|_{\widehat{U}}=\widehat{\gamma}^{-1}_n|_{\widehat{U}}$ is continuous.  Consequently $\widehat{q}$ is continuous on $\widehat{I}$.  


Set $N':=I'\setminus I'_0$ and 
\begin{equation}\label{eq:321-2}
	j': \widehat{I}\rightarrow I'_0,\quad \widehat{x}\mapsto \gamma'(\widehat{q}(\widehat{x})). 
\end{equation}
Then $j'$ is a continuous bijection and its restriction $j'|_{\widehat{F}_n}: \widehat{F}_n\rightarrow F'_n$ is a homoemorphism.  We prove that $j'$ is a local homeomorphism.  Indeed,  let $\widehat{x}\in\widehat{I}$ and $x':=j'(\widehat{x})$.  Since $I'$ is locally compact,  we take a precompact open set $V$ in $I'$ with $x'\in V$ and set $V':=V\cap I'_0$.  Then $V'$ is an open neighbourhood of $x'$ in $I'_0$ and the second step yields that 
\[
	V'\subset \overline{V}\cap I'_0\subset F'_n
\]
for some $n$,  where $\overline{V}$ is the closure of $V$ in $I'$.  Since $j'$ is continuous and $j'|_{\widehat{F}_n}$ is a homoemorphism,  it follows that $\widehat U:=j'^{-1}(V')\subset \widehat{F}_n$ is an open neighbourhood of $\widehat{x}$ in $\widehat{I}$ and $j'|_{\widehat{U}}: \widehat{U}\rightarrow V'$ is a homeomorphism.  Therefore \eqref{eq:321-2} is a local homeomorphism.  Note that $j'$ is also bijective.  Eventually we conclude that $j'$ is a homeomorphism.  

By making use of the facts that both $j'$ and $\widehat{q}$ are homeomorphisms,  one can find that $\gamma'|_{\tilde{I}_0}: \tilde{I}_0\rightarrow I'_0$ is also a homeomorphism and hence easily verify that $(\sE',\sF')$ is the image Dirichlet form of $(\widehat{\sE},\widehat{\sF})$ under $j'$.  

Finally we argue the uniqueness of $(N',j')$.  Take another pair $(N'_1, j'_1)$ with the same properties.  In view of Corollary~\ref{COR310}~(2),  every singleton contained in $\widehat{I}$ is not $\widehat{\sE}$-polar.  Since $j'_1$ is a homeomorphism,  it follows that every singleton contained in $I'\setminus N'_1$ is not $\sE'$-polar.  Consequently $I'\setminus N'_1\subset I'\setminus N'$ because $N'$ is $\sE'$-polar.  The contrary $I'\setminus N'\subset I'\setminus N'_1$ also holds true by a similar argument.  Therefore $N'=N'_1$.  The identity $j'=j'_1$ is obvious by using the fact that $j'(\widehat{X})$ and $j'_1(\widehat{X})$ are the identical Markov process associated with $(\sE',\sF')$.  That completes the proof.  
\end{proof}

The following corollary is immediate from this theorem.

\begin{corollary}
Let $(I_i,\fm_i, \sE^i,\sF^i)$ be regular representations of $(I_0,\fm,  \sE,\sF)$ for $i=1,2$.  Then there are $\sE^i$-polar sets $N_i\subset I_i$ for $i=1,2$ and a homeomorphism $j: I_1\setminus N_1\rightarrow I_2\setminus N_2$ such that $(\sE^2,\sF^2)$ is the image Dirichlet form of $(\sE^1,\sF^1)$ under $j$.  
\end{corollary}

Denote by $X'$ the $\fm'$-symmetric Hunt process associated with $(\sE',\sF')$.  We call $N'$ obtained in Theorem~\ref{THM311} the \emph{essentially exceptional set} of $(I',\fm',\sE',\sF')$,  and $\bs':=j'^{-1}$,  the inverse of $j'$,  the \emph{scale function} of  $X'$ or $(\sE',\sF')$.  Particularly,  the essentially exceptional set of $(\widehat{I},\widehat{\fm},\widehat{\sE},\widehat{\sF})$ (resp.  $(I^*,\fm^*,\sE^*,\sF^*)$) is empty,   and the scale function of $\widehat{X}$ (resp.  $X^*$) is $\widehat{\bs}$ (resp.  $\bs^*$).  We wish to state emphatically that although $(\sE',\sF')$ is also a regular Dirichlet form on $L^2(I'\setminus N',  \fm')$,  the essentially exceptional set $N'$ of $(I',\fm', \sE',\sF')$ is not necessarily empty.  

\begin{example}
Consider that $I=[0,1]$,  $\fm$ is the Lebesgue measure on $[0,1]$ and $\bs(x)$ is continuous and strictly increasing on $I$ such that $\bs(0)=0$ and $\bs(1)=\infty$.  Then $I_0=[0,1)$ and $(I_0,\fm, \sE,\sF)$ corresponds to the regular diffusion on $I_0$ with scale function $\bs$,  speed measure $\fm$ and no killing inside.  

Making use of \cite[Theorem~2.1]{LY19},  one can get that $(\sE,\sF)$ is regular on not only $L^2(I_0,\fm)$ but also $L^2(I,\fm)$.  Particularly,  $(I,\fm,\sE,\sF)$ is a regular representation of $(I_0,\fm,  \sE,\sF)$,  and its essentially exceptional set is $\{1\}$.  
\end{example}

\section{Skip-free Hunt processes in one dimension}\label{SEC5}

In this section we will study skip-free Hunt processes in one dimension that may be of independent interest.  Unless otherwise specified we would not adopt the notations in the previous sections.

\subsection{Skip-free Hunt processes}\label{SEC51}

Let $E\in \overline{\mathscr K}$ be a nearly closed subset of $\overline{\bR}$,  i.e.  $\overline E:= E\cup \{l,r\}$ is a closed subset of $\overline{\bR}$ where $l=\inf\{x: x\in E\}$ and $r=\sup\{x: x\in E\}$.   Write $[l,r]\setminus \overline{E}$ as a disjoint union of open intervals:
\begin{equation}\label{eq:51}
	[l,r]\setminus \overline{E}=\cup_{k\geq 1}(a_k,b_k).  
\end{equation}
We add a ceremony $\partial$ to $E$ and define $E_\partial:=E\cup \{\partial\}$.  More precisely,  $\partial$ is an additional isolated point when $E=\overline{E}$.  When $E=\overline E\setminus \{l\}$ or $E=\overline{E}\setminus \{r\}$,  $\partial$ is identified with $l$ or $r$.  When $E=\overline{E}\setminus \{l,r\}$,  $E_\partial$ is the one-point compactification of $E$. 
Further let 
\[
X=\left\{\Omega,  \mathscr F_t,  \theta_t,  X_t,  \mathbf{P}_x,  \zeta\right\}
\]
be a \emph{Hunt process} on $E_\partial$,  where $\{\mathscr F_t\}_{t\in [0,\infty]}$ is the minimum augmented admissible filtration and $\zeta=\inf\{t>0: X_t=\partial\}$ is the lifetime of $X$.  
The other notations and terminologies are standard and we refer readers to,  e.g.,  \cite[Appendix~A]{CF12}. 

\begin{definition}\label{DEF51}
Let $X$ be a Hunt process on $E_\partial$.  Then $X$ is called a \emph{skip-free Hunt process} if the following are satisfied:
\begin{itemize}
\item[(SF)] \emph{Skip-free property}: $(X_{t-} \wedge X_t,  X_{t-}\vee X_t)\cap E=\emptyset$ for any $t<\zeta$,  $\mathbf{P}_x$-a.s.  and all $x\in E$.   
\item[(SR)] \emph{Regular property}: $\mathbf{P}_x(T_y<\infty)>0$ for any $x, y\in E$,  where $T_y:=\inf\{t>0: X_t=y\}$ ($\inf \emptyset:=\infty$).   
\item[(SK)] There is \emph{no killing inside} in the sense that if $\mathbf{P}_x(\zeta<\infty)>0$ for $x\in E$,  then $l$ or $r$ does not belong to $E$ and $\mathbf{P}_x(X_{\zeta-}\notin E,  \zeta<\infty)=\mathbf{P}_x(\zeta<\infty)$.  
\end{itemize}
\end{definition}
\begin{remark}
\begin{itemize}\label{RM53}
\item[(1)] When $\overline{E}$ is a closed interval,  (SF) is identified with the continuity of all sample paths and $X$ is a regular diffusion with no killing inside;  see,  e.g.,  \cite[Chapter VII,  \S3]{RY99}. 
\item[(2)] (SF) particularly implies the following: Given states $x,y,z\in E$ with $x<y<z$ (or $x>y>z$),  and times $r<s$,  if $X_r=x$ and $X_s=z$,  then there is $t\in (r,s)$ such that $X_t=y$ or $X_{t-}=y$.   
To see this,  set $t:=\inf\{u>r: X_u>y\}$.  Then $r<t\leq s$,  $X_t\geq y$ and $X_{t-}\leq y$.  If $X_t>y$ and $X_{t-}<y$,  then $y\in (X_{t-}, X_t)$ as violates (SF).   
\item[(3)] (SR) is slightly different from the regular property for a diffusion in \cite[Chapter VII,  \S3]{RY99} (see also \cite[Chapter V,  (45.2)]{RW87}).    Here we impose $\mathbf{P}_x(T_y<\infty)>0$ for all $x\in E$,  while in \cite{RY99} this condition is only assumed for $x\in E\setminus \{l,r\}$.  This stronger assumption bars the possibility that $l$ or $r$ that contained in $E$ becomes an absorbing point.  See Corollary~\ref{COR55} for further discussion. 
\item[(4)] The process $X$ is called \emph{conservative} if $\mathbf{P}_x(\zeta=\infty)=1$ for all $x\in E$.  Note that if $l,r\in E$,  then $X$ is conservative due to (SK).  This definition is equivalent to that $\mathbf{P}_x(\zeta=\infty)=1$ holds for one $x\in E$.  To see this,  consider the case $l\notin E,  r\in E$ and the other cases can be argued analogously.  Lemma~\ref{LM54}~(4) tells us that $\zeta=T_l=\lim_{n\rightarrow \infty} T_{a_n}$ with $a_n\downarrow l$.  Take $x,y\in E$ with $a_n<x<y$.  It follows from Lemma~\ref{LM54}~(1) and the strong Markov property of $X$ that
\[
	\mathbf{P}_y(T_{a_n}<\infty)=\mathbf{P}_y(T_x<\infty)\mathbf{P}_x(T_{a_n}<\infty).  
\]
Letting $n\uparrow \infty$ we get that $\mathbf{P}_y(\zeta<\infty)=\mathbf{P}_y(T_x<\infty)\mathbf{P}_x(\zeta<\infty)$.  On account of (SR),  $\mathbf{P}_y(\zeta<\infty)>0$ is equivalent to $\mathbf{P}_x(\zeta<\infty)>0$.  Particularly,  $X$ is conservative if $\mathbf{P}_x(\zeta=\infty)=1$ holds for one $x\in E$.  
\item[(5)] When $l\notin E$ or $r\notin E$,  $X_{\zeta-}=\partial$ if $\zeta<\infty$.   In abuse of notations we will also write $X_{\zeta-}=l$ or $r$ to stand for that the convergence $\lim_{t\uparrow \zeta} X_t=\partial$ is along the decreasing or increasing direction.  Loosely speaking,  $l$ or $r$ or both of them are viewed as the ceremony.   
\end{itemize}

\end{remark}

From now on we always fix a skip-free Hunt process $X$ on $E_\partial$.    When there is no risk of ambiguity we would omit the subscript of $E_\partial$ and write $E$ for the state space of $X$.  For $a\in E$, define
\[
	T_{>a}:=\inf\{t>0: X_t>a\},\quad T_{<a}:=\inf\{t>0:X_t<a\}
\]
and 
\[
	T_{\geq a}:=\inf\{t>0: X_t\geq a\},\quad T_{\leq a}:=\inf\{t>0:X_t\leq a\}. 
\]
Set $T_{\neq a}:=\inf\{t>0: X_t\neq a\}=T_{>a}\wedge T_{<a}$.  
The lemma below states some useful facts concerning $X$. 

\begin{lemma}\label{LM54}
\begin{itemize}
\item[\rm (1)] For $x\in E$,
\[
	\mathbf{P}_x\left(\bigcup_{y,z\in E \text{ s.t. }x<y<z\text{ or }z<y<x} \{T_z< T_y\} \right)=0.  
\]
\item[\rm (2)] For $x,y \in E$ with $x>y$ (resp. $x<y$),  
\[
\mathbf{P}_x(T_{\leq y}=T_y)=1\quad  (\text{resp.  } \mathbf{P}_x(T_{\geq y}=T_y)=1). 
\]   
\item[\rm (3)] Let $x<y_n\uparrow y$ or $x>y_n\downarrow y $ with $x,y_n,y\in E$.  Then $$\mathbf{P}_x(\lim_{n\rightarrow \infty}T_{y_n}=T_y)=1. $$  
\item[\rm (4)] When $j\notin E$ for $j=l$ or $r$,  it holds $\mathbf{P}_x$-a.s.  for any $x\in E$ that 
\[
	T_j:=\lim_{n\rightarrow \infty}T_{a_n}=\left\lbrace
	\begin{aligned}
		&\zeta,\quad \text{if }\zeta<\infty,  X_{\zeta-}=j,\\
		&\infty,\quad \text{otherwise},
	\end{aligned}
	\right.
\]
where $a_n\rightarrow j$ is taken to be a monotone sequence contained in $E$.  Particularly,  
\[
	\zeta=T_l,\quad  (\text{resp. }  \zeta=T_r;\;  \zeta=T_l\wedge T_r)
\]
if $l\notin E, r\in E$ (resp.  $l\in E, r\notin E$; $l,r\notin E$).   
\item[\rm (5)] It holds that
\begin{equation}\label{eq:52}
	\mathbf{P}_x(T_{>x}=0)=\left\lbrace \begin{aligned}
		1,\quad x\neq a_k, \\
		0,\quad x=a_k,
	\end{aligned} \right.  \quad x\in E\setminus \{r\},
\end{equation}
and 
\begin{equation}\label{eq:53}
	\mathbf{P}_x(T_{<x}=0)=\left\lbrace \begin{aligned}
		1,\quad x\neq b_k, \\
		0,\quad x=b_k,
	\end{aligned} \right.  \quad x\in E\setminus \{l\},
\end{equation}
where $a_k,b_k$ are given in \eqref{eq:51}.  
\item[(6)]  For any $x\in E$,  $\mathbf{P}_x(T_x=0)=1$.  
\end{itemize}
\end{lemma}
\begin{proof}
\begin{itemize}
\item[(1)] Let $\omega\in \Omega$ such that $T_z(\omega)<T_y(\omega)$ for $x,y,z\in E$ with  $x<y<z$.  We show that $(X_{t-}(\omega) \wedge X_t(\omega),  X_{t-}(\omega)\vee X_t(\omega))\cap E\neq  \emptyset$ for some $t$,  so that (SF) leads to the assertion.  To do this,  note that
\[
	T_{>y}\leq T_z<T_y\leq \hat{T}_y=\inf\{t>0:X_{t-}=y\},
\]
where $T_y\leq \hat{T}_y$ is due to,  e.g.,  \cite[Theorem~A.2.3]{FOT11}.   Set $t:=T_{>y}(\omega)$.  By means of the c\`adl\`ag property of all paths,  we have 
\[
X_t(\omega)\geq y,  \quad X_{t-}(\omega)\leq y,
\]
while $t<T_y(\omega)\leq \hat{T}_y(\omega)$.  Hence $X_t(\omega)>y$ and $X_{t-}(\omega)<y$.  Therefore $y\in (X_{t-}(\omega), X_t(\omega))\cap E$.  
\item[(2)] We only treat the case $x>y$.  Note that $T_{\leq y}\leq T_y$.  Take $\omega\in \Omega$ such that $T_{\leq y}(\omega)<T_y(\omega)$.   It suffices to show that $y \in (X_{t-}(\omega) \wedge X_t(\omega),  X_{t-}(\omega)\vee X_t(\omega))$ for some $t$.  In fact,  set 
\[
	t:=T_{\leq y}(\omega)<T_y(\omega)\leq \hat{T}_y(\omega).  
\]
Mimicking the argument in the first assertion,  we get that $X_t(\omega)< y$ and $X_{t-}(\omega)>y$.  Hence $y\in (X_t(\omega), X_{t-}(\omega))$.  
\item[(3)] We only consider the case $x<y_n\uparrow y$.  In view of the first assertion,  $T_{y_n}$ is increasing and $T_{y_n}\leq T_y$,  $\mathbf{P}_x$-a.s.  Hence  $T:=\lim_{n\rightarrow \infty}T_{y_n}\leq T_y$.  The quasi-left-continuity of $X$ implies that
\[
	X_T=X_{\lim_{n\rightarrow \infty} T_{y_n}}=\lim_{n\rightarrow\infty}  X_{T_{y_n}}=y,\quad \mathbf{P}_x\text{-a.s. on }T<\infty. 
\]
Thus $T\geq T_y$ and we eventually obtain that $T=T_y$,  $\mathbf{P}_x$-a.s.   
\item[(4)] We only treat the case $j=l\notin E$.  Fix $x\in E$ and assume without loss of generality that $a_1<x$.  It follows from the quasi-left-continuity of $X$ that
\[
	X_{T_l}=\lim_{n\rightarrow \infty} X_{T_{a_n}}=l,\quad \mathbf{P}_x\text{-a.s. on }T_l<\infty.  
\]
Hence $T_l\geq \zeta$,  $\mathbf{P}_x$-a.s.  As a result,  $T_l<\infty$ implies $T_l=\zeta$ and $X_{\zeta-}=X_{T_l-}=l$, because $T_{a_n}\leq T_l<\infty$ leads to $T_{a_n}<\zeta$.  To the contrary suppose $\zeta<\infty$ and $X_{\zeta-}=l$.  It suffices to show $T_{a_n}<\infty$, so that 
$T_{a_n}\leq \zeta$ and thus $T_l\leq \zeta$ leading to $T_l=\zeta$ by means of $T_l\geq \zeta$.  To accomplish this, note that $X_0=x>a_n$ and since $\lim_{s\uparrow \zeta}X_s=l$,  there exists $s<\zeta$ such that $X_s<a_n$.  Remark~\ref{RM53}~(2) indicates that there exists $t\in (0,s)$ such that $X_t=a_n$ or $X_{t-}=a_n$.  For the former case it holds that $T_{a_n}\leq t$.  For the latter case,  $T_{a_n}\leq \hat{T}_{a_n}\leq t$.  Hence $T_{a_n}<\infty$ is concluded.  
\item[(5)] We first prove \eqref{eq:52} and then \eqref{eq:53} can be argued similarly.  When $x=a_k$,   the second assertion implies that
\[
	T_{>x}=\inf\{t>0: X_t\geq b_k\}=T_{b_k}.  
\]
The right continuity of all paths leads to $T_{b_k}>0$,  $\mathbf{P}_{x}$-a.s.  Hence $\mathbf{P}_{x}(T_{>x}=0)=0$.  Next consider $x\neq a_k$.  We show that $X_{T_{>x}}=x$,  $\mathbf{P}_x$-a.s. on $\{T_{>x}<\infty\}$.  In fact, it follows from the right continuity of all paths that $X_{T_{>x}}\geq x$.  To the contrary,  take $E\ni x_n \downarrow x$.  Suppose $\omega\in \Omega$ such that $T_{>x}(\omega)<\infty$ and $X_{T_{>x}}(\omega)\geq x_n$.  
Set $t:=T_{>x}(\omega)$.  Then $X_{t-}(\omega)\leq x$ and $X_t(\omega)\geq x_n$.  Hence $x_{n+1}\in (X_{t-}(\omega), X_t(\omega))$.  (SF) implies that $X_{T_{>x}}< x_n$,  $\mathbf{P}_x$-a.s.  on $\{T_{>x}<\infty\}$.  Letting $n\uparrow \infty$,  we obtain that $X_{T_{>x}}\leq x$ and eventually $X_{T_{>x}}= x$ is concluded.  With this fact at hand,  we argue $\mathbf{P}_x(T_{>x}=0)=1$ by contradiction.  On account of Blumenthal 0-1 law,  $\mathbf{P}_x(T_{>x}=0)$ equals $0$ or $1$.  Suppose that it is equal to $0$.  Set
\[
	\Lambda=\{\omega\in \Omega: \exists \varepsilon>0\text{ s.t. }X_t(\omega)\leq x \text{ for all }t\leq \varepsilon\}=\{T_{>x}>0\}.  
\]
By the right continuity of all paths and the definition of $T_{>x}$,  it holds that
\[
\begin{aligned}
	0&=\mathbf{P}_x\left(T_{>x}<\infty,  \exists \varepsilon>0\text{ s.t. }X_{T_{>x}+t}\leq x\text{ for }0\leq t\leq \varepsilon \right) \\
	&=\mathbf{E}_x\left(1_\Lambda\circ \theta_{T_{>x}}; T_{>x}<\infty\right).  
\end{aligned}\]
Using strong Markov property, $X_{T_{>x}}=x$ on $\{T_{>x}<\infty\}$ and $\mathbf{P}_x(\Lambda)=1$,  we get that
\begin{equation}\label{eq:54}
	0=\mathbf{E}_x\left(\mathbf{P}_{X_{T_{>x}}}(\Lambda); T_{>x}<\infty\right)=\mathbf{P}_x(T_{>x}<\infty).  
\end{equation}
However $T_{>x}<T_{x_n}$ and (SR) implies that $\mathbf{P}_x(T_{x_n}<\infty)>0$.  Thus $\mathbf{P}_x(T_{>x}<\infty)>0$,  which contradicts with \eqref{eq:54}.  Therefore \eqref{eq:52} is concluded.  
\item[(6)] We first consider $x\in E\setminus \{l,r, a_k,b_k: k\geq 1\}$. Since  $\mathbf{P}_x( T_{<x}=0)=\mathbf{P}_x(T_{>x}=0)=1$,  it follows that for $\mathbf{P}_x$-a.s.  $\omega\in \Omega$,  there exist two sequences $t_n\downarrow 0$ and $s_n\downarrow 0$ such that
\[
	X_{t_n}(\omega)< x,\quad X_{s_n}(\omega)> x. 
\]
It suffices to show that for each $n\geq 1$,  there exists $t< t_n$ such that $X_t(\omega)=x$.  Assume without loss of generality that $s_n<t_n$.  
In view of Remark~\ref{RM53}~(2),  there exists $t'\in (s_n,t_n)$ such that $X_{t'}(\omega)=x$ or $X_{t'-}(\omega)=x$.  For the former case take $t:=t'$.  For the latter case note that $T_x(\omega)\leq \hat{T}_x(\omega) \leq t'$.   Hence there exists $t\leq t'<t_n$ such that  $X_t(\omega)=x$. 

Next we prove the assertion for the case $x=a_k\in E$ and the other cases can be treated analogically.  If $x$ is isolated in $E$,  then $\mathbf{P}_x(T_x=0)=1$ is the consequence of Corollary~\ref{COR55}.  Assume that $x_0<x_1<\cdots x_n\uparrow x$ with $x_n\in E$ for $n\geq 0$.  Argue by contradiction and suppose $\mathbf{P}_x(T_x>0)=1$.  Note that (SR) implies that $0<c:=\mathbf{P}_{x_0}(T_x<\infty)\leq 1$.  In view of 
\[
	1=\mathbf{P}_x(T_x>0)=\mathbf{P}_x(\lim_{\varepsilon\downarrow 0} \{T_x>\varepsilon\})=\uparrow \lim_{\varepsilon\downarrow 0} \mathbf{P}_x(T_x>\varepsilon),
\]
one can obtain $\varepsilon_0>0$ such that 
\begin{equation}\label{eq:45}
\mathbf{P}_x(T_x>\varepsilon_0)>1-c/3.
\end{equation}  
Set
\[
	\Lambda:=\{T_x>0\}=\{\omega\in \Omega: \exists \varepsilon>0,\text{s.t. } X_t<x,  0<t<\varepsilon\},\quad \mathbf{P}_x\text{-a.s.}
\]
Since $T_{x_n}=T_{\leq x_n}$,  $\mathbf{P}_x$-a.s.  due to the second assertion,  it follows that $T_{x_n}$ is decreasing, $\mathbf{P}_x$-a.s., and $\lim_{n\rightarrow \infty} T_{x_n}(\omega)=0$ for any $\omega\in \Lambda$.  These yield that 
\[
	1=\mathbf{P}_x(\lim_{n\rightarrow\infty}T_{x_n}<\varepsilon_0/2)=\uparrow \lim_{n\rightarrow \infty}\mathbf{P}_x(T_{x_n}<\varepsilon_0/2).  
\]
Hence there exists $N\in \bN$ such that for all $n>N$,  
\begin{equation}\label{eq:46}
\mathbf{P}_x(T_{x_n}<\varepsilon_0/2)>1-c/3.
\end{equation}
On account of \eqref{eq:45} and \eqref{eq:46},  we have
\[
	\mathbf{P}_x(T_{x_n}<\varepsilon_0/2,  T_x>\varepsilon_0)\geq \mathbf{P}_x(T_{x_n}<\varepsilon_0/2)-\mathbf{P}_x(T_x\leq \varepsilon_0)>1-\frac{2c}{3}.  
\]
On the event $\{T_{x_n}<\varepsilon_0/2,  T_x>\varepsilon_0\}$,  $T_x>T_{x_n}$ and thus $T_x=T_{x_n}+T_x\circ \theta_{T_{x_n}}$.  Using the strong Markov property,  one gets that
\[
\begin{aligned}
\mathbf{P}_x(T_{x_n}<\varepsilon_0/2,  T_x>\varepsilon_0)&\leq \mathbf{P}_x(T_{x_n}<\varepsilon_0/2,  T_x\circ \theta_{T_{x_n}}>\varepsilon_0/2) \\
&=\mathbf{P}_x(T_{x_n}<\varepsilon_0/2)\mathbf{P}_{x_n}(T_x>\varepsilon_0/2).  
\end{aligned}\]
Hence for all $n>N$,  
\begin{equation}\label{eq:47}
	\mathbf{P}_{x_n}(T_x>\varepsilon_0/2)\geq \mathbf{P}_x(T_{x_n}<\varepsilon_0/2,  T_x>\varepsilon_0)>1-\frac{2c}{3}.  
\end{equation}
On the other hand,  the third assertion indicates that
\begin{equation}\label{eq:48}
	c=\mathbf{P}_{x_0}(\uparrow \lim_{n\rightarrow \infty}T_{x_n}=T_x,  T_x<\infty)=\lim_{n\rightarrow \infty} \mathbf{P}_{x_0}(T_{x_n}>T_x-\varepsilon_0/2,  T_x<\infty).  
\end{equation}
Since $T_{x_n}\leq T_x$,  $\mathbf{P}_{x_0}$-a.s.,  it follows from the strong Markov property and \eqref{eq:47} that for all $n>N$,  
\[
\begin{aligned}
 \mathbf{P}_{x_0}(T_{x_n}>T_x-\varepsilon_0/2,  T_x<\infty)&=\mathbf{P}_{x_0}(T_{x_n}>T_x-\varepsilon_0/2,  T_{x_n}<\infty) \\
 &=\mathbf{P}_{x_0}(T_x\circ \theta_{T_{x_n}}<\varepsilon_0/2,  T_{x_n}<\infty)  \\
 &=\mathbf{P}_{x_0}(T_{x_n}<\infty) \mathbf{P}_{x_n}(T_x<\varepsilon_0/2)<\frac{2c}{3},
\end{aligned}\]
as leads to a contradiction with \eqref{eq:48}.  
\end{itemize}
That completes the proof.  
\end{proof}

Due to the strong Markov property of $X$,  there is a constant $\kappa(x)\in [0,\infty]$ depending on $x\in E$ such that 
\[
\mathbf{P}_x(T_{\neq x}>t)=e^{-\kappa(x) t};
\]
see,  e.g.,  \cite[Proposition~2.19]{RY99}.  	(SR) bars the possibility of $\kappa(x)=0$.  If $\kappa(x)=\infty$,  then the process leaves $x$ at once.   When $\kappa(x)\in (0,\infty)$,  $x$ is called a \emph{holding point} because the process stays at $x$ for an exponential holding time.  

\begin{corollary}\label{COR55}
The family of holding points consists of all isolated points in $E$.  
\end{corollary}
\begin{proof}
Let $x\in E$ be not isolated.  Then \eqref{eq:52} and \eqref{eq:53} imply that $\mathbf{P}_x(T_{>x}=0)=1$ or $\mathbf{P}_x(T_{<x}=0)=1$.  Since $T_{\neq x}=T_{>x}\wedge T_{<x}$,  it follows that $\mathbf{P}_x(T_{\neq x}=0)=1$.  Hence $x$ is not a holding point.  Now consider an isolated point $x\in E$.  It has three possibilities: $x=l=a_p$,  $x=r=b_q$ or $x=a_p=b_q$ for some $p,q\geq 1$.  When $x=l=a_p$,  we have
\[
	T_{\neq l}=T_{\geq b_p}=T_{b_p}.  
\]
Thus $\mathbf{P}_l(T_{\neq l}=0)=\mathbf{P}_l(T_{b_p}=0)=0$ and $l$ is a holding point.  Another case $x=r=b_q$ can be argued similarly.  When $x=a_p=b_q$,  it suffices to note that  $T_{\neq x}=T_{\leq a_q}\wedge T_{\geq b_p}=T_{a_q}\wedge T_{b_p}$,  $\mathbf{P}_x$-a.s.,  and hence 
\[
\mathbf{P}_x(T_{\neq x}=0)\leq \mathbf{P}_x(T_{a_q}=0)+\mathbf{P}_x(T_{b_p}=0)=0.  
\] That completes the proof.   
\end{proof}

\subsection{Scale function of skip-free Hunt process}

Let $a,x,b\in E$ with $a<x<b$.  Set $H:=(a,b)\cap E$ and $\tau_H:=\inf\{t>0: X_t\notin H\}$.  In view of Lemma~\ref{LM54}~(2),
\[
\tau_H=T_a\wedge T_b,\quad \mathbf{P}_x\text{-a.s.}
\]
Repeating the argument in \cite[Proposition~3.1]{RY99},  we can obtain that $\mathbf{E}_x\tau_H<\infty$ and thus $\tau_H<\infty$,  $\mathbf{P}_x$-a.s.  
Particularly,  $\mathbf{P}_x(T_a<T_b)+\mathbf{P}_x(T_b<T_a)=1$.  The following result extends \cite[Proposition~3.2]{RY99}. 

\begin{theorem}\label{THM58}
There exists a continuous and strictly increasing real valued function $\bs$ on $E$ such that for any $a,x,b\in E$ with $a<x<b$,  
\begin{equation}\label{eq:55}
\mathbf{P}_x(T_b<T_a)=\frac{\bs(x)-\bs(a)}{\bs(b)-\bs(a)}. 
\end{equation}
If $\tilde{\bs}$ is another function with the same properties,  then $\tilde{\bs}=\alpha \bs +\beta$ with $\alpha>0$ and $\beta\in \bR$.  
\end{theorem}
\begin{proof}
We first consider the case $l,r\in E$.  Using the strong Markov property and repeating the argument in the first paragraph of the proof of \cite[Proposition~3.2]{RY99},  we can conclude that
\[
\mathbf{P}_x(T_r<T_l)=\mathbf{P}_x(T_a<T_b)\mathbf{P}_a(T_r<T_l)+\mathbf{P}_x(T_b<T_a)\mathbf{P}_b(T_r<T_l).  
\]
Setting $\bs(x)=\mathbf{P}_x(T_r<T_l)$ we get the formula \eqref{eq:55}.  To show $\bs$ is strictly increasing,  take $l\leq x<y\leq r$ and note that 
\begin{equation}\label{eq:56}
	\bs(x)=\mathbf{P}_x(T_r<T_l)=\mathbf{P}_x(T_y<T_l,  T_r\circ \theta_{T_y}<T_l\circ \theta_{T_y})=\mathbf{P}_x(T_y<T_l) \bs(y).  
\end{equation}
Hence $\bs(x)\leq \bs(y)$.  We assert that $\bs(y)>\bs(l)=0$.  Argue by contradiction and suppose $\bs(y)=0$,  i.e.  $\mathbf{P}_y(T_r<T_l)=0$.  Set $\sigma_0:=0$,  and for $n\geq 0$, 
\[
\tau_{n+1}:=\inf\{t>\sigma_n: X_t=y\},\quad \sigma_{n+1}:=\inf\{t>\tau_{n+1}: X_t=l  \text{ or }r\}.  
\]	 
Then $A_n:=\{X_{\sigma_{n}}=r,  \sigma_{n}<\infty\}=\{\tau_n<\infty,  T_r\circ \theta_{\tau_n}<T_l\circ \theta_{\tau_n}\}$ and 
\[
	0<\mathbf{P}_l(T_r<\infty)=\mathbf{P}_l(\cup_{n\geq 1}A_n)\leq \sum_{n\geq 1}\mathbf{P}_l(A_n),
\]
where the first inequality is due to (SR).  However the strong Markov property yields that 
\[
	\mathbf{P}_l(A_n)=\mathbf{P}_l(\tau_n<\infty)\mathbf{P}_y(T_r<T_l)=0,
\]
as leads to a contradiction.  Hence $\bs(y)>0$ for any $y>l$.  Now take $l<x<y$ in \eqref{eq:56} and clearly,  $\bs(x)=\bs(y)$ amounts to $\mathbf{P}_x(T_y<T_l)=1$.  We can find that this is impossible by mimicking the argument proving $\bs(y)>\bs(l)$.  As a result,  $\bs$ is strictly increasing on $E$.  Let us turn to prove that $\bs$ is continuous on $E$.  
To accomplish this,  fix $a\in E$ and let $E\ni a_n\downarrow a$.  Take $x,b\in E$ and $a_1<x<b$.  In view of Lemma~\ref{LM54}~(3) anf \eqref{eq:55},
\[
	\lim_{n\rightarrow \infty}\frac{\bs(x)-\bs(a_n)}{\bs(b)-\bs(a_n)}=\lim_{n\rightarrow \infty}\mathbf{P}_x(T_b<T_{a_n})=\mathbf{P}_x(T_b<T_a)=\frac{\bs(x)-\bs(a)}{\bs(b)-\bs(a)}.
\]
Hence $\lim_{n\rightarrow \infty}\bs(a_n)=\bs(a)$.  The left continuity can be argued analogously.  Therefore the continuity of $\bs$ on $E$ is eventually concluded.  

For the general case,  take $l_1,l_2,r_1,r_2\in E$ with $l_2<l_1<r_1<r_2$.  Set
\[
	\bs_1(x):=\mathbf{P}_x(T_{r_1}<T_{l_1}),\quad x\in (l_1,r_1)
\]
and 
\[
	\bs_2(x):=\mathbf{P}_x(T_{r_2}<T_{l_2}),\quad x\in (l_2,r_2).
\]
Mimicking the argument treating the case $l,r\in E$,  we can conclude that both $\bs_1$ and $\bs_2$ are continuous, strictly increasing and solving \eqref{eq:55} for $l_1\leq a<x<b\leq r_1$.  In addition,  it follows from the strong Markov property that for $x\in (l_1,r_1)$, 
\[
\begin{aligned}
\bs_2(x)&=\mathbf{P}_x(T_{l_1}<T_{r_1})\mathbf{P}_{l_1}(T_{r_2}<T_{l_2})+\mathbf{P}_x(T_{r_1}<T_{l_1})\mathbf{P}_{r_1}(T_{r_2}<T_{l_2}) \\
&=(c_{r_1}-c_{l_1})\bs_1(x)+c_{l_1},
\end{aligned}\]
where $c_{l_1}:=\mathbf{P}_{l_1}(T_{r_2}<T_{l_2})$ and $c_{r_1}:=\mathbf{P}_{r_1}(T_{r_2}<T_{l_2})$.  Then by a standard argument,  one can obtain a continuous and strictly increasing function $\bs$,  which is unique up to an affine transformation,  on $E$ such that \eqref{eq:55} is satisfied.  That completes the proof.  
\end{proof}
\begin{remark}\label{RM48}
Although defined up to an affine transformation,  the function $\bs$ of the preceding theorem is called the \emph{scale function} of $X$.  If $\bs$ may be taken equal to $x$ on $E$,  $X$ is said to be on its \emph{natural scale}.  Clearly,  $\tilde{X}=(\tilde{X}_t)_{t\geq 0}:=(\bs(X_t))_{t\geq 0}$ is a  skip-free Hunt process on $\tilde{E}=\bs(E):=\{\bs(x):x\in E\}$ and is on its natural scale.  It is worth pointing out that if $l\in E$ or $r\in E$,  then $\bs(l)\in \tilde{E}$ or $\bs(r)\in \tilde{E}$ is finite.  Particularly,  $\tilde{E}$ is a nearly closed subset of $\bR$ (not of $\overline{\bR}$),  i.e.  $\tilde{E}\in \mathscr K$. 
\end{remark}

We end this subsection with two corollaries concerning the scale function.  The first corollary shows that if $l\notin E$ is approachable in finite time,  $\bs(l):=\lim_{x\downarrow l}\bs(x)$ is finite.  The analogical assertion for $r$ also holds.  

\begin{corollary}\label{COR59}
If $l\notin E$ and $\mathbf{P}_x(X_{\zeta-}=l, \zeta<\infty)>0$ for some $x\in E$,  then $\bs(l)>-\infty$.  
\end{corollary}
\begin{proof}
Argue by contradiction and suppose $\bs(l)=-\infty$.  Letting $a\downarrow l$ in \eqref{eq:55},  we have $\mathbf{P}_x(T_b<T_l)=1$.  Set
\[
	S_0:=0,\quad T_{n+1}:=\inf\{t>S_n:X_t=b\},\quad S_{n+1}:=\{t>T_{n+1}: X_t=x\},\quad n\geq 0.  
\]
Note that 
\[
\{X_{\zeta-}=l, \zeta<\infty\}\subset \cup_{n\geq 0} \{S_n<\infty,  T_l<T_{n+1}\}.  
\]
Since $T_{n+1}=S_n+T_b\circ \theta_{S_n}$ and $T_l=S_n+T_l\circ \theta_{S_n}$ on $\{S_n<\infty\}$,  it follows from the strong Markov property that
\[
	\mathbf{P}_x(S_n<\infty,  T_l<T_{n+1})=\mathbf{P}_x(S_n<\infty)\mathbf{P}_x(T_l<T_b)=0.
\]
This yields $\mathbf{P}_x(X_{\zeta-}=l,\zeta<\infty)=0$,  as violates the condition.  That completes the proof. 
\end{proof}

Another corollary extends \cite[Proposition~3.5]{RY99}.  Recall that $T_j=\inf\{t>0:X_t=j\}$ for $j=l$ or $r$ in $E$.  When $j\notin E$,  we have setted that $T_j=\lim_{a\rightarrow j}T_a$.  
 Let $T:=T_l\wedge T_r$ and define the stopped process
\[
	X^T_t:=\left\lbrace
	\begin{aligned}
	& X_t,\quad t<T,   \\
	 & X_T, \quad t\geq T,
	 \end{aligned}\right. 
\]
where $X_T:=l$ (resp.  $r$) if $T=T_l<\infty$ (resp.  $T=T_r<\infty$).  The result below states that $\bs(X^T)$ is a (not necessarily continuous) local martingale.  Particularly,  the stopped process for $X$ on its natural scale is always a local martingale. 

\begin{corollary}
For any $x\in E$,  $\tilde{X}^T=(\bs(X^T_t))_{t\geq 0}$ is an $\{\sF_t,\mathbf{P}_x\}$-local martingale.  
\end{corollary}
\begin{proof}
Consider the case $l,r\in E$.  We may take $\bs(x)=\mathbf{P}_x(T_r<T_l)$.  We assert that $\tilde{X}^T_t=\mathbf{P}_x\left(T_r<T_l |\sF_t\right)$,  so that $\tilde{X}^T$ is a Doob martingale.  This trivially holds for $x=l$ or $r$.  It suffices to prove it for $x\in E\setminus \{l,r\}$.  The strong Markov property yields that
\begin{equation}\label{eq:57}
	\tilde{X}^T_t=\mathbf{P}_{X_{t\wedge T}}(T_r<T_l)=\mathbf{P}_x\left( T_r\circ \theta_{t\wedge T}<T_l\circ \theta_{t\wedge T}| \sF_{t\wedge T}\right).
\end{equation}
Since $t\wedge T<T_l$ and $t\wedge T<T_r$,  it follows that $T_l=t\wedge T + T_l\circ  \theta_{t\wedge T}$ and $T_r=t\wedge T + T_r\circ  \theta_{t\wedge T}$.  This,  together with $\{T_r<T_l\}\in \sF_T$,  implies that 
\begin{equation}\label{eq:58}
	\tilde{X}^T_t=\mathbf{P}_x\left(T_r<T_l|\sF_{t\wedge T} \right)=\mathbf{P}_x\left(T_r<T_l|\sF_{t} \right).
\end{equation}

Next,  we consider the case $l\notin E$ and $r\in E$.  In view of Lemma~\ref{LM54}~(4),  $\zeta=T_l=\lim_{E\ni l_n\downarrow l}T_{l_n}$.  When $X$ is conservative (see Remark~\ref{RM53}~(4)),  $T_{l_n}$ is a sequence of $\{\sF_t\}$-stopping times that increase to $\infty$.  Mimicking the argument treating the case $l,r\in E$,  we can obtain that $\bs(X_{t\wedge T_{l_n}\wedge T_r})$ is a Doob martingale.  Hence $\tilde{X}^T$ is a local martingale.  When $X$ is not conservative,  in view of Corollary~\ref{COR59},  $\bs(l)>-\infty$.  Letting $b=r$ and $a\downarrow l$ in \eqref{eq:55},  we get that
\[
\bs(y)=(\bs(r)-\bs(l)) \mathbf{P}_y(T_r<T_l) +\bs(l),  \quad y\in E. 
\]
When $t\wedge T\geq T_l$,  $\bs(X^T_t)=\bs(l)$.  It follows that
\[
	\bs(X^T_t)=(\bs(r)-\bs(l))\mathbf{P}_{X^T_t}(T_r<T_l) \cdot 1_{\{t\wedge T<T_l\}} +\bs(l).  
\]
The strong Markov property of $X$ yields that
\[
\begin{aligned}
	\mathbf{P}_{X^T_t}(T_r<T_l) \cdot 1_{\{t\wedge T<T_l\}}&=\mathbf{P}_x\left(T_r\circ \theta_{t\wedge T}<T_l\circ \theta_{t\wedge T},  t\wedge T<T_l |\sF_{t\wedge T }\right) \\
	&=\mathbf{P}_x(T_r<T_l,  t\wedge T<T_l| \sF_{t\wedge T}) \\
	&=\mathbf{P}_x(T_r<T_l|\sF_{t}).
\end{aligned}\]
As a result,  $\tilde{X}^T$ is a Doob martingale and hence a local martingale. 

The other cases $l\in E, r\notin E$ and $l,r\notin E$ can be treated analogously.  That completes the proof.  
\end{proof}

\subsection{Speed measure of skip-free Hunt process}\label{SEC53}

In this subsection we define the speed measure $\mu$ for a skip-free Hunt process $X$ on its natural scale,  i.e.  $\bs(x)=x$. 


\subsubsection{$l,r\notin E$}\label{SEC531}

Let us begin with the case that $l,r\notin E$.   
Take $a, b\in E$ with $a<b$ and define a function on $[a,b]$ as follows:
\[
	h_{a,b}(x):=\left\lbrace
	\begin{aligned}
	&\mathbf{E}_x(T_a\wedge T_b),\quad x\in [a,b]\cap E,  \\
	&\frac{h_{a,b}(b_k)-h_{a,b}(a_k)}{b_k-a_k}\cdot (x-a_k)+h_{a,b}(a_k),\quad x\in (a_k,b_k)\subset (a,b),  k\geq 1.  
	\end{aligned}
	\right.
\]
Recall that a real valued function $f$ defined on $(a,b)$ is called \emph{convex} if 
\[
f(tx+(1-t)y)\leq t f(x)+(1-t)f(y),\quad \forall 0\leq t\leq 1,  x,y \in (a,b).  
\]
It is called \emph{concave} if $-f$ is convex.  For a convex function $f$,  both its left derivative $f'_-$ and right derivative $f'_+$ are increasing,  respectively left and right continuous,  and the set $\{x: f'_-(x)\neq f'_+(x)\}$ is at most countable.  Particularly,  the second derivative $f''$ of $f$ in the sense of distribution is a positive Radon measure on $(a,b)$;  see, e.g.,  \cite[Appendix~\S3]{RY99}.   We have this.

\begin{lemma}\label{LM511}
$h_{a,b}$ is concave on $(a,b)$.  
\end{lemma}
\begin{proof}
For convenience write $h$ for $h_{a,b}$.  It suffices to show that for  $x,y,z\in [a,b]$ with $x<y<z$,  
\begin{equation}\label{eq:59}
h(y)\geq \frac{z-y}{z-x}h(x)+\frac{y-x}{z-x} h(z).  
\end{equation}
Firstly take $x,y,z\in E$ with $a<x<y<z<b$ and set $T_1:=T_a\wedge T_b$ and $T_2:=T_x\wedge T_z$.  It follows from Lemma~\ref{LM54}~(2) that $T_2<T_1$,  $\mathbf{P}_y$-a.s.  Using the strong Markov property and Theorem~\ref{THM58},   we get that
\begin{equation}\label{eq:511}
\begin{aligned}
	h(y)&=\mathbf{E}_y(T_2+T_1\circ \theta_{T_2}) \\
	&=\mathbf{E}_yT_2+\frac{z-y}{z-x}h(x)+\frac{y-x}{z-x} h(z)\\
	&>\frac{z-y}{z-x}h(x)+\frac{y-x}{z-x} h(z).  
\end{aligned}\end{equation}
Hence \eqref{eq:59} holds for $x,y,z\in E$.   When $x=a$ and $z<b$,  noting $h(a)=0$ and $T_z<T_b$ we have
\[
	h(y)\geq \mathbf{E}_y(T_a\wedge T_b,  T_z<T_a)\geq \mathbf{P}_y(T_z<T_a)\mathbf{E}_z (T_a\wedge T_b)
\] 
and \eqref{eq:59} still holds.  Analogously \eqref{eq:59} holds for $z=b$ and $x>a$.  The case $x=a$ and $z=b$ is obvious.  

Now we consider the general case $x_0,y_0,z_0\in (a,b)$ with $x_0<y_0<z_0$.  Set $\dagger_-:=\sup\{a\in E: a<\dagger_0\}$ and $\dagger_+:=\inf\{a\in E: a>\dagger_0\}$,  where $\dagger$ stands for $x,y$ or $z$,  and clearly $h(\dagger_0)=t_{\dagger_0} h(\dagger_-)+(1-t_{\dagger_0})h(\dagger_+)$ for $t_{\dagger_0}:=\frac{\dagger_+-\dagger_0}{\dagger_+-\dagger_-}$ for $\dagger_0\notin E$ and $t_{\dagger_0}:=0$ for $\dagger_0\in E$.  The first step has proved that \eqref{eq:59} holds for $x=x_\pm$,  $y=y_-$ and $z=z_-$.  Using $h(x_0)=t_{x_0} h(x_-)+(1-t_{x_0})h(x_+)$,  one  can easily obtain that \eqref{eq:59} holds for $x=x_0, y=y_-$ and $z=z_-$.  Analogously \eqref{eq:59} holds for $x=x_0, y=y_-$ and $z=z_+$.  Using $h(z_0)=t_{z_0} h(z_-)+(1-t_{z_0})h(z_+)$,  we find that \eqref{eq:59} holds for $x=x_0, y=y_-$ and $z=z_0$.  Analogously it holds also for $x=x_0,  y=y_+$ and $z=z_0$. Finally by means of $h(y_0)=t_{y_0} h(y_-)+(1-t_{y_0})h(y_+)$ we eventually conclude \eqref{eq:59} for $x=x_0,y=y_0$ and $z=z_0$.  That completes the proof. 
\end{proof}

Take $a',b'\in E$ with $a'<a<b<b'$ and define $h_{a',b'}$ analogously.  It is easy to compute that
\[
	h_{a',b'}(x)=h_{a,b}(x)+\frac{b-x}{b-a}h_{a',b'}(a)+\frac{x-a}{b-a}h_{a',b'}(b),\quad x\in (a,b),
\]
so that the restriction of the Radon measure $-h''_{a',b'}$ to $(a,b)$ is identified with $-h''_{a,b}$.  Define the speed measure $\mu$ on $(l,r)$  as follows:
\begin{equation}\label{eq:510}
	\mu|_{(a,b)}:=-\frac{1}{2}h''_{a,b},\quad \forall a,b\in E\text{ s.t. } l<a<b<r.  
\end{equation}
Noting the following result,  we denote the restriction of $\mu$ to $E$ still by $\mu$.  

\begin{lemma}\label{LM512}
Let $\mu$ be defined as \eqref{eq:510}.  Then $\mu$ is a positive Radon measure with $\text{supp}[\mu]=E$.  
\end{lemma}
\begin{proof}
Clearly,  $\mu$ is a positive Radon measure and $\mu((a_k,b_k))=0$ for any $k\geq 1$.  Take $(\alpha,\beta)\subset (l,r)$ such that $\mu((\alpha,\beta))=0$.  We need to show that $$(\alpha,\beta)\cap E=\emptyset$$ so that the support of $\mu$ is $E$.  To do this take $(a,b)\supset [\alpha, \beta]\supset (\alpha,\beta)$ with $a,b\in E$ and let $h:=h_{a,b}$.  Then $h''=0$ on $(\alpha,\beta)$ implies that there is a constant $k$ depending on $\alpha,  \beta$ such that 
\begin{equation}\label{eq:512}
	h(x)-h(\alpha)=k(x-\alpha),\quad x\in [\alpha,\beta]. 
\end{equation}
Argue by contradiction and suppose $x\in (\alpha, \beta)\cap E$.  Using \eqref{eq:511} and repeating the argument in the second paragraph of the proof of Lemma~\ref{LM511},  we can conclude that
\[
	h(x)>\frac{\beta-x}{\beta-\alpha} h(\alpha)+\frac{x-\alpha}{\beta-\alpha} h(\beta).  
\]
This violates \eqref{eq:512}.  That completes the proof.  
\end{proof}

\subsubsection{$l\in E$ but $r\notin E$}\label{SEC532}

Take $b\in E$ with $l<b<r$.  Using $\mathbf{P}_l(T_b<\infty)>0$ and repeating the argument in \cite[Chapter VII,  Proposition~3.1]{RY99},   one have that $\mathbf{E}_xT_b<\infty$ for all $x\in [l,b]$ and particularly,  $\mathbf{P}_x(T_b<\infty)=1$. Define
\[
	h_b(x):=\left\lbrace
	\begin{aligned}
	&\mathbf{E}_x(T_b),\quad x\in [l,b]\cap E,  \\
	&\frac{h_{b}(b_k)-h_{b}(a_k)}{b_k-a_k}\cdot (x-a_k)+h_{b}(a_k),\quad x\in (a_k,b_k)\subset (l,b),  k\geq 1.  
	\end{aligned}
	\right.
\]
Mimicking Lemma~\ref{LM511},  we can obtain that $h_b$ is concave on $(l, b)$.  Since $h_b$ is decreasing on $[l,b)$,  the extended function $h_b$ obtained by letting $h_b(x):=h_b(l)$ for $x\leq l$ is concave on $(-\infty,  b)$.  Take another $b'\in E$ such that $b<b'$ and let $h_{b'}$ be the concave function defined on $(-\infty, b')$ analogously.  Then
\[
	h_{b'}(x)=h_b(x)+\mathbf{E}_b(T_{b'}).  
\]
Hence the restriction of $-h''_{b'}$ to $(-\infty, b)$ is identified with $-h''_b$.  Define the speed measure $\mu$ on $(-\infty, r)$ as
\begin{equation}\label{eq:513}
	\mu|_{(-\infty, b)}:=-\frac{1}{2}h''_b,\quad  r>b\in E.  
\end{equation}
As an analogue of Lemma~\ref{LM512},  $\mu$ a positive Radon measure and $\text{supp}[\mu]=E$.  The restriction of $\mu$ to $E$ is still denoted by $\mu$.  

\subsubsection{$l\notin E$ and $r\in E$}\label{SEC533}

As an analogue of $h_b$ in \S\ref{SEC532},  we set a function $h_a$ on $(a,  \infty)$ for $a\in E$ with $l<a<r$:
\[
	h_a(x):=\left\lbrace
	\begin{aligned}
	&\mathbf{E}_x(T_a),\quad x\in [a,r]\cap E,  \\
	&\frac{h_{b}(b_k)-h_{b}(a_k)}{b_k-a_k}\cdot (x-a_k)+h_{b}(a_k),\quad x\in (a_k,b_k)\subset (a,r),  k\geq 1, \\
	&\mathbf{E}_rT_a,\quad x\geq r.  
	\end{aligned}
	\right.
\]
Note that $h_a$ is concave and $-h''_a$ extends consistently to $(l, \infty)$ as $a\downarrow l$.  Define the speed measure as this limit,  i.e. 
\begin{equation}\label{eq:514}
	\mu|_{(a,\infty)}:=-\frac{1}{2}h''_a,\quad l< a\in E. 
\end{equation}
The support of $\mu$ is $E$ and we denote still by $\mu$ the restriction of $\mu$ to $E$.  

\subsubsection{$l, r\in E$}

Let $h_{l,r}:=h_{a,b}$ with $a=l, b=r$ as in \S\ref{SEC531},  $h_r:=h_b$ with $b=r$ as in \S\ref{SEC532} and $h_l:=h_a$ with $a=l$ as in \S\ref{SEC533}.  Then $-h''_{l,r},  -h''_l$ and $-h''_r$ are all positive Radon measures on $(l,r)$,  $[l, r)$ and $(l,r]$ respectively.  

\begin{lemma}
The restrictions of $-h''_{l,r},  -h''_l$ and $-h''_r$ to $(l,r)$ are all identified.  
\end{lemma}
\begin{proof}
Take $x\in (l,r)$.  We have $h_l(x)=\mathbf{E}_x(T_l; T_l<T_r)+\mathbf{E}_x(T_l; T_l>T_r)$.  Note that
\[
\begin{aligned}
	\mathbf{E}_x(T_l; T_l>T_r)&=\mathbf{E}_x(T_r+T_l\circ \theta_{T_r}; T_l>T_r) \\
	&=\mathbf{E}_r(T_r; T_l>T_r)+\mathbf{E}_rT_l \cdot \frac{x-l}{r-l}.  
\end{aligned}\]
It follows that $h_l(x)=\mathbf{E}_x(T_l\wedge T_r)+\mathbf{E}_rT_l \cdot \frac{x-l}{r-l}=h_{l,r}(x)+\mathbf{E}_rT_l \cdot \frac{x-l}{r-l}$.  Hence $-h''_l=-h''_{l,r}$ on $(l,r)$.  Another identity $-h''_r=-h''_{l,r}$ on $(l,r)$ can be obtained similarly.  That completes the proof. 
\end{proof}

Due to this lemma we define the speed measure $\mu$ on $[l,r]$ as follows:
\begin{equation}\label{eq:515}
	\mu|_{[l,r)}:=-\frac{1}{2}h''_r,\quad \mu|_{(l,r]}:=-\frac{1}{2}h''_l. 
\end{equation}
Clearly $\mu$ is a positive Radon measure with $\text{supp}[\mu]=E$. 

\subsubsection{Summary}

We summarize the definition of speed measure for a skip-free Hunt process as follows.  

\begin{definition}\label{DEF514}
\begin{itemize}
\item[(1)] Let $X$ be a skip-free Hunt process on $E$ on its natural scale, i.e. $\bs(x)=x$. The measure $\mu$ defined as \eqref{eq:510},  \eqref{eq:513},  \eqref{eq:514} and \eqref{eq:515} for the cases mentioned above respectively is called the \emph{speed measure} of $X$. 
\item[(2)] For a skip-free Hunt process $X$ on $E$ with scale function $\bs$,  let $\tilde{\mu}$ be the speed measure of $\tilde{X}=(\bs(X_t))_{t\geq 0}$.  Then the image measure $\mu:=\tilde\mu\circ \bs^{-1}$ of $\tilde\mu$ under the map $\bs$ is called the \emph{speed measure} of $X$.  
\end{itemize}
\end{definition}

Note that the speed measure of $X$ is a positive Radon measure on $E$ with full support.  As an extension of the classical characterization for regular diffusions,  it will turn out in Theorem~\ref{THM71} that a skip-free Hunt process on its natural scale can be always constructed from a Brownian motion,  by way of a time change using this speed measure.  Particularly,  a skip-free Hunt process is uniquely determined by its scale function and speed measure;  see Corollary~\ref{COR64}.  

\section{Quasidiffusions}\label{SEC6}

Before moving on we introduce a special family of skip-free Hunt processes that have been widely studied.  

\subsection{Quasidiffusions as standard processes}\label{SEC61}

Let $m$ be an extended real valued, right continuous,  (not necessarily strictly) increasing and non-constant function on $\bR$,  and set $m(\infty):=\lim_{x\uparrow \infty}m(x),  m(-\infty):=\lim_{x\downarrow -\infty}m(x)$.   Put 
\begin{equation}\label{eq:51-2}
\begin{aligned}
	&l_0:=\inf\{x\in \bR: m(x)>-\infty\},\quad r_0:=\sup\{x\in \bR: m(x)<\infty\},  \\
	&l:=\inf\{x>l_0: m(x)>m(l_0)\},\quad r:=\sup\{x<r_0: m(x)<m(r_0-)\}.  
\end{aligned}\end{equation}
To avoid trivial case assume that $l<r$.  Define
\[
E_m:=\{x\in [l,r]\cap (l_0,r_0): \exists \varepsilon>0\text{ s.t. }m(x-\varepsilon)<m(x+\varepsilon)\}.  
\]
The function $m$ corresponds to a measure on $\bR$,  still denoted by $m$ if no confusions caused.  

\begin{lemma}\label{LM51}
The set $E_m$ is a nearly closed subset of $\bR$ ended by $l$ and $r$, and the restriction of $m$ to $E_m$ is a  positive Radon measure with full support.
\end{lemma}
\begin{proof}
Let $\overline{E}_m:=E_m\cup \{l,r\}$.  For $x\in [l,r]\setminus \overline E_m$,  $m(x-\varepsilon)=m(x+\varepsilon)$ for some $\varepsilon>0$.  Hence $(x-\varepsilon,x+\varepsilon)\subset [l,r]\setminus \overline{E}_m$ whenever $x-\varepsilon>l$ and $x+\varepsilon<r$.  As a result,  $\overline{E}_m$ is closed in $\overline{\bR}$.   If $l\notin E_m$,  then we must have $l=l_0$ and $l$ can be approximated by points in $E_m$ from the right.  Analogical fact holds for $r$.  Since $E_m\subset \bR$ we can eventually conclude that $E_m\in \mathscr K$ and $E_m$ is ended by $l$ and $r$.  The second assertion for $m$ is obvious.  That completes the proof. 
\end{proof}
\begin{remark}
There may appear various cases for the positions of $r_0$ and $r$:
\begin{itemize}
\item[(1)] $r=r_0=\infty$ whenever $m(\infty)\leq  \infty$ and $m(x)<m(\infty)$ for any $x<\infty$.
\item[(2)] $r<r_0=\infty$ whenever $m(x)=m(\infty)<\infty$ for $x\in [r, \infty)$ and $m(x)<m(r)$ for $x<r$.  In this case $r\in E_m$.  
\item[(3)] $r=r_0<\infty$ whenever $m(x)=\infty$ for $x\in [r,\infty)$ and $m(x)<m(r-)$ for $x<r$.  In this case $r\notin E_m$.  
\item[(4)] $r<r_0<\infty$ whenever $m(x)=\infty$ for $x\in [r_0,\infty)$,  $m(x)=m(r)<\infty$ for $x\in [r, r_0)$ and $m(x)<m(r-)$ for $x<r$.  In this case $r\in E_m$.  
\end{itemize}
The positions of $l_0$ and $l$ can be argued similarly.  
\end{remark}

Let $W=(W_t, \mathscr{F}^W_t, \mathbf{P}_x)$ be a Brownian motion on $\bR$ and $\ell^W(t,x)$ be its local time normalized such that for any bounded Borel measurable function $f$ on $\bR$ and $t\geq 0$, 
\[
	\int_0^t f(W_s)ds=2\int_\bR \ell^W(t,x)f(x)dx.  
\]
Define $S_t:=\int_{\bR} \ell^W(t,x)m(dx)$ for $t\geq 0$ and
\[
	T_t:=\inf\{u>0: S_u>t\},\quad t\geq 0.  
\]
Then $T=(T_t)_{t\geq 0}$ is a strictly increasing (before $\zeta$ defined as below), right continuous family of $\sF_t$-stopping times with $T_0=0$,  $\mathbf{P}_x$-a.s.   Define
\[
	\sF_t:=\sF^W_{T_t},\quad \zeta:=\inf\{t>0: W_{T_t}\notin (l_0,r_0)\},\quad X_t:=W_{T_t}, \; 0\leq t<\zeta.  
\]
Then $\{X_t, \sF_t, (\mathbf{P}_x)_{x\in E_m},\zeta\}$,  called \emph{a quasidiffusion with speed measure $m$},  is a standard process with state space $E_m$ and lifetime $\zeta$; see \cite{K86}.   The definition of standard process is referred to in,  e.g.,  \cite{BG68}.  

\begin{lemma}\label{LM515}
Let $\tau:=\inf\{t>0: W_t\notin (l_0,r_0)\}$.  Then $\zeta=S_{\tau-}$,  $\mathbf{P}_x$-a.s.,  $x\in E_m$.  Furthermore,  the following hold:
\begin{itemize}
\item[(1)] When $\tau=\infty$,  i.e.  $|l_0|=|r_0|=\infty$, 
\[
\begin{aligned}
	 &S_u<\infty \;  (u<\infty),  \quad \zeta=\lim_{u\uparrow \infty}S_u=\infty;  \\
	 &T_t<\infty\;  (t<\infty), \quad T_\infty:=\lim_{t\uparrow \infty}T_t=\infty.  
\end{aligned}\]
\item[(2)] When $\tau<\infty$,  i.e.  $|l_0|<\infty$ or $|r_0|<\infty$,
\[
\begin{aligned}
	& S_u<\infty\; (u<\tau),\quad S_u=\infty\; (u>\tau),\quad \zeta=S_{\tau-}\leq \infty;  \\
	& T_t<\tau\;  (t<\zeta),\quad T_t=\tau\;  (t\geq \zeta\text{ if }\zeta<\infty).  
\end{aligned}
\]
\end{itemize}
\end{lemma}
\begin{proof}
Consider first that $\tau=\infty$.  Clearly $m$ is a Radon measure on $\bR$.  Hence \cite[Chapter X,  Corollary~2.10]{RY99} yields that $S_u<\infty$ for $u<\infty$.  Since $\ell^W(\infty,x)=\infty$ for any $x\in \bR$,  it follows that $\lim_{u\uparrow \infty}S_u=\infty$.  Particularly $T_t<\infty$ for $t<\infty$ and $T_\infty=\infty$.  Note that $\zeta=\inf\{t>0: T_t=\infty\}$ by its definition.  As a result,  $\zeta=\infty$.  

Now consider $\tau<\infty$.  Without loss of generality we only treat the case $l_0>-\infty,  r_0=\infty$ and $W_\tau=l_0$.  The other cases can be argued similarly.  
The fact $S_u<\infty$ for $u<\tau$ can be obtained by means of \cite[Chapter X, Corollary~2.10]{RY99} when the zero extension of $m|_{(l_0,\infty)}$ to $\bR$ is a Radon measure.  Otherwise set $\tau_n:=\inf\{t>0: W_t\notin (l_n,\infty)\}$,  where $l_n\downarrow l_0$,  and  \cite[Chapter X, Corollary~2.10]{RY99} implies that $S_u<\infty$ for $u<\tau_n$.  Since $\tau_n\uparrow \tau$,  it follows that $S_u<\infty$ for $u<\tau$.  
To prove $S_{\tau+\varepsilon}=\infty$ for $\varepsilon>0$,  note that $l_0>-\infty$ must lead to $m(\{l_0\})=\infty$ or $m((l_0, l_0+\delta))=\infty$ for $\delta>0$.   Since $W_0=x,  W_\tau=l_0$,  it follows that $\ell^W(\tau+\varepsilon, y)>0$ for any $y\in [l_0,  x)$.  Using the continuity of $\ell^W(\tau+\varepsilon, \cdot)$,  we get that
\[
	S_{\tau+\varepsilon}\geq \int_{[l_0,l_0+\delta)} \ell^W(\tau+\varepsilon, y)m(dy)=\infty.  
\]
Particularly $T_t<\tau$ for $t<S_{\tau-}$ and for $t\geq S_{\tau-}$ (if $S_{\tau-}<\infty$),  $T_t=\tau$.  Finally it suffices to prove $\zeta=S_{\tau-}$.  In fact,  $T_t<\tau$ for $t<S_{\tau-}$ implies that $\zeta\geq S_{\tau-}$.  To the contrary we have $T_t=\tau$ if $t:=S_{\tau-}<\infty$.  Hence $W_{T_t}=l_0$, as leads to $\zeta\leq S_{\tau-}$.  Therefore $\zeta=S_{\tau-}$ is eventually concluded.  
That completes the proof. 
\end{proof}

A quasidiffusion satisfies (SF) and (SR) in Definition~\ref{DEF51}; see \cite{G75, K86,  BK87}.  
Attaching the ceremony $\partial$ to $E_m$ by the same way as in \S\ref{SEC51},  one can find in Theorem~\ref{THM71} that a quasidiffusion is,  in fact,  a Hunt process on $E_m$. 
However it admits killing at $l$ or $r$ whenever $-\infty<l_0<l$ or $r<r_0<\infty$; see Remark~\ref{RM73}.  To bar this possibility we will assume in the next section that 
\begin{itemize}
\item[(QK)] If $l_0>-\infty$ (resp.  $r_0<\infty$),  then $l=l_0$  (resp.  $r=r_0$).  
\end{itemize}
In other words,  under the assumption \text{(QK)},  a quasidiffusion is a skip-free Hunt process on $E_m$. 

\subsection{Markov local times}

Define 
\[
\begin{aligned}
&\ell(t,x):=\ell^W(T_t,x),\quad 0\leq t<\zeta, x\in E_m,  \\
&\ell(t,x):=\lim_{s\uparrow \zeta}\ell(s,x),\quad t\geq \zeta,  x\in E_m,
\end{aligned}
\]
which are called the \emph{local times} of $X$ in \cite{BK87}.  To avoid ambiguity with semimartingale local times,  we call them the \emph{Markov local times} of $X$.  Note that $\ell$ is jointly continuous in $t$ and $x$, and $\ell(\cdot, x)$ increases at $t$ if and only if $T_t<T_\infty$ and $X_t=x$ or $X_{t-}=x$.  Moreover,  by virtue of \cite[Chapter V,  Proposition~1.4]{RY99},  it holds a.s.  that for any bounded measurable function $f$ and $0\leq t<\zeta$,
\begin{equation}\label{eq:61}
\int_0^t f(X_s)ds=\int_{E_m}\ell(t,x)f(x)m(dx);
\end{equation}
see also \cite[(2.1)]{BK87}.  In view of Lemma~\ref{LM515},  $T_t\leq \tau$ for any $t\geq 0$.  We set $\ell(t,j):=\ell^W(T_t,j)\equiv 0$ for $j=l_0$ or $r_0$ whenever $l_0$ or $r_0$ is finite.  

\begin{remark}
When $\tau<\infty$ and $\zeta<\infty$,  we have $\ell(\zeta,x)=\ell^W(T_\zeta, x)$ for $x\in E_m$ even if $T_{\zeta-}<T_\zeta$.  In fact,  for $u\in (T_{\zeta-}, T_\zeta)$,  $S_u$ is constant and hence $W_u\notin E_m$.  This yields $\ell^W(T_{\zeta-},x)=\ell^W(T_\zeta, x)$ for $x\in E_m$.  Particularly,  $\ell(\zeta,x)=\ell^W(T_\zeta,x)$.  
\end{remark}

As a time change of Brownian motion,  $X$ is always a semimartingale by letting $X_t:=W_{T_\zeta}$ for $t\geq \zeta$ if $\zeta<\infty$; see,  e.g., \cite{M78}.  In what follows we present its semimartingale decomposition obtained in \cite{BK87}. 
Let $\tilde{E}_m:=E_m\cup \left(\{l_0,r_0\}\cap \bR\right)$,  which is closed in $\bR$.  Then $\bR\setminus \tilde{E}_m$ is the union of mutually disjoint open intervals $I_k:=(a_k,b_k)$,  where $k$ runs through a subset $K$ of $\bN$ such that $0\in K$ (resp.  $1\in K$) if and only if the left (resp.  right) endpoint point of some interval is $-\infty$ (resp.  $\infty$).   Set $I_0:=(-\infty,  b_0)$ (resp.  $I_1:=(a_1,\infty)$) whenever $0\in K$ (resp.  $1\in K$) and $K':=K\setminus \{0,1\}$.  Note that $0\in K$ amounts to either $l_0>-\infty$ or $l>l_0=-\infty$.  In the former case $b_0=l_0$ and in the latter case $b_0=l$.  This holds analogously for $1\in K$.  For convenience we make the convention $b_0=-\infty$ or $a_1=\infty$ if $0\notin K$ or $1\notin K$,  and $\ell(\cdot, \pm\infty)=0$.  
 Let $A_{k,i}$ (resp.  $B_{k,i}$) be the time of the $i$-th jump of $X$ from $a_k$ to $b_k$ (resp. from $b_k$ to $a_k$) if this jump occurs,  otherwise let it be equal to $\infty$.  For $k\in K'$ put
\[
	A_k(t):=\sum_{i=1}^\infty 1_{[A_{k,i},\infty)}(t),\quad B_k(t):=\sum_{i=1}^\infty 1_{[B_{k,i},\infty)}(t),\quad t\geq 0,
\]
i.e.  the number of jumps from $a_k$ to $b_k$ or from $b_k$ to $a_k$ before time $t$.  The following result is due to \cite{BK87}.  

\begin{theorem}\label{THM517}
\begin{itemize}
\item[\rm(1)] For $k\in K'$,  $A_k(t)-\lambda_k \ell(t,a_k)$ and $B_k(t)-\lambda_k \ell(t,b_k)$ are local $\sF_t$-martingales,  where $\lambda_k:=(b_k-a_k)^{-1}$.  
\item[\rm(2)] $X$ is an $\sF_t$-semimartingale and admits the representation $X=M^c+M^d+A$,  where 
\[
	M^c_t:=\int_0^{T_t} 1_{\tilde{E}_m}(W_s)dW_s
\]
is a continuous local martingale,  
\[
	M^d_t:=\sum_{k\in K'}\left([(b_k-a_k)A_k(t)-\ell(t,a_k)]-[(b_k-a_k)B_k(t)-\ell(t,b_k)] \right)
\]
is a purely discontinuous local martingale and $A_t=\ell(t,b_0)-\ell(t,a_1)$ is the adapted continuous process with locally integrable variation.  
\end{itemize}
\end{theorem}
\begin{remark}\label{RM64}
\begin{itemize}
\item[(1)]  We should emphasize that,  as a standard process,  $X_t$ is defined as the ceremony $\partial$ for $t\geq \zeta$.  However viewed as a semimartingale in this theorem,  $X_t$ is by definition equal to $W_{T_\zeta}$ for $t\geq \zeta$.  To see this difference,  consider the special case $-\infty<l_0<l<r<r_0<\infty$.  The lifetime $\zeta$ is the first time of $(W_{T_t})_{t\geq 0}$ leaving $(l_0,r_0)$ and $W_{T_t}=W_{T_\zeta}=l_0$ or $r_0$ for $t\geq \zeta$ corresponding to $W_{T_{\zeta-}}=l$ or $r$.  As a standard process,  $X_{\zeta-}=W_{T_{\zeta-}}=l$ or $r$ while $X$ jumps to the ceremony $\partial$ at lifetime $\zeta$.  To be more exact,  the semimartingale in this theorem should be $(W_{T_t})_{t\geq 0}$ and the quasidiffusion $X$ is obtained by killing it at $\zeta$.  
\item[(2)] Note that $0\notin K$ (resp.  $1\notin K$) amounts to $l_0=l=-\infty$ (resp.  $r=r_0=\infty$).  
Suppose $0\in K$ (resp.  $1\in K$).  Then $\ell(\cdot, b_0)\equiv 0$ (resp.  $\ell(\cdot, a_1)\equiv 0$) if and only if $b_0=l_0>-\infty$ (resp.  $a_1=r_0<\infty$).  Clearly $l_0>-\infty$ (resp. $r_0<\infty$) implies $b_0=l_0$ (resp.  $a_1=r_0$).   Particularly,  $X$ is a local martingale,  if and only if 
\[
\text{neither }-\infty=l_0<l \text{ nor }r<r_0=\infty. 
\]
 For $k\in K'$,  $\ell(\cdot, a_k)\equiv 0$ (resp.  $\ell(\cdot,b_k)\equiv 0$),  if and only if $-\infty<l_0=a_k<b_k=l$ (resp.  $r=a_k<b_k=r_0<\infty$);  hence this is impossible if \text{(QK)} is assumed.  
\item[(3)] Recall that in \eqref{eq:51} we use another index set $\{k: k\geq 1\}$ to decompose the open set $[l,r]\setminus \overline{E}$.  When \text{(QK)} is assumed and $E=E_m$,  the sequence of open intervals in \eqref{eq:51} is identified with $\{I_k, k\in K'\}$.  In addition,  $0\in K$ (resp.  $1\in K$) if and only if $l$ (resp.  $r$) is finite.  Meanwhile $b_0=l$  (resp.  $a_1=r$).  
\end{itemize}
\end{remark}

\subsection{Semimartingale local times}

Now we turn to prepare the It\^o-Tanaka-Meyer formula (see,  e.g.,  \cite[Chapter VIII,  (29)]{DM82} and \cite[Theorem~9.46]{HWY92}) for the semimartingale $X$.  
More precisely,  let $F$ be a convex function on $\bR$,  $F'_-$ be its left derivative and $F''$ be the Radon measure corresponding to the second derivative of $F$ in the sense of distribution.  Then the It\^o-Tanaka-Meyer formula states that $F(X)$ is also a semimartingale and admits the representation:
\begin{equation}\label{eq:516}
\begin{aligned}
	F(X_t)&=F(X_0)+\int_0^t F'_-(X_{s-})dX_s \\
	&+\sum_{0<s\leq t}\left[F(X_s)-F(X_{s-})-F'_-(X_{s-})\Delta X_s \right]+\frac{1}{2}\int_{\bR}L(t,x)F''(dx),
\end{aligned}\end{equation}
where $\Delta X_s:=X_s-X_{s-}$ and for any $x\in \bR$, $L(\cdot, x)$,  called the \emph{semimartingale local time} of $X$ at $x$,  is a continuous adapted increasing process with $L(0,x)=0$.  The expression of semimartingale local times is formulated in the following lemma.  For convenience we make the convention $L(\cdot,  \pm \infty)\equiv 0$.  

\begin{lemma}\label{LM65}
The semimartingale local time of $X$ at $x\in \bR$ is
\[
	L(t,x)=\left\lbrace
	\begin{aligned}
	&2\ell(t,x),\quad x\in E_m\setminus \{a_k: k\in K\};\\
	&0,\qquad\quad\;\; \text{otherwise}. 
	\end{aligned}
	\right.
\]
\end{lemma}
\begin{proof}
Note that $W_{T_{t-}}\in \tilde{E}_m$ and thus \cite[Theorem~9.44]{HWY92} yields that $L(\cdot, x)\equiv 0$ for $x\notin \tilde{E}_m$.  For $x=l_0>-\infty$,  we assert that if $X_{t-}=l_0$ then $t\geq \zeta$ and $X_s=l_0$ for any $s\geq \zeta$.  This fact tells us that,  in view of \cite[Theorem~9.44]{HWY92},  $dL(\cdot, l_0)$ does not charge $(0,\zeta)\cup (\zeta, \infty)$,  as leads to $L(\cdot, l_0)\equiv 0$.  To prove the assertion,  note that $T_{t-}\leq T_t\leq \tau=\inf\{s>0:W_s\notin (l_0,r_0)\}$ by Lemma~\ref{LM515}.  Hence $W_{T_{t-}}=X_{t-}=l_0$ implies $T_{t-}=T_t=\tau<\infty$ and $W_\tau=l_0$.   We have shown in the proof of Lemma~\ref{LM515} that $T_\zeta=\tau$.  Since $T_s$ is strictly increasing before $\zeta$,  it follows that $t\geq \zeta$.  Particularly $X_s=l_0$ for any $s\geq \zeta$.  We eventually obtain $L(\cdot, l_0)\equiv 0$.  Analogously for $x=r_0<\infty$ we have $L(\cdot,  r_0)\equiv 0$.  

Next consider $x=a_k$ for some $k\in K$.  When $k=1\in K$,  it holds that $a_1=r_0<\infty$ or $a_1=r<r_0=\infty$.  In the former case we have already proved that $L(\cdot, a_1)\equiv 0$.  In the latter case $X_t\leq a_1$ and hence $X_{t-}\leq a_1$ for any $t\geq 0$.  Using Tanaka-Meyer formula (see,  e.g.,  \cite[Theorem~9.43]{HWY92}),   we get that $L(\cdot, a_1)\equiv 0$.  Now let $k\in K'$.  Set $\varphi(t):=\int_0^t 1_{I_k}(u)du$, which is the difference of two convex functions.  Applying \eqref{eq:516} to $\varphi$,  we have
\begin{equation}\label{eq:517}
\begin{aligned}
	\varphi&(X_t)-\varphi(X_0)=\int_0^t 1_{(a_k,b_k]}(X_{s-})dX_s \\ &+\sum_{0<s\leq t}\left[\varphi(X_s)-\varphi(X_{s-})-1_{(a_k,b_k]}(X_{s-})\Delta X_s \right] +\frac{1}{2}L(t,a_k)-\frac{1}{2}L(t,b_k).  
\end{aligned}\end{equation}
It is easy to verify that $\varphi(X_t)-\varphi(X_0)=(b_k-a_k)(A_k(t)-B_k(t))$;  see also \cite[page 243]{BK87}.  Since $\ell(\cdot,  b_0)$ increases only at $t$ such that $X_t=b_0$ or $X_{t-}=b_0$,  it follows that
\begin{equation}\label{eq:518}
\int_0^t 1_{(a_k,b_k]}(X_{s-})d\ell(s, b_0)\leq 1_{\{a_k=b_0\}}\int_0^t 1_{\{X_{s-}=b_k, X_s=a_k\}}d\ell(s,b_0).
\end{equation}
Note that $\{0<s\leq t: X_{s-}=b_k, X_s=a_k\}$ contains only finite elements and $\ell(\cdot, b_0)$ is continuous.  Hence the left hand side of \eqref{eq:518} is equal to $0$.  Analogously one can compute that if $b_k=a_1$ then
\[
	\int_0^t 1_{(a_k,b_k]}(X_{s-})d\ell(s, a_1)=\ell(t,a_1);
\]
otherwise it is equal to $0$.  
If $b_k=a_p$ for some $p\in K'$,  the second term on the right side of \eqref{eq:517} is equal to 
\[
	(b_k-a_k)A_k(t)-(b_p-a_p)A_p(t);
\]
otherwise it is equal to $(b_k-a_k)A_k(t)$.  Using these computations and applying Theorem~\ref{THM517},  we can obtain that if $b_k=a_1$ or $a_p$,  then $\frac{1}{2}L(\cdot,a_k)-\frac{1}{2}L(\cdot, b_k)$ is a local martingale;  otherwise $\ell(\cdot, b_k)+\frac{1}{2}L(\cdot,a_k)-\frac{1}{2}L(\cdot, b_k)$ is a local martingale.  Hence either $L(\cdot, a_k)-L(\cdot, b_k)\equiv 0$ or $\ell(\cdot, b_k)+\frac{1}{2}L(\cdot,a_k)-\frac{1}{2}L(\cdot, b_k)\equiv 0$.  Note that $dL(\cdot, x)$ does not charge $\{t: X_{t-}\neq x\}$ for any $x\in \bR$;  see,  e.g.,  \cite[Theorem~9.44]{HWY92}.  Particularly,  the positive measures $dL(\cdot, a_k)$ and $dL(\cdot, b_k)$ on $[0,\infty)$ are singular.  
Therefore one can easily conclude that $L(\cdot, a_k)\equiv 0$.  In addition,  we also get that if $b_k\neq a_1, a_p$,  then $L(\cdot, b_k)=2\ell(\cdot, b_k)$.  

Thirdly consider $x\in E_m$ such that there are $y_0>y>x$ such that $[x,y]\subset [x, y_0)\subset E_m$.  Put $\varphi(t):=\int_0^t 1_{(x,y]}(u)du$.  Applying \eqref{eq:516} to $\varphi$ we have
\begin{equation}\label{eq:519}
	\varphi(X_t)-\varphi(X_0)=\int_0^t 1_{(x,y]}(X_{s-})dX_s +\frac{1}{2}L(t,x)-\frac{1}{2}L(t,y).  
\end{equation}
The first term on the right hand side is clearly a local martingale.  
On the other hand,  applying Tanaka's formula for Brownian motion we get that
\[
	\varphi(W_t)-\varphi(W_0)=\int_0^t 1_{(x,y]}(W_s)dW_s+\ell^W(t,x)-\ell^W(t,y).  
\]
Write $U_t:=\int_0^t1_{(x,y]}(W_s)dW_s$.  
It follows that
\begin{equation}\label{eq:520}
	\varphi(X_t)-\varphi(X_0)=U_{T_t}+\ell(t,x)-\ell(t,y).  
\end{equation}
Note that $W_u\notin E_m$ for $u\in (T_{t-},T_t)$.  The bracket $\langle U\rangle$ and thus also $U$ is constant on $(T_{t-},T_t)$.  This implies the $T$-continuity of $U$ in the sense of \cite[Chapter V,  Definition~1.3]{RY99}.  In view of \cite[Chapter V, Proposition~1.5]{RY99},   we obtain that $U_{T_t}$ is a continuous local martingale.  As a result,  \eqref{eq:519} and \eqref{eq:520} yield that 
\[
	\frac{1}{2}L(t,x)-\frac{1}{2}L(t,y)=\ell(t,x)-\ell(t,y).  
\]
Note that $dL(\cdot, x)$ and $dL(\cdot, y)$ are singular.  So are $d\ell(\cdot, x)$ and $d\ell(\cdot, y)$.  Particularly we must have $L(t,x)=2\ell(t,x)$ and $L(t,y)=2\ell(t,y)$.  

Finally let $x\in E_m$ such that there exists a sequence of intervals in $\{I_k: k\in K'\}$ decreasing to $x$.  Without loss of generality we assume that $x<b_2$ and $(b_2, a_1)\cap E_m\neq \emptyset$.  It may appear that the intervals decreasing to $x$ are all ended by points in $\{a_p: p\in K \}$.  We ignore this case at first and suppose that $b_2\notin \{a_p: p\in K\}$.  Put $\varphi(t):=\int_0^t 1_{(x, b_2]}(u)du$.  Applying \eqref{eq:516} to $\varphi$ we get that
\begin{equation}\label{eq:521}
	\varphi(X_t)-\varphi(X_0)=\int_0^t 1_{(x,b_2]}(X_{s-})dX_s+\frac{1}{2}L(t,x)-\frac{1}{2}L(t,b_2),
\end{equation}
where the first term on the right hand side is a local martingale due to $(b_2,a_1)\cap E_m\neq \emptyset$.  
Mimicking \eqref{eq:520} we also have
\begin{equation}\label{eq:522}
	\varphi(X_t)-\varphi(X_0)=\int_0^{T_t}1_{(x,b_2]}(W_s)dW_s+\ell(t,x)-\ell(t,b_2).  
\end{equation}
Put $\varphi_k(t):=\int_0^t 1_{(a_k,b_k]}(u)du$ for $k\in K$ such that $I_k\subset (x, b_2)$.  Denote by $K''$ the set of all these $k$.  Then 
\begin{equation}\label{eq:523}
\varphi_k(X_t)-\varphi_k(X_0)=(b_k-a_k)(A_k(t)-B_k(t))
\end{equation}
and on account of Tanaka's formula for Brownian motion, 
\begin{equation}\label{eq:524}
	\varphi_k(X_t)-\varphi_k(X_0)=\int_0^{T_t}1_{(a_k,b_k]}(W_s)dW_s + \ell(t,a_k)-\ell(t,b_k).  
\end{equation}
Repeating the argument in the previous step,  one can conclude that
\[
	t\mapsto \int_0^{T_t} 1_{(x,b_2]\setminus \cup_{k\in K''}(a_k,b_k]}(W_s)dW_s
\]
is a continuous local martingale.  Therefore \eqref{eq:521}, \eqref{eq:522}, \eqref{eq:523} and \eqref{eq:524} tell us that
\[
\frac{1}{2}L(t,x)-\frac{1}{2}L(t,b_2)=\ell(t,x)-\ell(t,b_2).  
\]
Since we have obtained $L(t,b_2)=2\ell(t,b_2)$ when treating the case $x=a_k$ for $k\in K$,  it follows that $L(t,x)=2\ell(t,x)$.  At last we consider the ignored case and meanwhile $b_2=a_k$ for some $k\in K'$.  Currently there is an additional term $-(b_k-a_k)A_k(t)$ in \eqref{eq:521}.  The above argument leads to $\frac{1}{2}L(t,x)-\frac{1}{2}L(t,b_2)=\ell(t,x)$,  while $L(t,b_2)=L(t,a_k)\equiv 0$ as proved in the second step.  Therefore we still have $L(t,x)=2\ell(t,x)$.  That completes the proof. 
\end{proof}

\section{Equivalence between image regularized Markov processes,  skip-free Hunt processes and quasidiffusions}\label{SEC7}

Three stochastic models in one dimension have been argued in previous sections:
\begin{itemize}
\item[(i)]  The first one is the image regularized Markov process $\widehat X$ associated with $(I,\bs,\fm)$ obtained in Theorem~\ref{THM6}.  Its speed measure is $\widehat{\fm}$.  
\item[(ii)] The second is a skip-free Hunt process,  which can be transformed into another skip-free Hunt process on its natural scale by virtue of the scale function obtained in Theorem~\ref{THM58}.  Its speed measure $\mu$ is defined for various cases in Definition~\ref{DEF514}. 
\item[(iii)]  The third is a quasidiffusion with speed measure $m$,  as reviewed in \S\ref{SEC6}. 
\end{itemize}
 In this section,  we will prove that when are on natural scales and have identifying speed measures,  they are all \emph{equivalent} in the sense that their transition functions coincide.   

\subsection{Equivalence of Markov processes on natural scales}

Note that the state spaces of image regularized Markov process $\widehat{X}$,  skip-free Hunt process on its natural scale and quasidiffusion are all nearly closed subsets of $\bR$; see Remarks~\ref{RM34} and \ref{RM48} and Lemma~\ref{LM51}.  Let $E\in \mathscr K$ be a nearly closed subset of $\bR$ ended by $l$ and $r$,  and we add a ceremony $\partial$  attached to $E$ by the same way as in \S\ref{SEC51}.  
The main result is as follows.

\begin{theorem}\label{THM71}
Let $E\in \mathscr K$ and $\mu$ be a fully supported positive Radon measure on $E$.  The following Markov processes are equivalent:
\begin{itemize}
\item[\rm(i)] The image regularized Markov process $\widehat{X}$ associated with $(I,\bs,\fm)$ whose image regularization $(\widehat{I},  \widehat{\bs}, \widehat{\fm})$ satisfies $\widehat{I}=E$ and $\widehat{\fm}=\mu$;
\item[\rm(ii)] A skip-free Hunt process on $E$ on its natural scale whose speed measure is equal to $\mu$;
\item[\rm(iii)] A quasidiffusion with speed measure $m$ such that $E_m=E$,  $m|_{E}=2\mu$ and \text{(QK)} holds for $m$.
\end{itemize}
\end{theorem}

The reminder of this subsection is devoted to proving this theorem.  Since involved and long,  the proof will be divided into several parts.  

\subsubsection{Quasidiffusion and time-changed Brownian motion}\label{SEC711}

Define an open interval
\[
V:=(l^0, r^0),
\]
where $l^0:=l$ if $l\notin E$ and $l^0:=-\infty$ otherwise,  $r^0:=r$ if $r\notin E$ and $r^0:=\infty$ otherwise. 
Let $\widehat{W}$ be the absorbing Brownian motion on $V$ associated with the Dirichlet form $(\frac{1}{2}\mathbf{D}, H^1_0(V))$ where $H^1_0(V)$ is the closure of the family of smooth functions with compact support on $V$ in the Sobolev space of order $1$ over $V$ and for $f,g\in H^1_0(V)$,  $\frac{1}{2}\mathbf{D}(f,g)=\frac{1}{2}\int_V f'(x)g'(x)dx$; see,  e.g.,  \cite[Example~3.5.7]{CF12}.  We first prove that the Dirichlet form $(\widehat{\sE},\widehat{\sF})$ of $\widehat{X}$ is the trace Dirichlet form of $(\frac{1}{2}\mathbf{D}, H^1_0(V))$ on $L^2(E,\mu)$.  In other words,  $\widehat{X}$ is the time-changed process of $\widehat{W}$ by the PCAF corresponding to $\mu$ (cf.  Theorem~\ref{THM19}).

\begin{lemma}
Let $\widehat{X}$ be in (i).  Then its Dirichlet form $(\widehat{\sE},\widehat{\sF})$ is the trace Dirichlet form of $(\frac{1}{2}\mathbf{D}, H^1_0(V))$ on $L^2(E,\mu)$.  
\end{lemma}
\begin{proof}
Adopt the same notations as in the proof of Theorem~\ref{THM19}.
Further let $(\sA,\sG)$ be the trace Dirichlet form of $(\frac{1}{2}\mathbf{D}, H^1_0(V))$ on $L^2(E,\mu)$.  Denote by $H^1_e(V)$ the extended Dirichlet space of $(\frac{1}{2}\mathbf{D}, H_0^1(V))$; see \eqref{eq:320-2}.  Note that $H^1_e(V)|_{\widehat{I}}=H^1_e(\widehat{J})|_{\widehat{I}}$.  It follows from \eqref{eq:17} that $\sG=H^1_0(V)|_{\widehat{I}}\cap L^2(\widehat{I},\widehat{\fm})=\widehat{\sF}$.  Repeating the argument in the proof of \cite[Theorem~2.1]{LY17},  we can still obtain that $\sA(\widehat\varphi,\widehat\varphi)=\widehat{\sE}(\widehat{\varphi},\widehat{\varphi})$ for $\widehat{\varphi}\in \sG$.  That completes the proof.  
\end{proof}

We turn to the quasidiffusion $X$ in (iii) and will prove that $X$ is actually a time-changed process of $\widehat{W}$ by the PCAF corresponding to $\mu$,  so that the equivalence between $X$ and $\widehat{X}$ is obtained.  

\begin{proof}[Proof of the equivalence between (i) and (iii)]
Adopt the same notations as in \S\ref{SEC6}.  Then $\widehat{W}$ is identified with the killed process of Brownian motion $W$ at $\tau_V:=\{t>0: W_t\notin V\}$.  The assumption \text{(QK)} implies that $V=(l_0,r_0)$ and thus $\tau_V=\tau$ where $\tau$ is defined in Lemma~\ref{LM515}. 

Note that for any $x\in V$,  $\ell^{\widehat W}(t,x):=\ell^W(t\wedge \tau, x), t\geq 0$,  is a PCAF of $\widehat{W}$ whose Revuz measure is $\frac{1}{2}\delta_x$; see, e.g.,  \cite[Chapter X,  Proposition~2.4]{RY99} and \cite[Proposition~4.1.10]{CF12}.  Set
\[
	\widehat{S}_t:=\int_{E_m} \ell^{\widehat{W}}(t,x)m(dx),\quad t\geq 0.  
\]	
Clearly $\widehat{S}$ is a PCAF of $\widehat{W}$ corresponding to the Revuz measure $\mu$.  Define
\[
	\widehat{T}_t:=\left\lbrace
	\begin{aligned}
		&\inf\{u: \widehat{S}_u>t\},\quad \text{if }t<\widehat{S}_{\tau-},\\
		&\infty,\qquad\qquad\qquad\;\;\, \text{if }t\geq \widehat{S}_{\tau-}. 
	\end{aligned}
	\right.
\]
We let 
\[
	\check{X}_t:=\widehat{W}_{\widehat{T}_t},\; t\geq 0,\quad \check{\zeta}:=\widehat{S}_{\tau-}.  
\]
Then $\check{X}$ is the time-changed process of $\widehat{W}$ by the PCAF $\widehat{S}$ and $\check{\zeta}$ is its lifetime; see,  e.g.,  \cite[Theorem~A.3.9]{CF12}.  

We need to show the identification of $X$ and $\check{X}$.  In fact,  for $u<\tau$,  we have $\widehat{S}_u=S_u$.  It follows from Lemma~\ref{LM515} that $\zeta=S_{\tau-}=\widehat{S}_{\tau-}=\check{\zeta}$.  For $t<\zeta=\check{\zeta}$,  since $T_t=\inf\{u: S_u>t\}<\tau$ as obtained in Lemma~\ref{LM515},  it follows that
\[
	\widehat{T}_t=\inf\{u: \widehat{S}_u>t\}=\inf\{u<\tau: S_u>t\}=T_t.  
\]
As a result $\check{X}_t=W_{\widehat{T}_t\wedge \tau}=W_{T_t\wedge \tau}=W_{T_t}=X_t$ for $t<\zeta=\check{\zeta}$.   Therefore $X$ and $\check{X}$ are identified.  That completes the proof. 
\end{proof}
\begin{remark}\label{RM73}
When \text{(QK)} is not assumed,  the argument in this proof still holds by replacing $\widehat{W}$ with the absorbing Brownian motion on $(l_0,r_0)$,  where $l_0$ and $r_0$ are defined in \eqref{eq:51-2}.  In other words,  $X$ is the time-changed process of absorbing Brownian motion on $(l_0,r_0)$ by the PCAF corresponding to $\mu$.  In the representation of Dirichlet form associated with $X$,  there appears additional killing term
\[
	\frac{\widehat{f}(l)^2}{2|l_0-l|} \quad \left(\text{resp. }\frac{\widehat{f}(r)^2}{2|r_0-r|} \right)
\]
whenever $-\infty<l_0<l$ (resp.  $r<r_0<\infty$),  
except for $\widehat{\sE}(\widehat{f},\widehat{f})$ expressed as \eqref{eq:321}.  That means $X$ admits killing at $l$ or $r$ if $-\infty<l_0<l$ or $r<r_0<\infty$,  and the killing ratio is $1/(2|l-l_0|)$ or $1/2(|r-r_0|)$.  When $l_0\downarrow -\infty$ or $r_0\uparrow \infty$,  the killing disappears and when $l_0\uparrow l$ or $r\downarrow r_0$,  the killing occurs so frequently that $l$ or $r$ becomes an absorbing endpoint.  
\end{remark}

\subsubsection{Quasidiffusion as a skip-free Hunt process}\label{SEC712}

As noted in \S\ref{SEC6},  a quasidiffusion $X$ is a standard process on $E$ satisfying (SF).  Due to the argument in \S\ref{SEC711},  it is the time-changed process of $\widehat{W}$ by the PCAF corresponding to $\mu$.  Note that the quasi support of $\mu$ with respect to $\widehat{W}$ coincides with its topological support $\text{supp}[\mu]=E$ in $V$.  In view of \cite[Corollary~5.2.10]{CF12},  the associated Dirichlet form of $X$ on $L^2(E,\mu)$ is regular.  Particularly,  $X$ is actually a Hunt process on $E$.  In addition,  applying \cite[Theorems~5.2.8~(2) and Theorem~3.5.6~(1)]{CF12},  we can obtain (SR) for $X$.   On account of \cite[Theorem~3.1.10]{CF12},  a semipolar set for $X$ must be empty.  Thus Hunt's hypothesis holds for $X$.  Since $\widehat{X}$ in (i) admits no killing inside as formulated in Theorem~\ref{THM6} and $X$ is equivalent to $\widehat{X}$,  it follows that (SK) holds true for $X$.  Therefore we can eventually conclude that $X$ is a skip-free Hunt process on $E$.  

The fact that $X$,  as a skip-free Hunt process,  is on its natural scale can be seen by virtue of \cite[(6,7)]{K86}.  In what follows we prove that the speed measure of $X$, as a skip-free Hunt process, is equal to $\mu$.  


\begin{proof}
Adopt the same notations as in \S\ref{SEC53}. 

Let us first treat the case $l,r\notin E$.  Take $a,b\in E$ with $a<b$.  Set $T:=T_a\wedge T_b$ and let $h:=h_{a,b}$ defined as in \S\ref{SEC531}. 
We assert that for any $x\in E$ with $a<x<b$,  
\begin{equation}\label{eq:72}
M_t:=h(X_{t\wedge T})+t\wedge T=\mathbf{E}_x(T|\sF_t)
\end{equation} 
is a $\mathbf{P}_x$-martingale.  In fact,  $T<\infty$,  $\mathbf{P}_x$-a.s.  and we have
\[
M_t=\mathbf{E}_x \left(T\circ \theta_{t\wedge T}|\sF_{t\wedge T} \right)+t\wedge T=\mathbf{E}_x (T|\sF_{t\wedge T})=\mathbf{E}_x(T|\sF_t).  
\]	
Hence \eqref{eq:72} holds.  Note that $h$ is concave as proved in Lemma~\ref{LM511}.  Applying It\^o-Tanaka-Meyer formula \eqref{eq:516} to $h$,  we get that
\begin{equation}\label{eq:71}
\begin{aligned}
h&(X_{t\wedge T})-h(X_0)=\int_0^{t\wedge T}h'_-(X_{s-})dX_s  \\
&+\sum_{0<s\leq t\wedge T} \left[h(X_s)-h(X_{s-})-h'_-(X_{s-})\Delta X_s\right]+\frac{1}{2}\int_\bR L(t\wedge T,  x)h''(dx).  
\end{aligned}
\end{equation}
In view of Remark~\ref{RM64}~(2),  the first term on the right hand side is a local martingale.  Note that $\Delta X_s\neq 0$ if and only if $X_s=a_k, X_{s-}=b_k$ or $X_s=b_k,  X_{s-}=a_k$ for some $k\in K'$,  where $K'$ is given before Theorem~\ref{THM517}.  Using the expression of $h$,  one can compute that the second term on the right hand side of \eqref{eq:71} is equal to 
\[
\sum_{k\in K'}(h'_+(a_k)-h'_-(a_k))(b_k-a_k) A_k(t\wedge T).  
\]
By virtue of Theorem~\ref{THM517} and Lemma~\ref{LM65},  it follows from \eqref{eq:72} and \eqref{eq:71} that
\begin{equation}\label{eq:73}
	t\mapsto \int_{E} \ell(t\wedge T,x)h''(dx)+t\wedge T
\end{equation}
is a local martingale.   On account of $T_t<\tau$ for $t\leq T<\zeta$,  we have that $\ell(t\wedge T, x)\leq \ell^W(T_t, x)<\infty$ for $x\in (a,b)$ and $\ell(t\wedge T, x)=0$ for $x\notin (a,b)$.  Note that $-h''((a,b))<\infty$.  Hence \cite[Corollary~2.10]{RY99} yields that \eqref{eq:73} is also a continuous process locally of bounded variation and must be equal to $0$ for all $t\geq 0$.  Letting $a\downarrow l$ and $b\uparrow r$ we get that
\begin{equation}\label{eq:74}
	-2\int_E \ell(t,x)\tilde{\mu}(dx)+t=0,\quad 0\leq t<\zeta,
\end{equation}
where $\tilde{\mu}$ defined as \eqref{eq:510} is the speed measure of $X$ as a skip-free Hunt process.  
Further let $t:=S_u<\zeta$ for $u<\tau$.  Since $W_v\notin E$ for $v\in (u, T_{t})$ if $u<T_t$,  it follows that $\ell(t,x)=\ell^W(T_t,x)=\ell^W(u,x)$.  Thus \eqref{eq:74} reads as
\begin{equation}\label{eq:75}
	2\int_E \ell^W(u\wedge \tau,x)\tilde{\mu}(dx)=S_u=2\int_E\ell ^W(u\wedge \tau ,x)\mu(dx),\quad u\geq 0,
\end{equation}
where the second identity is due to \eqref{eq:61}.  Note that the first  (resp.  second) integration in \eqref{eq:75} is a PCAF of $\widehat{W}$ whose Revuz measure is $\tilde{\mu}$ (resp.  $\mu$).  Therefore the uniqueness of Revuz correspondence leads to $\mu=\tilde{\mu}$.  

Next consider the case $l\in E,  r\notin E$.  Take $l< b\in E$ and set $T:=T_b$.  Let $h:=h_b$ defined as in \S\ref{SEC532}.  Then \eqref{eq:72} still holds for $x\in E$ with $x<b$ and particularly,  $M_t$ is a $\mathbf{P}_x$-martingale.   Applying the It\^o-Tanaka-Meyer formula to this $h$ we still get \eqref{eq:71}.  But now the first term on the right hand side of \eqref{eq:71} is the sum of a local martingale and
\[
	\int_0^{t\wedge T}h'_-(X_{s-})d\ell(s,l).  
\]
Note that $\ell(\cdot, l)$ increases only at $X_{s}=l$ or $X_{s-}=l$.  If $X_{s}=l$ but $X_{s-}\neq l$,  then $l=a_k<b_k=X_{s-}$ for some $k\in K'$ due to (SF).  The number of these times is finite before $t\wedge T$.  Since $\ell(\cdot, l)$ is continuous,  it follows that
\[
	\int_0^{t\wedge T}h'_-(X_{s-})d\ell(s,l)=h'_-(l)\ell(t\wedge T,l)=0.
\]
Repeating the argument for the first case,  we can still obtain \eqref{eq:75} for $\tilde{\mu}$ defined as \eqref{eq:513}.  Eventually $\tilde{\mu}=\mu$ can be concluded.  

Another two cases $l\notin E, r\in E$ and $l,r\in E$ can be argued analogously.  The proof is completed.  
\end{proof}

\subsubsection{Skip-free Hunt process and time-changed Brownian motion}

Finally let $X$ be a skip-free Hunt process on $E$ on its natural scale whose speed measure is equal to $\mu$.  We are to prove that $X$ is equivalent to the time-changed Brownian motion $\widehat{X}$ by the PCAF corresponding to $\mu$.  

\begin{proof}
As proved in \S\ref{SEC711} and \S\ref{SEC712},  $\widehat{X}$ is also a skip-free Hunt process on $E$ on its natural scale whose lifetime is denoted by $\widehat{\zeta}$.  Set $\widehat{T}_x:=\inf\{t>0:\widehat{X}_t=x\}$ for any $x\in E$.  

For any non-empty compact set $K\subset E\cup \{\partial\}$,  define the hitting distributions for $X$ and $\widehat{X}$ as follows: For $A\subset K$, 
\[
	P_K(x,A):=\mathbf{P}_x(X_{T_K}\in A, T_K<\zeta),\quad \widehat{P}_K(x,A):=\mathbf{P}_x(\widehat{X}_{\widehat{T}_K}\in A,  \widehat{T}_K<\widehat{\zeta}),
\]
where $T_K:=\inf\{t>0:X_t\in K\}$ and $\widehat{T}_K:=\inf\{t>0:\widehat{X}_t\in K\}$.  We assert that 
\begin{equation}\label{eq:77}
	P_K(x,\cdot)=\widehat{P}_K(x,\cdot),\quad \forall x\in E.
\end{equation}
Only the case $l\in E,  r\notin E$ will be treated and the others can be argued analogously.  Note that $\partial$ is identified with $r$,  and in view of Lemma~\ref{LM54}~(4), 
\begin{equation*}
\zeta=T_r:=\lim_{b\uparrow r}T_b,\quad \widehat{\zeta}=\widehat{T}_r:=\lim_{b\uparrow r}\widehat{T}_b.  
\end{equation*}
 When $x\in K$,  $T_K\leq T_x=0$ and $\widehat{T}_K\leq \widehat{T}_x=0$.  Hence $P_K(x,\cdot)=\widehat{P}_K(x,\cdot)=\delta_x$.  When $x\notin K$,  set $K_-:=\{y\in K: y<x\}$ and $K_+:=\{y\in K: y>x\}$.  We verify \eqref{eq:77} for different cases as follows:
\begin{itemize}
\item[(1)] $K_-\neq \emptyset,  K_+=\emptyset$: Set $a:=\sup\{y: y\in K_-\}$.  On account of Lemma~\ref{LM54}~(2),  $T_K=T_a$ and $\widehat{T}_K=\widehat{T}_a$.  Then both $P_K(x,\cdot)$ and $\widehat{P}_K(x,\cdot)$ are concentrated on $\{a\}$.  It follows from Theorem~\ref{THM58} that
\[
P_K(x,\{a\})=\mathbf{P}_x(T_a<\zeta)=\lim_{b\uparrow r}\mathbf{P}_x(T_a<T_b)=\lim_{b\uparrow r}\frac{b-x}{b-a},
\]
which is equal to $(r-x)/(r-a)$ if $r<\infty$ and $1$ if $r=\infty$.  Analogously one can also compute that 
\[
	\widehat P_K(x,\{a\})=\lim_{b\uparrow r}\frac{b-x}{b-a}.
\]
Thus \eqref{eq:77} is verified.
\item[(2)] $K_-=\emptyset,  K_+\neq \emptyset$: Set $b:=\inf\{y: y\in K_+\}$.  If $b=r$,  then $T_K=\zeta$ and $\widehat{T}_K=\widehat{\zeta}$.  Hence $P_K(x,\cdot)=\widehat{P}_K(x,\cdot)\equiv 0$.  If $b<r$,  then $T_K=T_b$ and $\widehat{T}_K=T_b$.  Both $P_K(x,\cdot)$ and $\widehat{P}_K(x,\cdot)$ are concentrated on $\{b\}$.  Note that $\mathbf{P}_x(T_b<\infty)=\mathbf{P}_x(\widehat{T}_b<\infty)=1$;  see \S\ref{SEC532}.  One can easily obtain that $P_K(x, \{b\})=\widehat{P}_K(x,\{b\})=1$.  Hence \eqref{eq:77} holds true.
\item[(3)] $K_-,K_+\neq \emptyset$: Set $a:=\sup\{y: y\in K_-\}$ and $b:=\inf\{y: y\in K_+\}$.   If $b=r$,  then \eqref{eq:77} has been verified in the case $K_-\neq \emptyset,  K_+=\emptyset$.  Now suppose $b<r$.  Then both $P_K(x,\cdot)$ and $\widehat{P}_K(x,\cdot)$ are concentrated on $\{a,  b\}$.  Applying Theorem~\ref{THM58},  we get that
$P_K(x,\{a\})=\mathbf{P}_x(T_a<T_b)=\frac{b-x}{b-a}$ and $P_K(x,\{b\})=\frac{x-a}{b-a}$.  Similarly $\widehat P_K(x,\{a\})=\frac{b-x}{b-a}$ and $\widehat P_K(x,\{b\})=\frac{x-a}{b-a}$.  Therefore \eqref{eq:77} is verified.  
\end{itemize}

With \eqref{eq:77} at hand,  we apply the Blumenthal-Getoor-McKean theorem (see,  e.g.,  \cite[Chapter V,  Theorem~5.1]{BG68}) to obtain a CAF $\widehat A=(\widehat A_t)_{t\geq 0}$ of $\widehat{X}$ which is strictly increasing and finite on $[0,\widehat{\zeta})$ such that $X$ is equivalent to the time-changed process of $\widehat X$ by $\widehat A$.  In view of \cite[Chapter V, Theorem~2.1]{BG68} (The symmetric measure $\mu$ of $\widehat{X}$ is clearly a reference measure in the sense of \cite[Chapter V, Definition~1.1]{BG68} due to \cite[Chapter V, Proposition~1.2]{BG68}),  $\widehat{A}$ is a perfect CAF in the sense of \cite[Chapter IV,  Definition~1.3]{BG68}.  Hence $\widehat{A}$ is a PCAF of $\widehat{X}$ in the sense of,  e.g.,  \cite[Definition~A.3.1]{CF12}.  Denote the Revuz measure of $\widehat{A}$ by $\widehat{\mu}$.  Since $\widehat{A}$ is strictly increasing and the polar set for $\widehat{X}$ must be empty,  it follows that the quasi support of $\widehat{\mu}$ is identified with its topological support $\text{supp}[\widehat{\mu}]=E$.  

We prove that $\widehat{\mu}$ is a Radon measure on $E$.  In fact,  there is an $\widehat{\sE}$-nest $\{\widehat{F}_n:n\geq 1\}$ such that $\widehat{\mu}(\widehat{F}_n)<\infty$.  For any compact set $K\subset E$,   it follows from Lemma~\ref{LM310} that $K\subset \widehat{F}_n$ for some $n$.  Particularly,  $\widehat{\mu}(K)<\infty$.   

Now applying \cite[Corollary~5.2.12]{CF12},  $X$,  as a time-changed process of $\widehat{X}$ by $\widehat{A}$,  is also a time-changed Brownian motion by the PCAF corresponding to $\widehat{\mu}$.  Repeating the argument in \S\ref{SEC712},  we can conclude that its speed measure in the sense of skip-free Hunt process is also equal to $\widehat{\mu}$.  Therefore $\widehat{\mu}=\mu$ and particularly,  $X$ is equivalent to $\widehat{X}$.  That completes the proof.  
\end{proof}

\subsection{Existence and uniqueness}

As a corollary of Theorem~\ref{THM71},  we present the correspondence between a skip-free Hunt process and its scale function and speed measure.  The uniqueness in this corollary means that all processes satisfying the desirable conditions are equivalent.  

\begin{corollary}\label{COR64}
\begin{itemize}
\item[(1)] Let $E\in \mathscr K$ and $\mu$ be a fully supported positive Radon measure on $E$.  Then there exists a unique skip-free Hunt process on $E$ on its natural scale whose speed measure is $\mu$.  
\item[(2)] Let $E\in \overline{\mathscr K}$,  $\mu$ be a fully supported positive Radon measure on $E$ and $\bs$ be a strictly increasing and continuous real valued function on $E$.  Then there exists a unique skip-free Hunt process on $E$ with scale function $\bs$ and speed measure $\mu$. 
\end{itemize}
\end{corollary}
\begin{proof}
To prove the first assertion,  note that a time-changed Brownian motion is uniquely determined by its speed measure,  because so is its associated Dirichlet form.  Hence the uniqueness holds true by means of Theorem~\ref{THM71}.  It suffices to show the existence.  In view of the equivalence between skip-free Hunt process and quasidiffusion,  we only need to find a function $m$ as in \S\ref{SEC6} such that $E_m=E,  m|_E=2\mu$ and \text{(QK)} holds for $m$.  To accomplish this,  assume without loss of generality that $0\in (l,r)$ where $l$ and $r$ are endpoints of $E$.  Set
\begin{equation}\label{eq:78}
	m(x):=2\mu( (0,x]),\; x\in [0, r),\quad m(x):=-2\mu((x,0]),\; x\in [l,0).
\end{equation}
When $r<\infty$,  set
\[
\begin{aligned}
	&m(x)=\infty,  \;  x\geq r, \qquad\qquad  \text{if }r\notin E, \\
	&m(x)=2\mu((0,r]),  \;  x\geq r, \quad \text{if }r\in E. 
\end{aligned}
\]
When $l>-\infty$,  define
\[
\begin{aligned}
	&m(x)=-\infty,  \;  x< l,   \qquad\qquad  \text{if }l \notin E, \\
	&m(x)=-2\mu([l,0]),  \;  x< l, \quad\; \text{if }l\in E. 
\end{aligned}
\]
It is straightforward to verify that $m$ satisfies all these conditions.  Therefore the first assertion is concluded.  The second assertion can be obtained by applying the first one to $\tilde{E}:=\bs(E)$ and $\tilde{\mu}:=\mu\circ \bs^{-1}$.  That completes the proof.  
\end{proof}
\begin{remark}
The desirable function $m$ in this proof is unique (up to a constant) for a given pair $(E,\mu)$.  However if \text{(QK)} is not required,  then the uniqueness of $m$ does not hold.  To obtain $m$ uniquely without assuming \text{(QK)},  one need an additional function $k: \{l,r\}\cap E\rightarrow [0,\infty)$ to determine the killing ratio at $l\in E$ or $r\in E$.  More precisely,  let $m$ be defined as \eqref{eq:78} for $x\in [l,r)$.  When $r\in E$,  set
\[
	m(x):=2\mu((x,r]),\; x\in \left[r,  r+\frac{1}{2k(r)}\right),\quad m(x):=\infty,  \; x\geq r+\frac{1}{2k(r)}.
\]
When $r\notin E$ and $r<\infty$,  set $m(x):=\infty$ for $m\geq r$.  The definition of $m(x)$ for $x\leq l$ is analogical.  This $m$ leads to the unique quasidiffusion for the given triple $(E,\mu,k)$.  
\end{remark}

Theorem~\ref{THM71} also tells us that the image regularized Markov process associated with a triple $(I,\bs,\fm)$ as in \S\ref{SEC2} is always a skip-free Hunt process on its natural scale (as well as a quasidiffusion).  However different triples may have the same image regularization.   Thus there does not exist a one-to-one correspondence between $(I,\bs,\fm)$ and  skip-free Hunt process (or quasidiffusion).  Instead we have the following.  

\begin{corollary}
Let $E\in \mathscr K$ and $\mu$ be a fully supported positive Radon measure on $E$.  Further let $X$ be a  skip-free Hunt process on $E$ on its natural scale whose speed measure is $\mu$. Then there exists a triple $(I,\bs,\fm)$ such that the image regularized Markov process $\widehat{X}$ associated with it is equivalent to $X$.  
\end{corollary}
\begin{proof}
Let $I:=[l,r]$ where $l$ and $r$ are endpoints of $E$.  Write $I\setminus \overline{E}$ as a union of disjoint open intervals
\[
	I\setminus \overline{E}=\cup_{k\geq 1}(a_k,b_k),
\]
where $\overline{E}:=E\cup\{l,r\}$.  Define $\bs(x):=x$ for $x\in E$ and $\bs(x):=\bs(a_k)$ for $x\in (a_k,b_k)$ and $k\geq 1$.  Further set $\fm|_E:=\mu$,  $\fm|_{(l,r)\setminus E}:=0$,  and $\fm(\{l\})=\infty$ (resp.  $\fm(\{r\})=\infty$)  if $l\notin E$ (resp.  $r\notin E$).  It is easy to verify that $(I,\bs,\fm)$ satisfies (DK) and (DM),  and the image regularization of $(I,\bs,\fm)$ is $\widehat{I}=E$,  $\widehat{\bs}(x)=x$ and $\widehat{\fm}=\mu$.  Applying Theorem~\ref{THM71} we can conclude that the image regularized Markov process $\widehat{X}$ associated with $(I,\bs,\fm)$ is equivalent to $X$.  That completes the proof. 
\end{proof}


\subsection{Boundary classification in Feller's sense}\label{SEC73}

Let $X$ be a  skip-free Hunt process on $E\in \mathscr K$ on its natural scale whose speed measure is $\mu$ as in Theorem~\ref{THM71}.  
We turn to classify the endpoints $l$ and $r$ for $X$ in Feller's sense.  The endpoints can be classified for a general  skip-free Hunt process or the regularized Markov process $X^*$ associated with $(I,\bs,\fm)$ accordingly.  

Without loss of generality assume that $0\in (l,r)$.  
Set for ${x}\in ({l},{r})$,  
\[
	{\sigma}({{x}}):=\left\lbrace
	\begin{aligned}
	&\int_0^{{x}} {\mu}\left((0, {y}] \right) d{y}, \quad {x}\geq 0,\\
	& \int_{{x}}^0 {\mu}(({y},0])d{y},\quad {x}<0;
	\end{aligned} \right.  \quad 
	{\lambda}({{x}}):=\left\lbrace
	\begin{aligned}
	&\int_{(0,  {x}]} {y} {\mu}(d{y}),\quad {x}\geq 0,\\
	&\int_{({x},0]}(-{y}) {\mu}(d{y}),\quad {x}<0. 
	\end{aligned}
	\right.
\]
For ${j}={l}$ or ${r}$,  define ${\sigma}({j}):=\lim_{{x}\rightarrow {j}}{\sigma}({x})$ and ${\lambda}({j}):=\lim_{{x}\rightarrow {j}}{\lambda}({x})$. 
The following classification in Feller's sense is very well known.

\begin{definition}
The endpoint ${r}$ (resp.  ${l}$) for $X$ is called
\begin{itemize}
\item[(1)] \emph{regular},  if ${\sigma}({{r}})<\infty,  {\lambda}({{r}})<\infty$ (resp.  ${\sigma}({{l}})<\infty,  {\lambda}({{l}})<\infty$); 
\item[(2)] \emph{exit},  if ${\sigma}({{r}})<\infty,  {\lambda}({{r}})=\infty$ (resp.  ${\sigma}({{l}})<\infty,  {\lambda}({{l}})=\infty$); 
\item[(3)] \emph{entrance},  if ${\sigma}({{r}})=\infty,  {\lambda}({{r}})<\infty$ (resp.  ${\sigma}({{l}})=\infty,  {\lambda}({{l}})<\infty$); 
\item[(4)] \emph{natural},  if ${\sigma}({{r}})= {\lambda}({{r}})=\infty$ (resp.  ${\sigma}({{l}})= {\lambda}({{l}})=\infty$).  
\end{itemize}
\end{definition}
\begin{remark}
Note that ${\sigma}({{r}}),  {\lambda}({{r}})<\infty$ amounts to that ${r}$ is finite and ${\mu}((0,{r}))<\infty$.   When $r\in E$,  $r$ is clearly regular.  In accord with Definition~\ref{DEF21},  we call $r$ absorbing (resp.  reflecting)  if $r\notin E$ (resp. $r\in E$) is regular.  If $r$ is exit,  then $r<\infty$ and $\mu((0,r))=\infty$.  If $r$ is entrance,  then $r=\infty$ and $\mu((0,r))<\infty$.  If $r$ is natural,  then $r+\mu((0,r))=\infty$.  The left endpoint $l$ can be argued similarly. 
\end{remark}

For $j=l$ or $r$,  recall that $T_j:=\inf\{t>0:X_t=j\}$ when $j\in E$ and $T_j:=\lim_{x\rightarrow j}T_x$ when $j\notin E$.  The endpoint $j=l$ or $r$ is called \emph{inaccessible} if $\mathbf{P}_x(T_j<\infty)=0$ for all (equivalently,  one) $x\in (l,r)\cap E$.  Otherwise it is called \emph{accessible}.  The following result characterizes this property by means of the parameter $\sigma$.  Particularly,  $l$ or $r$ is inaccessible,  if and only if it is entrance or natural. 

\begin{proposition}\label{PRO69}
The endpoint $l$ (resp.  $r$) is inaccessible,  if and only if $\sigma(l)=\infty$ (resp.  $\sigma(r)=\infty$).  Particularly,  $X$ is conservative if and only if $\sigma(l)=\infty$ whenever $l\notin E$ and $\sigma(r)=\infty$ whenever $r\notin E$.  
\end{proposition}
\begin{proof}
We only prove that $r$ is accessible if and only if $\sigma(r)=\infty$.  When $r\in E$,  we have noted that $r$ is regular and hence $\sigma(r)<\infty$.  On the other hand,  (SR) implies that $r$ is accessible.  

Now suppose $r\notin E$.  
Take $l_1\in E$ such that $l<l_1<0$.  Fix a constant $\alpha>0$ and set for $x\in (l_1,r)$, 
\[
	u^0({x})\equiv 1,\quad u^{n}({x}):=\int_0^{ x} \int_{(0, y]} u^{n-1}(\xi){\mu}(d\xi)d{y},\; n\geq 1,
\]
where $\int_0^{{x}}:=-\int_{{x}}^0$ and $\int_{(0, y]}:=-\int_{({y},0]}$ for ${x}<0$ and ${y}<0$.  
Define 
\[
u({x}):=\sum_{n=0}^\infty \alpha^n u^n({x}),\quad u_-({x})=u({x})\int_{l_1}^{{x}} u({y})^{-2}d{y},\quad {x}\in (l_1,r). 
\]
Note that $\int_{{l_1}}^{{r}}u({y})^{-2}d{y}<\infty$ and $1+\alpha {\sigma}({x})\leq u({x})\leq \exp\{\alpha {\sigma}({x})\}$ for $x\in (l_1,r)$; see,  e.g.,  \cite[II, \S2]{M68}.  
 Thus $u_-(r):=\lim_{x\uparrow {r}} u_-(x)<\infty$ if and only if  ${\sigma}({r})<\infty$.
 
 Using (SF) and mimicking the arguments in \cite[Theorems~3.1 and 4.1]{L22},  we can obtain that for a fixed $x\in (l_1,r)\cap E$,  
 \[
 	\frac{u_-(x)}{u_-(r_1)}=\mathbf{E}_x\left(e^{-\alpha T_{r_1}}; T_{l_1}>T_{r_1} \right),\quad x<r_1\in E;
 \]
 see also \cite[(7)]{K86}.  Letting $r_1\uparrow r$ and applying Lemma~\ref{LM54}~(4),  we get that
\[
	\mathbf{E}_x\left(e^{-\alpha T_{r}}; T_{l_1}>T_{r} \right)=\frac{u_-(x)}{u_-(r)}.  
\]
If $\sigma(r)<\infty$,  then $u_-(r)<\infty$ and particularly $\mathbf{E}_x\left(e^{-\alpha T_{r}}; T_{l_1}>T_{r} \right)>0$.  We must have $\mathbf{P}_x(T_r<\infty)>0$,  because otherwise $T_r=\infty$,  $\mathbf{P}_x$-a.s., leading to $\mathbf{E}_x\left(e^{-\alpha T_{r}}; T_{l_1}>T_{r} \right)=0$.  If $\sigma(r)=\infty$,  then $\mathbf{E}_x\left(e^{-\alpha T_{r}}; T_{l_1}>T_{r} \right)=0$.  Argue by contradiction and suppose $r$ is accessible.  It follows from Corollary~\ref{COR59} that $r<\infty$.  In view of Theorem~\ref{THM58},  we have
\[
	\mathbf{P}_x(T_{l_1}>T_r)=\frac{x-l_1}{r-l_1}>0.  
\]
Note that $T_r<\infty$ on $\{T_{l_1}>T_r\}$.  Hence $\mathbf{E}_x\left(e^{-\alpha T_{r}}; T_{l_1}>T_{r} \right)>0$,  which leads to a contradiction.  

The characterization of conservativeness is a consequence of Lemma~\ref{LM54}~(4).  That completes the proof.  
\end{proof}

\section{Unregularized Markov processes with discontinuous scales}\label{SEC9}

Now we come back to a triple $(I,\bs,\fm)$ defined in \S\ref{SEC2} but assume further that $\bs$ is strictly increasing.  Then (DK) trivially holds and (DM) reads as that $\fm$ is fully supported on $I$ and $\fm(\{x\})>0$ for $x\in D^0$; see Remark~\ref{RM22}.  Recall that $(\sE,\sF)$ defined as \eqref{eq:25} turns to be a Dirichlet form on $L^2(I_0, \fm)$ in the wide sense in Theorem~\ref{LM12}.  

\begin{lemma}\label{LM81}
\begin{itemize}
\item[(1)] If $\bs$ is strictly increasing,  then $(\sE,\sF)$ is a Dirichlet form on $L^2(I_0,\fm)$.
\item[(2)] Assume that $\bs$ is strictly increasing.  Then $(\sE,\sF)$ is regular on $L^2(I_0,\fm)$ if and only if $\bs$ is continuous. 
\end{itemize}
\end{lemma}
\begin{proof}
\begin{itemize}
\item[(1)] Since $\bs$ is strictly increasing,  only the transformation of scale completion need be operated to obtain the regularization $(I^*,\bs^*,\fm^*)$ of $(I,\bs,\fm)$.  Meanwhile $I^*$ can be treated as the union of $I$ and another set of at most countably many points,  $\fm^*|_I=\fm$ and $\fm^*(I^*\setminus I)=0$.  In view of the regularity of $(\sE^*,\sF^*)$ obtained in Theorem~\ref{THM6},  one can easily conclude the denseness of $\sF$ in $L^2(I_0,\fm)$.  
\item[(2)]  We first assert that for $\alpha,\beta\in (l,r)\setminus D$ with $\alpha<\beta$,  there is a constant $C_{\alpha,\beta}$ depending on $\alpha$ and $\beta$ such that
\begin{equation}\label{eq:81}
	\sup_{x\in [\alpha, \beta]}|f(x)|^2\leq C_{\alpha,\beta}\sE_1(f,f),\quad f\in \sF.
\end{equation}
In fact,  write $f=f^c+f^++f^-\in \sF$ and let $g^c=df^c/d\mu_c,  g^\pm=df^\pm/d\mu^\pm_d$. Take $x,y\in [\alpha, \beta]$ with $x<y$.  We have
\[
	|f^c(x)-f^c(y)|^2\leq \mu_c((x,y)) \int_x^y \left(g^c\right)^2d\mu_c \leq 2|\bs(y)-\bs(x)| \sE(f,f).   
\]
Analogous assertion holds for $f^\pm$.  Thus
\[
	|f(x)-f(y)|^2\leq 6|\bs(y)-\bs(x)|\sE(f,f).  
\]
It follows that for any $x,y\in [\alpha,\beta]$, 
\[
f(x)^2\leq 2f(y)^2+2|f(x)-f(y)|^2\leq 2f(y)^2 +12|\bs(\beta)-\bs(\alpha)|\sE(f,f),
\]
and hence for any $y\in [\alpha,\beta]$,
\[
\sup_{x\in [\alpha,\beta]}f(x)^2\leq 2f(y)^2 +12|\bs(\beta)-\bs(\alpha)|\sE(f,f).  
\]
Integrating both sides by $\fm$ on $[\alpha,\beta]$ we arrive at \eqref{eq:81}.  Now we prove the assertion.  If $\bs$ is continuous,  then $(\sE,\sF)=(\sE^*,\sF^*)$,  $I_0=I^*$ and $\fm=\fm^*$.  In view of Theorem~\ref{THM6},  $(\sE,\sF)$ is regular on $L^2(I_0,\fm)$.  To the contrary,  argue by contradiction.  Assume that $\bs$ is not continuous at $x\in (l,r)$ while $(\sE,\sF)$ is regular.   Then we may take a function $f\in \sF$ such that $f$ is not continuous at $x$.  By means of the regularity of $(\sE,\sF)$,  there exists a sequence $f_n\in \sF\cap C_c(I_0)$ such that $\sE_1(f_n-f,f_n-f)\rightarrow 0$.   Take $\alpha,\beta\in (l,r)\setminus D$ such that $\alpha<x<\beta$.   Applying \eqref{eq:81} to $f_n-f$,  we get that $f_n$ converges to $f$ uniformly on $[\alpha,\beta]$.  Hence $f$ is continuous on $[\alpha,\beta]$,  as violates  the discontinuity of $f$ at $x$.  
\end{itemize}
That completes the proof. 
\end{proof}

Consider that $\bs$ is strictly increasing  but not continuous.  This lemma tells us that $(\sE,\sF)$ is not regular on $L^2(I_0,\fm)$.  
In what follows we give an $\fm$-symmetric continuous simple Markov process $\dot X$ on $I_0$ whose Dirichlet form is $(\sE,\sF)$.  Here a simple Markov process means that $\dot X$ satisfies the Markov property but does not satisfy the strong Markov property.  Let
\[
	\widehat X=\left\{\widehat{\Omega},\widehat{\sF}_t, \widehat{X}_t, (\widehat{\mathbf{P}}_{\widehat{x}})_{\widehat{x}\in \widehat{I}}, \widehat \zeta\right\}
\]
be the image regularized Markov process associated with $(I,\bs,\fm)$.   Note that
\begin{equation}\label{eq:80}
	\widehat{I}=\bs(I_0)\cup \{\bs(x+): x\in D^+\}\cup \{\bs(x-):x\in D^-\},
\end{equation}
and the second and third sets on the right hand side are of zero $\widehat\fm$-measure.  Set a function $\mathbf{r}: \widehat{I}\rightarrow I_0$ by $\br((\bs(x)):=x$ for $x\in I_0$ and $\br(\bs(x\pm)):=x$ for $x\in D^\pm$.  Define $\dot \Omega:=\left\{\omega\in \widehat{\Omega}: \widehat{X}_0(\omega)\in \bs(I_0)\right\}\in \widehat\sF_0$ and 
\[
\begin{aligned}
	&\dot\sF_t:=\widehat{\sF}_t\cap \dot \Omega=\{A\cap \dot \Omega: A\in \widehat{\sF}_t\},\quad \dot X_t(\omega):=\br(\widehat{X}_t(\omega)), \; \omega\in \dot\Omega,    \\
	&\dot{\mathbf{P}}_x:=\widehat{\mathbf{P}}_{\bs(x)}|_{\dot \Omega},\; x\in I_0, \quad \dot \zeta(\omega):=\widehat{\zeta}(\omega),\; \omega\in \dot\Omega.  
\end{aligned}\]
Clearly $(\dot\sF_t)_{t\geq 0}$ is a right continuous filtration,  and $(\dot X_t)_{t\geq 0}$ is a family of random variables on $\dot \Omega$ adapted to $(\dot\sF_t)_{t\geq 0}$.  The following result is inspired by \cite[Theorem~3.6]{S79}. 

\begin{theorem}\label{THM82}
Assume that $\bs$ is strictly increasing but not continuous. 
The stochastic process 
\[
	\dot X=\{\dot \Omega,  \dot\sF_t, \dot X_t, (\dot{\mathbf{P}}_x)_{x\in I_0}, \dot\zeta\}
\]
is an $\fm$-symmetric continuous simple Markov process on $I_0$ whose Dirichlet form on $L^2(I_0,\fm)$ is $(\sE,\sF)$.  
\end{theorem}
\begin{proof}
Firstly we note that for any $A\in \dot \sF\subset \widehat{\sF}$,  $\dot{\mathbf{P}}_x(A)=\widehat{\mathbf{P}}_{\bs(x)}(A)$ is Borel measurable in $x$.  Secondly,  let us prove the continuity of all paths of $\dot X$.  The right continuity of all paths is due to that for $\widehat{X}$.  If $\widehat{X}_{t-}=\widehat{X}_t$,  then obviously $\dot{X}_{t-}=\dot{X}_t$.  If $\widehat{X}_{t-}\neq \widehat{X}_t$,  then (SF) for $\widehat{X}$ implies that $(\widehat X_{t-} \wedge \widehat X_t,  \widehat X_{t-}\vee \widehat X_t)=(\widehat{a}_k,\widehat{b}_k)$ for some interval $(\widehat{a}_k,\widehat{b}_k)$ in \eqref{eq:15}.  Note that $\br(\widehat{a}_k)=\br(\widehat{b}_k)$.  It follows that $\dot X_{t-}=\dot X_t$.  Hence $\dot X$ is a continuous stochastic process.

Next we turn to verify the Markov property of $\dot X$.  Denote by $\widehat{P}_t$ the transition functions of $\widehat{X}$,  i.e.  $\widehat{P}_t(\widehat{x},  \widehat{\Gamma})=\widehat{\mathbf{P}}_{\widehat{x}}(\widehat{X}_t\in \widehat{\Gamma})$ for $\widehat{x}\in \widehat{I}$ and $\widehat{\Gamma}\in \mathcal{B}(\widehat{I})$.  As obtained in Theorem~\ref{THM71},  $\widehat{X}$ is equivalent to a quasidiffusion.  Then $\widehat{P}_t$ admits a transition density with respect to $\widehat{\fm}$,  as mentioned in, e.g.,  \cite{K75}.  Particularly,  
\begin{equation}\label{eq:82}
	\widehat{P}_t(\widehat{x},  \widehat{I}\setminus \bs(I_0))=\widehat{\mathbf{P}}_{\widehat{x}}(\widehat{X}_t\in \widehat{I}\setminus \bs(I_0))=0. 
\end{equation}
Define $\dot P_t(x,\Gamma):=\dot{\mathbf{P}}_x(\dot{X}_t\in \Gamma)$ for $t\geq 0$, $x\in I_0$ and $\Gamma\in \mathcal{B}(I_0)$.  By the definition of $\dot X$,  we have $\dot P_t(x,\Gamma)=\widehat{\mathbf{P}}_{\bs(x)}(\br(\widehat{X}_t)\in \Gamma)=\widehat{\mathbf{P}}_{\bs(x)}(\widehat{X}_t\in \br^{-1}\Gamma)$. 
In view of \eqref{eq:82},  one gets 
\begin{equation}\label{eq:83}
	\dot P_t(x,\Gamma)=\widehat{\mathbf{P}}_{\bs(x)}\left(\widehat{X}_t\in \bs(\Gamma)\right)=\widehat{P}_t\left(\bs(x), \bs(\Gamma) \right).  
\end{equation}
To prove the Markov property of $\dot X$,  it suffices to show that for $\Gamma\in \mathcal{B}(I_0)$ and $A\in \dot\sF_t\subset \widehat{\sF}_t$,  
\begin{equation}\label{eq:84}
	\dot{\mathbf{P}}_x\left(\dot X_{t+s}\in \Gamma; A \right)=\int_A \dot P_s(\dot X_t,  \Gamma) d\dot{\mathbf{P}}_x.  
\end{equation}
In fact,  on account of the Markov property of $\widehat{X}$,   the left hand side of \eqref{eq:84} equals
\[
	\widehat{\mathbf{P}}_{\bs(x)}\left(\widehat{X}_{t+s}\in \br^{-1}\Gamma; A \right)=\int_{A} \widehat{P}_s(\widehat{X}_t, \br^{-1}\Gamma) d\widehat{\mathbf{P}}_{\bs(x)}.
\]
Using \eqref{eq:82} and \eqref{eq:83},  we get that $\dot{\mathbf{P}}_x(\dot X_{t+s}\in \Gamma; A)$ equals
\[
	\int_{A \cap \{\widehat{X}_t\in \bs(I_0)\}} \widehat{P}_s(\widehat{X}_t, \bs(\Gamma)) d\widehat{\mathbf{P}}_{\bs(x)}=\int_A \dot P_s(\dot X_t,  \Gamma) d\dot{\mathbf{P}}_x.  
\]
Hence \eqref{eq:84} is concluded. 

Fourthly we derive the symmetry of $\dot{X}$ with respect to $\fm$.  Note that \eqref{eq:83} gives the transition functions of $\dot X$.  Then the symmetry of $\dot X$ can be easily obtained by using the symmetry of $\widehat{X}$ with respect to $\widehat{\fm}$,  \eqref{eq:80} and $\widehat{\fm}=\fm\circ \bs^{-1}$.  In addition,  applying \cite[(1.3.17)]{FOT11} to $\dot P_t$ and noting \eqref{eq:35}, \eqref{eq:38-3},  one can verify that the Dirichlet form of $\dot X$ on $L^2(I_0,\fm)$ is $(\sE,\sF)$.  

Finally it suffices to show that $\dot X$ does not satisfy the strong Markov property.  In fact, otherwise $\dot X$ is a diffusion on $I_0$ which is symmetric with respect to a fully supported Radon measure $\fm$.  By virtue of \cite[Theorem~3]{L21},  its Dirichlet form $(\sE,\sF)$ on $L^2(I_0,\fm)$ must be regular,  as violates the second assertion of Lemma~\ref{LM81} because $\bs$ is not continuous.  That completes the proof. 
\end{proof}
\begin{remark}
When $\bs$ is neither strictly increasing nor continuous,  we may operate only the transformation of darning on $I$,  as stated in \S\ref{SEC31},  to obtain a new ``interval" $I^*$.  Then $(\sE^*,\sF^*)$ defined as \eqref{eq:38-2} is a Dirichlet form on $L^2(I^*,\fm^*)$ but not regular.  Mimicking Theorem~\ref{THM82},  one may still find an $\fm^*$-symmetric continuous simple Markov process on $I^*$ whose Dirichlet form is $(\sE^*,\sF^*)$.  
\end{remark}


\section{Examples}\label{SEC8}

In this section we present various classical Markov processes that can be recovered from Theorem~\ref{THM6}.  

\subsection{Regular diffusions}\label{SEC71}

Consider that $\bs$ is continuous and strictly increasing on $I$ and that $\fm$ is fully supported on $I$.  The assumptions (DK) and \text{(DM)} trivially hold. 

Now neither scale completion nor darning operates on $I$.  The regularization of $(I,\bs,\fm)$ is as follows:
\[
	I^*=I_0, \quad \bs^*=\bs,  \quad \fm^*=\fm|_{I^*},
\]
where $I_0$ is obtained by getting rid of all non-reflecting endpoints from $I$.  
Clearly $(\sE,\sF)=(\sE^*,\sF^*)$.  More precisely, 
\begin{equation}\label{eq:91}
\begin{aligned}
	&\sF^* =\sF=\{f\in L^2(I_0,\fm):f\ll \bs,  df/d\bs\in L^2(I,d\bs),  \\
	&\qquad\qquad \qquad f(j)=0\text{ if } j\notin I_0\text{ and }|\bs(j)|<\infty  \text{ for }j=l\text{ or }r\}, \\
	&\sE^*(f,g)=\sE(f,g)=\frac{1}{2}\int_{I_0}\frac{df}{d\bs}\frac{dg}{d\bs},\quad f,g\in \sF,
\end{aligned}
\end{equation}
where $f\ll \bs$ stands for that $f$ is absolutely continuous with respect to $\bs$.  
The regularized Markov process $X^*$ is a regular diffusion with canonical scale function $\bs$ and canonical speed measure $\fm$ and no killing inside.  We refer readers to,  e.g.,  \cite[\S2.2.3]{CF12} and \cite{F14,  L21, LY19} for related studies. 

\subsection{Snapping out diffusions}

Let $I=[-\infty, \infty]$.  Consider a strictly increasing scale function $\bs$,  which is continuous on $(-\infty,  \infty)\setminus \{0\}$,  right continuous but not left continuous at $0$.  Assume further that $\fm$ is fully supported on $(-\infty,\infty)$ and $\fm(\{-\infty\})=\fm(\{\infty\})=\infty$.  
Although $\bs$ is not continuous at $0$ as assumed in \S\ref{SEC2},  we may take another point in place of $0$.  
Clearly (DK) and \text{(DM)} hold true.  

Only scale completion need be operated and the regularization of $(I,\bs,\fm)$ is
\begin{itemize}
\item[(a$^*$)] $I^*=(-\infty, 0-]\cup [0+,\infty)$;
\item[(b$^*$)] $\bs^*(x)=\bs(x)$ for $x\neq 0$ and $\bs^*(0-)=\bs(0-), \bs^*(0+)=\bs(0)$;
\item[(c$^*$)] $\fm^*|_{I^*\setminus \{0\pm\}}=\fm|_{(-\infty,\infty)\setminus \{0\}}$, $\fm^*(\{0+\})=\fm(\{0\})$ and $\fm^*(\{0-\})=0$.  
\end{itemize} 
(If wants $\fm^*$ having positive mass at both $0+$ and $0-$,  one may take a right continuous $\bs$ such that $\bs(x)=\bs(0)$ for $x\in [0,1)$,  $\bs(1)>\bs(0)$ and $\bs$ is strictly increasing and continuous elsewhere.)  Set
\[
I^*_+:=[0+,\infty),\quad I^*_-:=(-\infty, 0-],\quad  \bs^*_\pm:=\bs^*|_{I^*_\pm},  \quad \fm^*_\pm:=\fm^*|_{I^*_\pm}.
\] 
For $f^*\in L^2(I^*,\fm^*)$,  set $f^*_+:=f^*|_{I^*_+}$ and  $f^*_-=f^*|_{I^*_-}$.
The Dirichlet form $(\sE^*,\sF^*)$ on $L^2(I^*,\fm^*)$ is 
\[
\begin{aligned}
	&\sF^*=\{f^*\in L^2(I^*,\fm^*): f^*_\pm\ll \bs^*_\pm,  df^*_\pm/d\bs^*_\pm \in L^2(I^*_\pm, d\bs^*_\pm),   \\
	&\qquad\qquad \qquad\qquad\qquad  f^*_\pm(\pm\infty)=0\text{ whenever }|\bs^*(\pm \infty)|<\infty \},  \\
	&\sE^*(f^*,f^*)=\frac{1}{2}\sum_{\dagger=+,-}\int_{I^*_\dagger}\left(\frac{df^*_\dagger}{d\bs^*_\dagger}\right)^2d\bs^*_\dagger+\frac{(f^*_+(0+)-f^*_-(0-))^2}{2(\bs^*(0+)-\bs^*(0-))},\quad f^*\in \sF^*. 
\end{aligned}
\]
The regularized Markov process $X^*$ is called a \emph{snapping out diffusion process} as studied in \cite[\S3.4]{LS20}.  It behaves like a regular diffusion on $I^*_\pm$,  and may jump to the dual splitting point and start as a new regular diffusion on another component of $I^*$ by chance when hits $0-$ or $0+$.  

Note that $I_0=\bR$ and $(\sE,\sF)=(\sE^*,\sF^*)$.  In view of \S\ref{SEC9},  $(\sE,\sF)$ is not regular on $L^2(\bR)$ but there is a simple Markov process $\dot X$ whose Dirichlet form on $L^2(\bR)$ is $(\sE,  \sF)$.  

\begin{example}
We consider a more explicit example with $\bs(x)=x$ for $x<0$,  $\bs(x)=\frac{2}{\kappa}+x$ for $x\geq 0$,  and $\fm$ being the Lebesgue measure,  where $\kappa>0$ is a given constant.  Denote by $H^1(J)$ the Sobolev space of order $1$ over an interval $J$.  Then
\[
\begin{aligned}
	&\sF^*=\sF=\left\{f\in L^2(\bR): f|_{(0,\infty)} \in H^1((0,\infty)),  f|_{(-\infty, 0)}\in H^1((-\infty, 0))\right\},   \\
	&\sE^*(f,f)=\sE(f,f)=\frac{1}{2}\int_{\bR\setminus \{0\}}f'(x)^2dx+\frac{\kappa}{4}(f(0+)-f(0-))^2,\quad f\in \sF.  
\end{aligned}\]
Note that $(\sE,\sF)$ is a regular Dirichlet form on $L^2((-\infty, 0-]\cup [0+,\infty))$ and its associated Markov process $X^*$ is called a \emph{snapping out Brownian motion} with parameter $\kappa$,  which was raised in \cite{L16}.    

However $(\sE,\sF)$ is a Dirichlet form but not regular on $L^2(\bR)$.  Theorem~\ref{THM82} gives a continuous simple Markov process $\dot X$ whose Dirichlet form on $L^2(\bR)$ is $(\sE,\sF)$.  Before hitting $0$,  $\dot X$ moves as a Brownian motion,  and the excursions of $\dot X$ at $0$ can be described as follows: There exists a sequence of i.i.d.  random times $\{\tau_n: n\geq 1\}$ such that 
\[
S_n=\tau_0+\cdots+\tau_n\in Z:=\{t: \dot{X}_t=0\}, \quad \forall n\geq 0,
\]
where $\tau_0=\inf\{t>0: \dot X_t=0\}$,  and for $k\in \bN$,  
\begin{itemize}
\item[(1)] When $\dot X_0>0$,  the excursion intervals contained in $(S_{2k}, S_{2k+1})$ are on the right axis while those contained in $(S_{2k+1}, S_{2k+2})$ are on the left axis.
\item[(2)] When $\dot X_0<0$,  the excursion intervals contained in $(S_{2k}, S_{2k+1})$ are on the left axis while those contained in $(S_{2k+1}, S_{2k+2})$ are on the right axis.
\item[(3)] When $\dot X_0=0$,  the above two cases occur with equal probability.  
\end{itemize}
Particularly,  the strong Markov property at $\tau_0$ fails for $\dot X$.  
It is worth pointing out that the sequence of times $\{S_n: n\geq 1\}$ consists of the successive jumping times of snapping out Brownian motion,  and $\tau_n\overset{d}{=} \inf\{t>0: L_t>\xi\}$ is equal to the killing time of an elastic Brownian motion in distribution,  where $L_t$ is the local time of a certain reflecting Brownian motion on $[0,\infty)$ at $0$ and $\xi$ is an independent exponential random variable with parameter $\kappa$.  More details about elastic Brownian motion and snapping out Brownian motion are referred to in \cite{L16,  LS20}.  

\end{example}

\subsection{Random walk in one dimension}\label{SEC43}

Take $p,q\in \bN\cup \{\infty\}$ and an increasing sequence of constants indexed by $\bZ_{-p,q}:=\{-p,1-p,\cdots,  -1, 0, 1,\cdots,  q\}\cap \mathbb{Z}$: $$c_{-p}<\cdots <c_{-1}<c_0<c_1<\cdots<c_q.$$  
Consider the following:
\begin{itemize}
\item[(a)] $I=[-p, q+1]$;
\item[(b)] $\bs(x)=c_k$ for $x\in [k, k+1)$ and $k\in \mathbb{Z}_{-p,q}$.  
\item[(c)] $\fm$ is fully supported, Radon on $I\cap \bR$ and $\fm(\{-\infty\})=\infty$ or $\fm(\{\infty\})=\infty$ whenever $p=\infty$ or $q=\infty$.  

\end{itemize}
Clearly (DK) and \text{(DM)} hold true.  

The regularization of $(I,\bs,\fm)$ is as follows:
\begin{itemize}
\item[(a$^*$)] $I^*$ is identified with the discrete space $\bZ_{-p,q}$.   
\item[(b$^*$)] $\bs^*(k)=c_k$ for $k\in I^*$;
\item[(c$^*$)] $\fm^*(\{k\})=\fm([k, k+1))$ for $k\in \bZ$ with $-p\leq k\leq q-1$.  When $q\in \bN$,  $\fm^*(\{q\})=\fm([q, q+1])$.  
\end{itemize}
Set $\mu_{k,k+1}:=\frac{1}{2(c_{k+1}-c_k)}$ for $k\in \bZ$ with $-p\leq k \leq q-1$.  Write $\mu_{-p-1, -p}=\mu_{q,q+1}=0$ for convenience.  Define $\mu_k:=\mu_{k-1,k}+\mu_{k,k+1}$ for $k\in \bZ$ with $-p\leq k\leq q$.  Then the Dirichlet form $(\sE^*,\sF^*)$ on $L^2(I^*,\fm^*)$ is 
\begin{equation}
\begin{aligned}
&\sF^*=\{f^*\in L^2(I^*,\fm^*): \sE^*(f^*,f^*)<\infty, \\
&\qquad\qquad f^*(j^*)=0\text{ if } |j^*|=\infty, \lim_{k\rightarrow j^*}|c_k|<\infty\text{ for  }j^*=-p\text{ or }q \},\\
&\sE^*(f^*,f^*)=\sum_{-p\leq k\leq q-1} \mu_{k,k+1}\cdot (f^*(k+1)-f^*(k))^2,\quad f^*\in \sF^*. 
\end{aligned}\end{equation}
The regularized Markov process $X^*$ is a continuous time random walk on $I^*$.  More precisely,  for each $k\in I^*$,  $X^*$ waits at $k$ for a (mutually independent) exponential time with mean $\fm^*(\{k\})/\mu_k$ and then jumps to $k-1$ or $k+1$ according to the distribution
\[
	P_k(k-1)=\frac{\mu_{k-1,k}}{\mu_k},\quad P_k(k+1)=\frac{\mu_{k,k+1}}{\mu_k},\quad P_k(j)=0,\; j\neq k-1,  k+1. 
\]
When $\fm^*(\{k\})=\mu_k$,  $X^*$ is called \emph{in the constant speed}.  That means the holding exponential times are independent and identically distributed; see,  e.g.,  \cite[\S2.1]{K14}.  

Particularly,  when $p=q=\infty$, $c_k=k$ for $k\in \bZ$,  and $\fm|_\bR$ is the Lebesgue measure,  $X^*$ is identified with the continuous time simple random walk on $\bZ$.

\subsection{Birth-death processes}

Take a birth-death $Q$-matrix
\[
Q=\left(\begin{array}{ccccc}
-q_0 &  b_0 & 0 & 0 &\cdots \\
a_1 & -q_1 & b_1 & 0 & \cdots \\
0 & a_2 & -q_2 & b_2 & \cdots \\
\cdots & \cdots &\cdots &\cdots &\cdots 
\end{array}  \right),
\]
where $a_k>0$,  $k \geq 1$,  $b_k>0$,  $k\geq 0$ and $q_k=a_k+b_k$ (Set $a_0=0$ for convenience).  We are to obtain a birth-death process  corresponding to $Q$,  also called a \emph{birth-death $Q$-process} for simplicity,  on $\bN$ by virtue of Theorem~\ref{THM6}.  The terminologies about birth-death $Q$-processes are referred to in \cite{C04}.  

To do this,  consider a triple $(I,\bs,\fm)$ as in \S\ref{SEC43} and assume further that $p=0, q=\infty$ and
\[
\begin{aligned}
	&c_0=0,  \quad c_1=\frac{1}{2b_0}, \quad c_k=\frac{1}{2b_0}+\sum_{i=2}^k \frac{a_1a_2\cdots a_{i-1}}{2b_0b_1\cdots b_{i-1}}, \;k\geq 2; \\
	&\fm([0,1))=1,  \quad \fm([k,k+1))=\frac{b_0b_1\cdots b_{k-1}}{a_1a_2\cdots a_k}, \; k\geq 1,\quad \fm(\{\infty\})=\infty.  
\end{aligned}\]
A straightforward computation yields that 
\[
\begin{aligned}
	&\mu_{0,1}=\mu_0=b_0,  \quad \mu_{k,k+1}=\frac{b_0 b_1\cdots b_{k}}{a_1a_2\cdots a_{k}},\mu_k=\frac{b_0 \cdots b_{k-1} q_{k}}{a_1a_2\cdots a_{k}},\;  k\geq 1,  \\
	&\fm^*(\{0\})=1,\quad \fm^*(\{k\})=\frac{b_0b_1\cdots b_{k-1}}{a_1a_2\cdots a_k}, \; k\geq 1.  
\end{aligned}\]
The regularized Markov process $X^*$ waits at $k\geq 0$ for an exponential time with mean $1/q_k$ and then jumps to $k-1$ or $k+1$ with probability $a_k/q_k$ or $b_k/q_k$.  In other words,  $X^*$ is a birth-death $Q$-process.   

There may exist other birth-death $Q$-processes that are even not right processes.  In fact,  $X^*$ is the so-called \emph{minimal birth-death $Q$-process} in the sense of \cite[Theorem~2.21]{C04} and $(\sE^*,\sF^*)$ is the so-called \emph{minimal Dirichlet form}; see \cite[\S6.8]{C04}.  Define
\[
\begin{aligned}
	&\overline{\sF}^*:=\{f^*\in L^2(I^*,\fm^*): \sE^*(f^*,f^*)<\infty\},\\
	&\overline{\sE}^*(f^*,f^*):=\sE^*(f^*,f^*),\quad f^*\in \overline{\sF}^*,
\end{aligned}
\]
which is called the \emph{basic Dirichlet form} in \cite[Definition~6.53]{C04}.  Although is not regular on $L^2(I^*,\fm^*)$ if $\overline{\sF}^*\neq \sF^*$,  $(\overline{\sE}^*,\overline{\sF}^*)$ corresponds to another birth-death $Q$-process,  a maximal symmetric one in the sense of \cite[Theorem~6.61]{C04}: For any other $\fm^*$-symmetric birth-death Q-process whose Dirichlet form is $(\sE',\sF')$,  it holds
\[
\sF^*\subset \sF'\subset \overline{\sF}^*,\quad \sE'(f,f)=\sE^*(f,f),\;\forall f\in \sF^*.  
\] 
The construction of other birth-death $Q$-processes is referred to in,  e.g.,  \cite[Chapter 17]{WZ18}.  Instead we present a corollary for the uniqueness of birth-death $Q$-processes.

\begin{corollary}
\begin{itemize}
\item[\rm (1)] The symmetric birth-death $Q$-processes are unique,  if and only if $\infty$ is not regular,  i.e.  
\[
\text{either }\lim_{k\rightarrow \infty} c_k=\infty \text{ or }\sum_{k\geq 0} \fm^*(\{k\})=\infty.  
\]
\item[\rm (2)] Each of the following conditions is equivalent to the uniqueness of birth-death $Q$-processes:
\begin{itemize}
\item[\rm (i)] $X^*$ is conservative;
\item[\rm (ii)] $\infty$ is inaccessible for $X^*$;
\item[\rm (iii)] $\sum_{k\geq 0}|c_{k+1}-c_k| \sum_{i=0}^k \fm^*(\{i\})=\infty$.  
\end{itemize}
\end{itemize}
\end{corollary}
\begin{proof}
The first assertion is due to \cite[Corollary~6.62]{C04}.  Now we prove the second one.  The equivalence between (i) and (ii) is clear because $X^*$ has no killing inside $\bN$.  The equivalence between (ii) and (iii) is a consequence of Proposition~\ref{PRO69}.  The equivalence between the uniqueness of birth-death $Q$-processes and (iii) can be obtained by means of \cite[Corollary~3.18]{C04}.  That completes the proof.  
\end{proof}
\begin{remark}
When $\infty$ is exit (equivalently,  inaccessible but not regular),  the symmetric birth-death $Q$-processes are unique but there exist other birth-death $Q$-processes that are not symmetrizable.  
\end{remark}

\subsection{Time-changed Brownian motions related to regular Dirichlet subspaces}\label{SEC45}

Let $K\subset [0,1]$ be a generalized Cantor set (see, e.g.,  \cite[page 39]{F99}) and write $[0,1]\setminus K$ as a disjoint union of open intervals:
\begin{equation}\label{eq:42}
	[0,1]\setminus K=\cup_{k\geq 1}(a_k,b_k).  
\end{equation}
Assume that  the Lebesgue measure of $K$ is positive. 
Consider the following triple:
\begin{itemize}
\item[(a)]  $I=[0,1]$;
\item[(b)]  $\bs(x)=x$ for $x\in K$ and $\bs(x):=a_k$ for $x\in (a_k,b_k)$ and $k\geq 1$;
\item[(c)] $\fm(dx)=1_K(x)dx$. 
\end{itemize}
It is easy to verify that (DK) and \text{(DM)} hold true. 

The transformation of scale completion in \S\ref{SEC31} divides each $b_k$ into $\{b_k+, b_k-\}$ and then darning transformation collapses $[a_k, b_k-]$ into an abstract point $p^*_n$.  By regarding $p^*_n$ and $b_k+$ as $a_k$ and $b_k$ respectively,  one may identify the regularization $(I^*,\bs^*,\fm^*)$ with its image $(\widehat{I}, \widehat{\bs},  \widehat{\fm})$.  More precisely,
\[
	I^*=\widehat{I}=K,\quad \bs^*(x)=\widehat{\bs}(x)=x,\; x\in K,\quad \fm^*=\widehat{\fm}=\fm.  
\]
Let 
\begin{equation}\label{eq:43}
\sS^*:=\left\{f^*=h^*|_{K}: h^*\text{ is absolutely continuous on }(0,1)\text{ and }\frac{dh^{*}}{dx} \in L^2((0,1))\right\}. 
\end{equation}
The Dirichlet form $(\sE^*,\sF^*)$ on $L^2(I^*,\fm^*)$ is expressed as
\[
\begin{aligned}
&\sF^*=L^2(I^*,\fm^*)\cap \sS^*,\\
&\sE^*(f^*,f^*)=\frac{1}{2}\int_{K} \left(\frac{df^*}{dx}\right)^2dx+\frac{1}{2}\sum_{k\geq 1}\frac{\left(f^*(a_k)-f^*(b_k) \right)^2}{|b_k-a_k|},\quad f^*\in \sF^*.  
\end{aligned}
\]
Its associated Hunt process $X^*$ is a time-changed Brownian motion on $K$ with the speed measure $\fm^*$.  We point out that $(\sE^*,\sF^*)$ and $X^*$ have been utilized in \cite{LY17} to study the regular Dirichlet subspaces of 1-dimensional Brownian motion.  


\subsection{Brownian motion on Cantor set}\label{SEC46}

Let $K\subset [0,1]$ be a generalized Cantor set as in \S\ref{SEC45} but we do not impose that $K$ has positive Lebesgue measure.  Set $K_n:=n+K=\{n+x: x\in K\}$ for $n\in \bZ$ and 
\[
\mathbf{K}:=\cup_{n\in \bZ}K_n.
\]
Write $[0,1]\setminus K$ as \eqref{eq:42} and let $a^n_k:=n+a_k,  b^n_k:=n+b_k$ for $n\in \bZ$.  Consider the following triple:
\begin{itemize}
\item[(a)] $I=[-\infty, \infty]$; 
\item[(b)] $\bs(x):=x$ for $x\in \mathbf{K}$ and $\bs(x):=a^n_k$ for $x\in (a^n_k,b^n_k)$ and $k\geq 1, n\in \bZ$; 
\item[(c)] $\fm(dx)=1_{\mathbf{K}}(x)dx+\sum_{k\geq 1,  n\in \bZ} \frac{|b^n_k-a^n_k|}{2}\left(\delta_{a^n_k}+\delta_{b^n_k}\right)$,  where $\delta_{a^n_k}, \delta_{b^n_k}$ are Dirac measures at $a^n_k,b^n_k$.  
\end{itemize}
Clearly (DK) and \text{(DM)} hold true.  
Mimicking \S\ref{SEC45},  one can obtain that the regularization of $(I,\bs,\fm)$ is 
\[
	I^*=\mathbf{K}, \quad \bs^*(x)=x,\; x\in \mathbf{K},\quad \fm^*=\fm.  
\]
Let $\sS^*$ be defined as \eqref{eq:43} with $\mathbf{K}$ and $\bR$ in place of $K$ and $(0,1)$ respectively.  Then the Dirichlet form $(\sE^*,\sF^*)$ on $L^2(I^*,\fm^*)$ is
\[
\begin{aligned}
&\sF^*=L^2(I^*,\fm^*)\cap \sS^*,\\
&\sE^*(f^*,f^*)=\frac{1}{2}\int_{\mathbf K} \left(\frac{df^*}{dx}\right)^2dx+\frac{1}{2}\sum_{k\geq 1,  n\in \bZ}\frac{\left(f^*(a^n_k)-f^*(b^n_k) \right)^2}{|b^n_k-a^n_k|},\quad f^*\in \sF^*.  
\end{aligned}
\]
Particularly,  when $K$ is a standard Cantor set,  $\mathbf{K}$ is of zero Lebesgue measure and the local part in the expression of $\sE^*(f^*,f^*)$ vanishes.  

The regularized Markov process $X^*$  is called a \emph{Brownian motion on Cantor set} in,  e.g.,  \cite{BEPP08}.  More precisely,  it is claimed in \cite{BEPP08} that $X^*$ is identified with the unique (in distribution) skip-free c\`adl\`ag process $\xi=(\xi_t)_{t\geq 0}$ (not necessarily a Markov process) on $\mathbf K$ such that both $\xi$ and $(\xi^2_t-t)_{t\geq 0}$ are martingales.  

\appendix

\section{Proof of Lemma~\ref{LM39}}\label{APPA}

\begin{proof}[Proof of Lemma~\ref{LM39}]
Recall that $\bs(0)=0$.  
Set 
\[
\begin{aligned}
	\widehat{F}(\widehat{x})&:=|(0,\widehat x)\cap \widehat{I}|,\quad 0\leq \widehat{x}\leq  \widehat{r},\\
	\widehat{F}(\widehat{x})&:=-|(\widehat x,0)\cap \widehat{I}|,\quad \widehat{l}\leq \widehat{x}< 0,
\end{aligned}\]
where $|\cdot|$ stands for the Lebesgue measure.  Clearly $\widehat{F}$ is continuous and increasing,  $\widehat{F}(\bs(x))=\bs_c(x)$ for $x\in (l,r)$.  Set $\widehat{F}^{-1}(\widehat{t}):=\widehat{F}^{-1}(\{\widehat{t}\})=\{\widehat{x}: \widehat{F}(\widehat{x})=\widehat{t}\}$ for $\widehat{t}\in [\widehat{F}(\widehat{l}),\widehat{F}(\widehat{r})]$.  Note that except for at most countably many $\widehat{t}$,  $\widehat{F}^{-1}(\widehat{t})$ is a singleton.  Denote by this exceptional set by $\{\widehat{t}_p: p\geq 1\}$.  Particularly,  $\widehat{F}^{-1}(\widehat{t}_p)$ is an interval and 
\[
\widehat{H}:=\left(\cup_{p\geq 1} \widehat{F}^{-1}(\widehat{t}_p)\right)\cap \widehat{I}
\]
is of zero Lebesgue measure.  Set further 
\[
\widehat{G}(\widehat s):=\inf\{\widehat{x}\in [\widehat{l},\widehat{r}]:  \widehat{F}(\widehat{x})>\widehat s\}
\]
with the convention $\inf \emptyset=\infty$.  For $\widehat{t}\in  [\widehat{F}(\widehat{l}),\widehat{F}(\widehat{r}))\setminus \{\widehat{t}_p: p\geq 1\}$,  $\widehat{G}(\widehat{t})$ is the unique element in $\widehat{F}^{-1}(\widehat{t})$,  and in a little abuse of notation, we write $\widehat{G}(\widehat{t})=\widehat{F}^{-1}(\widehat{t})$.  Let
\[
	H:=\left\{x\in (l,r): \bs(x)\neq \widehat{G}(\bs_c(x))\right\}.  
\]
Then
\begin{equation}\label{eq:A1}
\bs(H):=\{\bs(x): x\in H\}=\{\widehat{x}\in \widehat{I}: \widehat{x}\neq \widehat{G}(\widehat{F}(\widehat{x}))\}=\widehat{H},
\end{equation} 
where the second and third identities hold up to a set of at most countably many points.  More precisely,  both $\bs(H)\setminus \{\widehat{x}\in \widehat{I}: \widehat{x}\neq \widehat{G}(\widehat{F}(\widehat{x}))\}$ and $\{\widehat{x}\in \widehat{I}: \widehat{x}\neq \widehat{G}(\widehat{F}(\widehat{x}))\}\setminus \bs(H)$ are subsets of $\{\widehat{l},\widehat{r}\}$,  $\{\widehat{x}\in \widehat{I}: \widehat{x}\neq \widehat{G}(\widehat{F}(\widehat{x}))\}\subset \widehat{H}$ and $\widehat{H}\setminus \{\widehat{x}\in \widehat{I}: \widehat{x}\neq \widehat{G}(\widehat{F}(\widehat{x}))\}$ contains at most countably many points.  We further assert
\begin{equation}\label{eq:A2}
\mu_c(H)=0.  
\end{equation}
To accomplish this,  note that for any $x\in H$,  $\bs_c(x)=\widehat{F}(\bs(x))\in \{\widehat{t}_p: p\geq 1\}$ and hence $|\bs_c(H)|=0$.  It follows that $\mu_c(H)\leq \mu_c\circ \bs_c^{-1}(\bs_c(H))=|\bs_c(H)|=0$,  where the image measure $\mu_c\circ \bs_c^{-1}$ of $\mu_c$ under $\bs_c$ is actually the Lebesgue measure;  see,  e.g.,  \cite[\S3.5, Exercise 36]{F99}.  


Let $\widehat{f}=\widehat{h}|_{\widehat{I}}$ with $\widehat{h}\in \dot{H}^1_e\left((\widehat{l},\widehat{r})\right)$.  We are to prove that $f(\cdot):=\widehat{f}(\bs(\cdot))=\widehat{h}(\bs(\cdot))\in \sS$,  so that $\widehat{f}\in \widehat{\sS}$.  To do this, set
\[
\begin{aligned}
	&g^c(x):=\widehat{h}'(\widehat G(\bs_c(x))), \quad x\in (l,r);  \\
	&g^\pm(x):=\left\lbrace
	\begin{aligned}
	&\frac{\widehat{h}(\bs(x))-\widehat{h}(\bs(x\pm))}{\bs(x)-\bs(x\pm)},\quad x\in D^\pm, \\
	&0,\qquad\qquad\qquad\qquad\quad x\notin D^\pm. 
	\end{aligned}
	\right.
\end{aligned}\]
Since $\widehat{h}'\in L^2((\widehat{l},\widehat{r}))$ and $|\widehat{h}(\bs(x))-\widehat{h}(\bs(x\pm))|^2\leq \left|\int_{\bs(x)}^{\bs(x\pm)}\widehat{h}'(\widehat{t})^2d\widehat{t}\right| \cdot |\bs(x)-\bs(x\pm)|$,  it follows that $g^\pm\in L^2(I, \mu^\pm_d)$.  We assert that $g^c\in L^2(I, \mu_c)$.  In fact,  $\bs_c(l)=\widehat{F}(\widehat{l})$,  $\bs_c(r)=\widehat{F}(\widehat{r})$,  and using \cite[Chapter 0, Proposition~4.10]{RY99} with $A(\widehat t)=\widehat t$ and $u=\bs_c$ or $\widehat{F}$,  we get that $\int_l^r g^c(x)^2\mu_c(dx)$ is equal to 
\begin{equation}\label{eq:36}
	\int_{\bs_c(l)}^{\bs_c(r)}\widehat{h}'(\widehat{G}(\widehat t))^2d\widehat t=\int_{\widehat{l}}^{\widehat{r}} \widehat{h}' (\widehat{G}(\widehat{F}(\widehat{x})))^2d\widehat{F}(\widehat{x}) =\int_{\widehat{I}}  \widehat{h}' (\widehat{G}(\widehat{F}(\widehat{x})))^2d\widehat{x}.
\end{equation}
Since $\widehat{G}(\widehat{F}(\widehat{x}))=\widehat{x}$ for $\widehat{x}\in \widehat{I}\setminus \widehat{H}$ (see \eqref{eq:A1}) and $|\widehat{H}|=0$,  it follows that
\begin{equation}\label{eq:37}
	\int_l^r g^c(x)^2\mu_c(dx)=\int_{\widehat{I}}  \widehat{h}' (\widehat{x})^2d\widehat{x}<\infty.
\end{equation}
Hence $g^c\in L^2(I,\mu_c)$ is concluded.  
Now define for $x\in (l,r)$, 
\[
	f^+(x):=\int_{(0,x)}g^+(y)\mu^+_d(dy),\quad f^-(x):=\int_{(0,x]}g^-(y)\mu^-_d(dy).
\]
It suffices to show $f^c:=f-f^+-f^-\in \sS_c$.  Indeed,  take $0<x<r$ and we have
\begin{equation*}
\begin{aligned}
	f^c(x)-f^c(0)&=\widehat{f}(\bs(x))-\widehat{f}(0)-f^+(x)-f^-(x)=\int_{(0,\bs(x))\cap \widehat{I}} \widehat{h}'(\widehat{t})d\widehat{t}. 
\end{aligned}
\end{equation*}
Mimicking \eqref{eq:36} and \eqref{eq:37},  we get that
\[
f^c(x)-f^c(0)=\int_0^x g^c(y)\mu_c(dy).  
\]
On account of $g^c\in L^2(I,\mu_c)$,  we obtain that $f^c\in \sS_c$.  Therefore $f\in \sS$ is concluded. 

To the contrary,   take $f=f^c+f^++f^-\in \sS$ and let $\widehat{f}:=\widehat{\iota}f\in \widehat{\sS}$.  Define a function $\widehat{h}$ as follows:
\begin{equation}\label{eq:A5}
\begin{aligned}
	&\widehat{h}(\bs(x)):=\widehat{f}(\bs(x))=f(x),\quad x\in (l,r),  \\
	&\widehat{h}(\bs(x\pm)):=\widehat{f}(\bs(x\pm))=f(x\pm),\quad x\in D^\pm. 
\end{aligned}
\end{equation}
Note that $|\widehat{a}_k|+|\widehat{b}_k|<\infty$ (see \eqref{eq:15}).  
For $\widehat{x}\in (\widehat{a}_k,\widehat{b}_k),  k\geq 1$,  define further
\begin{equation}\label{eq:A6}
\widehat{h}(\widehat{x}):= \frac{\widehat{h}(\widehat b_k)-\widehat{h}(\widehat a_k)}{\widehat{b}_k-\widehat{a}_k}\cdot (\widehat{x}-\widehat{a}_k)+\widehat{h}(\widehat{a}_k). 
\end{equation}
We need to prove that $\widehat{h}\in \dot H^1_e((\widehat{l},\widehat{r}))$.  To do this,  define a function $\widehat{\varphi}$ as follows:
\begin{equation*}
	\widehat{\varphi}(\widehat{x}):=\left\lbrace 
	\begin{aligned}
	&\frac{df^c}{d\mu^c}(x),\quad \widehat{x}=\bs(x)\text{ with }x\in (l,r)\setminus H \text{ and the derivative exists}, \\
	&\frac{\widehat{h}(\widehat b_k)-\widehat{h}(\widehat a_k)}{\widehat{b}_k-\widehat{a}_k},\quad \widehat{x}\in (\widehat{a}_k,\widehat{b}_k), k\geq 1.  
	\end{aligned}
	\right. 
\end{equation*}
Note that $|\bs(H)|=|\widehat{H}|=0$ due to \eqref{eq:A1}.  
Let 
\[
	N:=\{x\in (l,r)\setminus H: df^c/d\mu^c(x)\text{ does not exist}\}, \quad \widehat{N}:=\bs(N).
	\]  
We point out that $|\widehat{N}|=0$,  so that $\widehat{\varphi}$ is defined a.e.  on $(\widehat{l},\widehat{r})$.  In fact,  $\widehat{N}\subset [\widehat{l},\widehat{r}]\setminus \widehat{H}$ and for $x\in N$,  $\bs(x)=\widehat{G}(\bs_c(x))=\widehat{F}^{-1}(\bs_c(x))$ because $\bs(x)\notin \widehat{H}$.  It follows that $\widehat{N}=\widehat{F}^{-1}(\bs_c(N))$ which is a subset of $\widehat{I}_\bs$,  and hence
\begin{equation}\label{eq:A8}
|\widehat N|=\int_{\widehat{F}^{-1}(\bs_c(N))}d\widehat{F}=|\bs_c(N)|=\mu_c\circ \bs_c^{-1}(\bs_c(N)),
\end{equation}
where the second and the third equalities hold because $(d\widehat{F})\circ \widehat{F}^{-1}$ and $\mu_c\circ \bs^{-1}_c$,  i.e.  the image measures of $d\widehat{F}$ and $\mu_c$ under the maps $\widehat{F}$ and $\bs_c$ respectively,  are actually the Lebesgue measure.  Except for at most countably many points $\widehat t=\bs_c(x)\in \bs_c(N)$,  $\bs_c^{-1}(\{\widehat t\})$ is a singleton equalling $\{x\}$.   If is not a singleton,  $\bs_c^{-1}(\{\widehat t\})$ is an interval of zero $\mu_c$-measure.  Consequently,  \eqref{eq:A8} yields that $|\widehat{N}|=\mu_c(N)=0$.  
 Next,  we prove that $\widehat{\varphi}\in L^2((\widehat{l}, \widehat{r}))$.  
A straightforward computation yields that
\begin{equation}\label{eq:38}
\begin{aligned}
	\int_{\cup_{k\geq 1}(\widehat{a}_k,\widehat{b}_k)} \widehat{\varphi}(\widehat{x})^2d\widehat{x}&=\sum_{k\geq 1}\frac{\left(\widehat{h}(\widehat b_k)-\widehat{h}(\widehat a_k)\right)^2}{|\widehat{b}_k-\widehat{a}_k|} \\ 
	&=\sum_{x\in D^-}\frac{\left(f(x)-f(x-)\right)^2}{\bs(x)-\bs(x-)}+\sum_{x\in D^+}\frac{\left(f(x+)-f(x)\right)^2}{\bs(x+)-\bs(x)}  \\
	&=\int_I \left(\frac{df^-}{d\mu^-_d}\right)^2d\mu^-_d+\int_I \left(\frac{df^+}{d\mu^+_d}\right)^2d\mu^+_d<\infty.  
\end{aligned}
\end{equation}
Mimicking \eqref{eq:36} and \eqref{eq:37} and using \eqref{eq:A2}, we get that
\begin{equation}\label{eq:39-2}
\begin{aligned}
&\int_{(\widehat{l},\widehat{r})\setminus  \cup_{k\geq 1}(\widehat{a}_k,\widehat{b}_k)}   \widehat{\varphi}(\widehat{x})^2d\widehat{x}  \\  &\qquad = \int_l^r  \widehat{\varphi}(\widehat{G}(\bs_c(x)))^2d\bs_c(x) =\int_{I\setminus H} \widehat{\varphi}(\widehat{G}(\bs_c(x)))^2d\bs_c(x)\\
	&\qquad=\int_{I\setminus H} \widehat{\varphi}(\bs(x)))^2d\bs_c(x)  =\int_{I} \left(\frac{df^c}{d\mu^c}\right)^2d\mu^c<\infty. 
\end{aligned}\end{equation}
Hence $\widehat{\varphi}\in L^2((\widehat{l},\widehat{r}))$ is concluded.   Finally it suffices to show
\begin{equation}\label{eq:39}
	\widehat{h}(\widehat{x})-\widehat{h}(0)=\int_0^{\widehat{x}} \widehat{\varphi}(\widehat{t})d\widehat{t},\quad \widehat{x}\in (\widehat{l},\widehat{r}). 
\end{equation}
Consider first $\widehat{x}=\bs(x)$ for $x\in (0,r)\setminus D$.  The left hand side of \eqref{eq:39} is equal to
\begin{equation}\label{eq:314}
f(x)-f(0)=f^c(x)-f^c(0)+f^+(x)+f^-(x).  
\end{equation}
Note that 
\[
	f^\pm(x)=\pm\sum_{y\in (0,x)\cap D^\pm} (f(y\pm)-f(y))
\]
and 
\[
	\left(\cup_{y\in (0,x)\cap D^+} \left (\bs(y),\bs(y+)\right)\right) \cup \left(\cup_{y\in (0,x)\cap D^-} \left(\bs(y-),\bs(y)\right)\right) =\cup_{k: 0\leq  \widehat{a}_k<\widehat{b}_k\leq \bs(x)}(\widehat{a}_k,\widehat{b}_k).  
\]
It follows from \eqref{eq:A6} that
\begin{equation}\label{eq:315}
	f^+(x)+f^-(x)=\sum_{k: 0\leq \widehat{a}_k<\widehat{b}_k\leq \bs(x)}\int_{\widehat{a}_k}^{\widehat{b}_k} \widehat{\varphi}(\widehat{t})d\widehat{t}.  
\end{equation}
Mimicking \eqref{eq:36} and \eqref{eq:37},  we have
\[
	\int_{(0,\bs(x))\setminus \cup_{k\geq 1}(\widehat{a}_k,\widehat{b}_k)} \widehat{\varphi}(\widehat{t})d\widehat{t}=\int_{(0,x)}\widehat{\varphi}(\widehat{G}(\bs_c(y)))\mu_c(dy). 
\]
Since $\widehat{G}(\bs_c(y))=\bs(y)$ for $y\notin H$ and \eqref{eq:A2},    the last term is equal to
\begin{equation}\label{eq:316}
	\int_0^x \frac{df^c}{d\mu^c}d\mu^c=f^c(x)-f^c(0).  
\end{equation}
In view of \eqref{eq:314},  \eqref{eq:315} and \eqref{eq:316},  we obtain \eqref{eq:39} for $x\in (0,r)\setminus D$.  Analogously \eqref{eq:39} holds for $x\in (l,0)\setminus D$.  Due to the $\bs$-continuity of $f$ as stated in Lemma~\ref{LM31},  $\widehat{h}$ is continuous on $(\widehat{l},\widehat{r})$.  By this continuity and $\widehat{\varphi}\in L^2((\widehat{l},\widehat{r}))$,  we can obtain \eqref{eq:39} for $\widehat{x}=\bs(x\pm)$ with $x \in D^\pm$.  Then the identity \eqref{eq:39} for $\widehat{x}\in \cup_{k\geq 1}(\widehat{a}_k,\widehat{b}_k)$ is obvious by means of \eqref{eq:A6}.  Eventually we conclude that $\widehat{h}\in \dot H^1_e((\widehat{l},\widehat{r}))$.  That completes the proof. 
\end{proof}

\section*{Acknowledgment}

The idea of the proof of Lemma~\ref{LM54}~(6) was inspired by the discussion with a PhD student Dongjian Qian.  The author also would like to thank Professor Jiangang Ying for reminding me of exploring the connection between $(\widehat{\sE},\widehat{\sF})$ and the regular representations of $(\sE,\sF)$.  


\bibliographystyle{siam} 
\bibliography{DisScale4} 

\end{document}